\documentclass[11pt,letterpaper]{amsart}
\usepackage{amssymb,amstext,amscd,amsthm,amsfonts,mathdots,mathrsfs}
\usepackage{pdflscape}
\usepackage{hyperref}
\usepackage{pdflscape}
\usepackage{wasysym}
\usepackage{setspace}
\usepackage[all]{xy}
\usepackage{enumerate}
\usepackage{indentfirst}
\usepackage{ytableau}
\usepackage{color}
\usepackage{colortbl}
\usepackage{amsmath}
\usepackage{tikz}
\usetikzlibrary{shapes.geometric, arrows}
\usetikzlibrary{calc}
\usepackage{tikz-cd}

\numberwithin{equation}{section}

\usepackage[thicklines]{cancel}

\usepackage{cleveref}
\crefname{theorem}{Theorem}{Theorems}
\crefname{definition}{Theorem}{Definitions}
\crefname{proposition}{Theorem}{Propositions}
\crefname{corollary}{Corollary}{Corollaries}
\crefname{alphatheorem}{Theorem}{Theorems}

\newtheorem{theorem}{Theorem}[section]
\newtheorem{lemma}[theorem]{Lemma}
\newtheorem{proposition}[theorem]{Proposition}
\newtheorem{definition}[theorem]{Definition}

\newtheorem{corollary}[theorem]{Corollary}
\newtheorem{example}[theorem]{Example}

\theoremstyle{remark}
\newtheorem{rem}[theorem]{Remark}

\newtheorem{alphatheorem}{\bf Theorem}

\newcommand{\clr}{rgb:black,3;blue,2;red,0}

\tikzset{anchorbase/.style={baseline={([yshift=-0.5ex]current bounding box.center)}}}
\tikzset{ 
    centerzero/.style={>=To,baseline={([yshift=-0.5ex](#1))}},
    centerzero/.default={0,0}
}
\tikzset{wipe/.style={white,line width=4pt}}
\newcommand{\cred}{rgb:black,.5;blue,0;red,1.5}

\DeclareFontFamily{OT1}{pzc}{}
\DeclareFontShape{OT1}{pzc}{m}{it}{ <-> s*[1.2] pzcmi7t }{}
\DeclareMathAlphabet{\mathpzc}{OT1}{pzc}{m}{it}

\newcommand{\HH}{\mathpzc{H}}
\newcommand{\qW}{\mathpzc{WeB}}  
 
\newcommand{\qAW}{\mathpzc{WeB}^\bullet} 

\newcommand{\qAH}{\mathpzc{H}^\bullet} 
\newcommand{\qH}{\mathpzc{H}} 
\newcommand{\qSch}{\mathpzc{SchuR}}
\newcommand{\qASch}{\mathpzc{SchuR}^{{\hspace{-.03in}}\bullet}}

\newcommand{\qSchu}{\mathpzc{SchuR}_{\bfu}}
 
\newcommand{\bT}{\mathbf{T}}
\newcommand{\bS}{\mathbf{S}}

\newcommand{\Mat}{\text{Mat}}
\newcommand{\PMat}{\text{ParMat}}
\newcommand{\RPMat}{\text{RParMat}}
\newcommand{\Par}{\text{Par}}

\newcommand{\RPar}{\text{RPar}}
\newcommand{\SST}{\text{SST}}
\newcommand{\bfB}{\mathbf{B}}
\newcommand{\bfc}{\mathbf{c}}
\newcommand{\bfd}{\mathbf{d}}
\newcommand{\bfH}{\mathbf{H}}
\newcommand{\bfi}{\mathbf{i}}
\newcommand{\bfj}{\mathbf{j}}
\newcommand{\bfk}{\mathbf{k}}

\newcommand{\bfm}{\mathbf{m}}
\newcommand{\bfn}{\mathbf{n}}
\newcommand{\bfs}{\mathbf{s}}
\newcommand{\bfS}{\mathbf{S}}
\newcommand{\bfu}{\mathbf{u}}
\newcommand{\bfW}{\mathbf{W}}
\newcommand\cA{\mathcal{A}}
\newcommand\cB{\mathcal{B}}
\newcommand{\cF}{\mathcal{F}}
\newcommand{\cH}{\mathcal{H}}
\newcommand{\cM}{\mathcal{M}}
\newcommand{\cN}{\mathcal{N}}
\newcommand{\cO}{\mathcal{O}}
\newcommand{\cS}{\mathcal{S}}
\newcommand{\cZ}{\mathcal{Z}}

\newcommand{\Ind}{\operatorname{Ind}\nolimits}

\newcommand{\iInd}{\mathbf{i}\operatorname{-Ind}\nolimits}
\newcommand{\Res}{\operatorname{Res}\nolimits}
\newcommand{\iRes}{\mathbf{i}\operatorname{-Res}\nolimits}
\newcommand{\res}{\operatorname{res}\nolimits}

\newcommand{\modf}{\operatorname{-mod}\nolimits}
\newcommand{\prmod}{\operatorname{-pmod}\nolimits} 

\newcommand{\eps}{\varepsilon}
\newcommand{\gl}{\mathfrak{gl}}
\newcommand{\sll}{\mathfrak{sl}}
\newcommand{\qq}{\texttt{q}}
\newcommand\C{\mathbb{C}}
\newcommand\Z{\mathbb{Z}}

\newcommand\N{\mathbb{N}}
\newcommand{\st}{\text{st}}
\newcommand{\AHe}{\dot{\bf H}}
\newcommand{\ASch}{\dot{\bf S}}
\newcommand{\ASym}{\widehat{\mathfrak S}}

\def\bfla{\lambda}
\def\bfmu{\mu}

\newcommand\wS{{}^{\bold{w}}\bold{S}}
\newcommand\hwS{{}^{\bold{w}}\dot{\bold{S}}}
\newcommand\K{\bf{K}}
\newcommand\KS{{}^{\bold{k}}\dot{\bold{S}}}
\newcommand\KH{{}^{\bold{k}}\dot{\bold{H}}}
\newcommand\KBi{{}^{\bold{k}}\bold{T}}

\newcommand{\Uglzhalf}{\U_\Z^-(\widehat\gl_p)}

\newcommand\U{\bold{U}}
\newcommand\UA{\bold{U}\!_\cA}
\newcommand\Hall{\bold{Hall}_{p}}
\newcommand\HallZ{\bold{Hall}_{p,\Z}}
\newcommand\Heis{\mathfrak{Heis}}

\newcommand\bbF{\mathbb{F}}
\newcommand\hatT{\dot{\mathbf{T}}}
\newcommand\I{\mathbb{I}}
\newcommand\pen{\Large{\pentagon}}
\newcommand\penp{\Large{\pentagon}\!_p}
\newcommand\peni{\Large{\pentagon}\!_\infty}

\newcommand{\qbinom}[2]{\begin{bmatrix} #1\\#2 \end{bmatrix} }
\newcommand{\qbinomsm}{\textstyle\genfrac{[}{]}{0pt}{}}

\newcommand\kk{\Bbbk}
\newcommand\la{\lambda}
\newcommand{\Hom}{{\rm Hom}}
\newcommand{\End}{{\rm End}}

\newcommand{\arxiv}[1]{\href{http://arxiv.org/abs/#1}{\tt arXiv:\nolinkurl{#1}}}

\def\t{\mathfrak t}

\newcommand{\stra}{\begin{tikzpicture}[baseline = 10pt, scale=0.4, color=\clr]
            \draw[-,line width=1.2pt] (0,0.7)to[out=up,in=down](0,1.8);
            \draw (0,0.49) node{$\scriptstyle a$};
\end{tikzpicture} 
}

\newcommand{\stru}{\begin{tikzpicture}[baseline = 10pt, scale=0.4, color=\cred]
\draw[-,line width=1pt] (0,0.4) to (0,1.7);
\draw(0,0.1) node {$\scriptstyle u$};
\end{tikzpicture}
}

\newcommand{\zdot}{ node[circle,fill=white,draw,   
thin, inner sep=0pt, minimum width=3.2pt]{}
} 

\newcommand{\bdot}{ node[circle, draw, fill=\clr, thick, inner sep=0pt, minimum width=2.8pt]{}
}

\newcommand{\xdota}{
\begin{tikzpicture}[baseline = 3pt, scale=0.4, color=\clr]
\draw[-,line width=1.2pt] (0,0) to[out=up, in=down] (0,1.4);
\draw(0,0.6) \bdot;
\node at (0,-.22) {$\scriptstyle a$};
\end{tikzpicture}
}

\newcommand{\zdota}{
\begin{tikzpicture}[baseline = 3pt, scale=0.4, color=\clr]
\draw[-,line width=1.2pt] (0,0) to[out=up, in=down] (0,1.4);
\draw(0,0.6) \zdot;
\node at (0,-.22) {$\scriptstyle a$};
\end{tikzpicture}
}

\newcommand{\merge}
{\begin{tikzpicture}[baseline = -.5mm,scale=.8, color=\clr]
	\draw[-,line width=1pt] (0.28,-.3) to (0.08,0.04);
	\draw[-,line width=1pt] (-0.12,-.3) to (0.08,0.04);
	\draw[-,line width=1.5pt] (0.08,.4) to (0.08,0);
        \node at (-0.18,-.4) {$\scriptstyle a$};
        \node at (0.35,-.4) {$\scriptstyle b$};
        \node at (0.05,.55){$\scriptstyle a+b$};
        \end{tikzpicture} }
        
\newcommand{\splits}
{\begin{tikzpicture}[baseline = -.5mm,scale=.8, color=\clr]
	\draw[-,line width=1.5pt] (0.08,-.3) to (0.08,0.04);
	\draw[-,line width=1pt] (0.28,.4) to (0.08,0);
	\draw[-,line width=1pt] (-0.12,.4) to (0.08,0);
        \node at (-0.2,.5) {$\scriptstyle a$};
        \node at (0.36,.5) {$\scriptstyle b$};
        \node at (0.1,-.41){$\scriptstyle a+b$};
\end{tikzpicture}}

\newcommand{\crossingpos}{
\begin{tikzpicture}[baseline=-1mm, scale=.8, color=\clr]
	\draw[-,line width=1pt] (0.3,-.3) to (-.3,.4);
	\draw[-,line width=4pt,white] (-0.3,-.3) to (.3,.4);
	\draw[-,line width=1pt] (-0.3,-.3) to (.3,.4);
        \node at (-0.3,-.45) {$\scriptstyle a$};
        \node at (0.26,-.45) {$\scriptstyle b$};
\end{tikzpicture}
}

\newcommand{\crossingneg}{
\begin{tikzpicture}[baseline=-1mm, scale=.8, color=\clr]
 \draw[-,line width=1pt] (-0.3,-.3) to (.3,.4);
	\draw[-,line width=4pt,white] (0.3,-.3) to (-.3,.4);
 \draw[-,line width=1pt] (0.3,-.3) to (-.3,.4);
        \node at (-0.33,-.45) {$\scriptstyle b$};
        \node at (0.3,-.45) {$\scriptstyle a$};
\end{tikzpicture}
}

\newcommand{\rightcrossing}{\begin{tikzpicture}[baseline = 1mm, scale=.8, color=\clr]
 \draw[-,line width=1.2pt] (-0.3,0) to (.3,.7);
\draw[-,line width=1pt,color=\cred] (0.3,0) to (-.3,.7);
\draw(-.3,-0.1) node{$\scriptstyle a$};
\draw (.3, -0.1) node{$\scriptstyle \red{u}$};
\end{tikzpicture}}

\newcommand{\leftcrossing}{\begin{tikzpicture}[baseline = 1mm, scale=.8, color=\clr]
 \draw[-,line width=1pt,color=\cred] (-0.3,0) to (.3,.7);
\draw[-,line width=1.2pt] (0.3,0) to (-.3,.7);
\draw(-.3,-.1) node{$\scriptstyle \red{u}$};
\draw (.3, -.1) node{$\scriptstyle a$};
\end{tikzpicture}}

\newcommand{\wkdotaa}{\begin{tikzpicture}[baseline = 3pt, scale=0.4, color=\clr]
\draw[-,line width=1.2pt] (0,0) to[out=up, in=down] (0,1.4);
\draw(0,0.6) \bdot; 
\node at (0,-.22) {$\scriptstyle a$};
\end{tikzpicture} }

\newcommand{\red}[1]{{\color{red}#1}}

\newcommand{\unit}[1]{1_{#1}}

\begin{document}
\setlength{\baselineskip}{17pt}
\title{Decomposition matrices of cyclotomic $q$-web categories}

\author{Linliang Song}
\address{School of Mathematical Science, Tongji University, Shanghai, 200092, China} \email{llsong@tongji.edu.cn}

\author{Weiqiang Wang}
 \address{Department of Mathematics, University of Virginia, Charlottesville, VA 22904, USA} \email{ww9c@virginia.edu}



\begin{abstract}
We develop the cellular structures for the endomorphism algebras of cyclotomic $q$-webs, which form a new family of quantum algebras sitting in between cyclotomic Hecke algebras and cyclotomic $q$-Schur algebras. We show that the Grothendieck group of a module category of the cyclotomic $q$-webs for $q$ generic or a root of unity is isomorphic to an integrable highest weight module over quantum affine  $\mathfrak{gl}_p$; moreover, the isomorphism maps the classes of projective indecomposable modules to the canonical basis. This substantially generalizes the classic works of Lascoux-Leclerc-Thibon, Ariki, and Varagnolo-Vasserot.
\end{abstract}

\maketitle

\setcounter{tocdepth}{1}
\tableofcontents

%
\section{Introduction}
\label{sec:intr}

\subsection{$q$-Web categories}

Two monoidal categories, the affine web category and the affine Schur category, were introduced by the authors \cite{SW25web, SW24Schur}, and their cyclotomic quotients were also constructed; also see \cite{DKM25} for a generalization of the affine and cyclotomic web categories. Two strict $\kk$-linear monoidal categories, the $q$-deformed affine web and Schur categories, $\qAW$ and $\qASch$, were subsequently formulated in \cite{SSW25}, and we will assume $\kk=\C[q,q^{-1}]$ in this paper. The thin strand full subcategory of $\qAW$ is identified with the affine Hecke algebra category $\qAH$.

These three monoidal categories and their corresponding cyclotomic quotient categories can be assembled into the following commutative diagram:
\begin{equation} \label{CD:categories}
\begin{tikzcd}
\qAH \ar[r,hook]\ar[d,two heads] & \qAW\ar[r,hook]\ar[d,two heads] & \qASch\ar[d,two heads]\\
\qH_\bfu  \ar[r,hook]&\qW_\bfu \ar[r,hook]&\qSch_\bfu
\end{tikzcd} 
\end{equation}
where $\bfu=(u_1,\ldots, u_\ell) \in (\kk^*)^\ell$, for $\ell \in\N$. These categories have provided (often new) diagrammatic presentations of old and new quantum algebras. In this paper, we are mostly focused on the algebras arising from the middle column and their relations to those from the left column.

The path algebra (i.e., the direct sum of all hom-spaces) of $\qAW$ is 
\[
\hwS:=\bigoplus_{n\ge 0} \hwS(n), 
\]
where it follows by \cite{SSW25} that $\hwS(n)$ is Morita equivalent to the affine Schur algebras $\ASch(n,n)$ from \cite{Gre99}.
The path algebra of $\qW_\bfu$ is 
\[
\wS_\bfu:=\bigoplus_{n\ge 0} \wS_\bfu(n); 
\]
we propose to call these new algebras $\wS_\bfu(n)$ {\em cyclotomic $q$-W-Schur algebras} (of level $\ell$ if $\bfu$ is an $\ell$-tuple); recall that the degenerate variants were identified in \cite{SW25web} with Schur algebras for finite W-algebras. 
The level one cyclotomic $q$-web category appeared in \cite{Bru25} (called $q$-$\mathbf{Schur}$ category) and these 2 level one cyclotomic $q$-(W-)Schur algebras are identical and Morita equivalent to Dipper-James $q$-Schur algebras $\bfS(n,n)$. 

We also recall the well-known fact that the path algebra of $\qAH$ is identified with the affine Hecke algebras: $\AHe :=\oplus_{n\ge 0} \AHe(n)$, while the path algebra of the cyclotomic category $\qH_\bfu$ is identified with the (Ariki-Koike) cyclotomic Hecke algebras: $\bfH_\bfu =\oplus_{n\ge 0} \bfH_\bfu(n)$.

\subsection{Works of \cite{LLT96, Ar96, LT96, VV99} }

Let $\bfs=(s_1,\ldots, s_\ell)\in \Z^\ell$. Let $\eps \in \C$ be generic or a root of $1$, and let $p$ be the order of $\eps^2$. Denote by $\bfH_{\eps^{2\bfs}}(n)$ the cyclotomic Hecke algebras by specializing $q=\eps$ and $\bfu=\eps^{2\bfs}$. We denote by $[\mathcal C]$ the Grothendieck group of an abelian category $\mathcal C$. Denote by $V(\omega_\bfs)$ the integrable highest weight module $V(\omega_\bfs)$ over the affine quantum group $\U_v(\widehat\sll_p)$, and denote by $\varpi_\bfs: \U^-_v(\widehat\sll_p) \rightarrow V(\omega_\bfs)$ the $\U^-_v(\widehat\sll_p)$-module homomorphism sending $1$ to the highest weight vector $\eta_{\omega_\bfs}$. The surjection $\qH \rightarrow \qH_{\eps^{2\bfs}}$ induces a natural functor $\pi_\bfs: \qH_{\eps^{2\bfs}}\modf \rightarrow \qH\modf$ which preserves the simples, which in turn induces a surjective map on the duals, $\pi_\bfs^*: [\qH\modf]^* \rightarrow [\qH_{\eps^{2\bfs}}\modf]^*$. Proving and extending Lascoux-Leclerc-Thibon conjecture \cite{LLT96}, Ariki \cite{Ar96} established the following commutative diagram of $\U^-(\widehat\sll_p)$-module homomorphisms: 
\begin{equation}  
\label{eq:Ariki}
\begin{tikzcd}
\U^-_\Z(\widehat\sll_p) \ar[r,"\cong"]\ar[d,two heads,"\varpi_\bfs"]&{[}\qH\modf_\eps{]}^* \ar[d,two heads,"\pi_\bfs^*"]\\
V(\omega_\bfs)_\Z  \ar[r,"\cong"]&{[} \qH_{\eps^{2\bfs}}\modf{]}^* 
\end{tikzcd} 
\end{equation}
where the horizontal isomorphisms send the canonical basis \cite{Lus90, Lus91} (and \cite{Kas91}) specialized at $v=1$ to the dual basis of simple modules, and the Chevalley generators $F_i, E_i$ of $\U_v(\widehat\sll_p)$ are realized as $i$-restriction and $i$-induction functors on $\qH_{\eps^{2\bfs}}\modf$. Here and below we use the subscript $\Z$ to denote the specialization at $v=1$. We refer to the survey of Kleshchev \cite{Kle10} for numerous subsequent exciting developments on representations of symmetric groups, Hecke algebras, and categorification.

Varagnolo-Vasserot \cite{VV99} showed that the decomposition matrices of the $q$-Schur algebras for $q=\eps$ being a root of 1 are given by the transition matrices between the canonical basis and Schur basis on the level one $\U_v(\widehat\gl_p)$-Fock space $\cF$, establishing the Leclerc-Thibon conjecture \cite{LT96}; this generalizes the level one result in \cite{LLT96, Ar96}. Their results can be paraphrased as a $\Z$-linear isomorphism:
\begin{align}  \label{eq:VV}
     \cF_\Z \stackrel{\cong}{\longrightarrow} \bigoplus_{n\ge 0} [\bfS(n,n)\modf]^*,
\end{align}
which maps the Schur basis and the canonical basis to the dual bases to Weyl modules and simple modules, respectively. 

\subsection{The main results}

This paper is a first step on categorification based on our diagrammatic web and Schur categories as outlined in \cite{SW25web, SW24Schur}. In the setting of the middle column of the diagram \eqref{CD:categories}, we show that the cyclotomic $\eps$-webs with $\eps^2$ being a primitive $p$th root of $1$ categorify the integrable highest weight modules over quantum affine $\gl_p$, which admit canonical bases exhibiting positivity (see \cite{Sch00}); this generalizes the main results of Ariki and Varagnolo-Vasserot. Similar results hold in the degenerate web setting. To that end, it is crucial and fruitful for us to develop new connections of affine $q$-webs to geometric convolution algebras and to Hall algebra of the cyclic quiver. 

The results of this paper will be a building block of the categorification of additional diagrammatic web/Schur categories of other types. 

\subsubsection{}

One advantage of diagrammatic categories is to allow us to construct restriction and induction functors on $\qW_{\bfu} \modf$, which naturally interact with combinatorics of semi-standard tableaux. It is shown in \cite{SSW25} (see \cite{SW24Schur}) that the path algebra of $\qSch_\bfu$ is isomorphic to (Dipper-James-Mathas) cyclotomic $q$-Schur algebras \cite{DJM98}: $\bfS_\bfu=\oplus_{n\ge 0} \bfS_\bfu(n)$. The three families of algebras arising from the diagram \eqref{CD:categories} are related: $\bfH_\bfu(n)\subset \wS_\bfu(n) \subset \bfS_\bfu(n)$. We obtain the cellular structure on $\wS_\bfu(n)$ by applying a Schur functor to $\bfS_\bfu(n)\modf$; note that $\bfS_\bfu(n)$ is quasi-hereditary \cite{SSW25} and provides a diagrammatic realization of the cyclotomic $q$-Schur algebras \cite{DJM98}. Denote by $\Par^\ell(n)$ the set of $\ell$-multipartitions of $n$.

\begin{alphatheorem} [Theorem \ref{thm:cellularforwebu}, Proposition \ref{prop:IndRes}] 
\label{thm:A}
\begin{enumerate}
    \item The algebra $\wS_\bfu(n)$ admits a cellular basis and cell modules $\Delta^\la$, for $\la \in \Par^\ell(n)$.
    \item 
    There exist cell module filtrations on $\Res^n_{n-a}(\Delta^\la)$ and $\Ind_n^{n+a} (\Delta^\la)$, for $\la \in \Par^\ell(n)$. 
\end{enumerate}
\end{alphatheorem}

\subsubsection{}
Fix $\I =\Z$ or $\Z/p\Z$ when $q=\eps$ is generic or a root of 1 with order of $\eps^2$ equal to $p$, and choose the parameter $\bfu =\eps^{2\bfs}$. The quotient functor $\qAW \to\qW_{\eps^{2\bfs}}$ induces a functor 
\[
\pi_\bfs: \qW_{\eps^{2\bfs}}\modf \longrightarrow \qAW\modf_\eps
\]
and then a surjection $\pi_\bfs^*: [\qAW\modf_\eps]^* \rightarrow [\qW_{\eps^{2\bfs}}\modf]^*$. Denote by $\bfm\in\cM\cS$ the set of all multi-segments \eqref{cMcS} over $\I$. By definition \eqref{eq:std}, the standard modules $\Delta_\bfm$, for $\bfm\in\cM\cS$, are $\hwS(n)$-modules of the form $\pi_\bfs(\Delta^\la)$, for some distinguished $\la \in \Par^\ell(n)$ associated to $\bfm$. 

The geometric affine $q$-Schur algebra, $\KS(N,n) =K^G(T^*\cF \times_{\cN} T^*\cF)$, is the convolution algebra on the Steinberg variety of $N$-step flags in $\C^n$; cf. \eqref{KSchur}. Denote by 
\[
{\K} :=\bigoplus_{n=0}^\infty [\KS(n,n)\modf]
\]
and by $\K^*$ its restricted dual; cf. \cite{GV93, V98} or \eqref{eq:KK}. Denote by $M_\bfm$, for $\bfm\in\cM\cS$,  the standard modules in $\K$ defined via Borel-Moore homology; cf. \cite{CG97, VV99}. The theorem below can be viewed as a generalization of the induction theorem of Kazhdan-Lusztig \cite{KL87} for affine Hecke algebras as adapted by Ariki \cite{Ar96}.

\begin{alphatheorem} [Theorem \ref{thm:induced}, Proposition \ref{prop:Morita} and Theorem \ref{thm:ind=std}]
\label{thm:B}
\qquad
    \begin{enumerate}
        \item 
        For $\la \in \Par^\ell(n)$, the $\hwS(n)$-module $\pi_\bfs (\Delta^\la)$ admits an induced module structure. 
        \item The algebras $\hwS(n)$ and $\KS(n,n)$ are Morita equivalent, inducing an isomorphism $\natural: [\qAW\modf_\eps] \stackrel{\cong}{\longrightarrow}\K$.
        \item $\natural$ matches the standard modules:  $\natural([\Delta_\bfm] )= [M_\bfm]$, for all $\bfm\in\cM\cS$. In particular, $\{[{\Delta}_\bfm]\mid \bfm\in\cM\cS\}$ forms a $\Z$-linear basis for $[\qAW\modf_\eps].$
    \end{enumerate}
\end{alphatheorem}

There is another version of affine $q$-Schur algebra $\ASch(N,n)$ defined in \eqref{def:affineSchur} as an endomorphism algebra \cite{Gre99}. It is known that (see \cite{LXY26}) 
\[
\KS(N,n) \cong \ASch(N,n)
\]
for $N\ge n$, and all $\ASch(N,n)$ for $N\ge n$ are Morita equivalent to $\ASch(n,n)$. On the other hand, it was shown in \cite{SSW25} that $\ASch(n,n)$ is Morita equivalent to $\hwS(n)$, whence Part (2). Part (3) follows from the induction theorem for the affine Hecke algebra \cite{KL87, Ar96} by applying the Schur functor and Part (1) (=Theorem \ref{thm:induced}). 

\subsubsection{}
Lusztig \cite{Lus91, Lus98} used the composition subalgebra of the generic Hall algebra $\Hall$ of a cyclic quiver $\pen_p$ and its IC basis to realize $\U^-_v(\widehat\sll_p)$ and its canonical basis. The generic Hall algebra $\Hall$ is isomorphic to half the affine quantum group $\U^-_\cA(\widehat\gl_p)$ with $\cA=\Z[v,v^{-1}]$, and its IC basis is regarded as the canonical basis of the latter; cf. \cite{LT96, VV99}. The Drinfeld double of $\Hall$ (i.e., double Hall algebra) over $\C(v)$ can be naturally identified with $\U_v(\widehat\gl_p)$; cf. \cite{X97, DF15, DX17}. 

In \cite[\S12.3]{VV99} (building on \cite{CG97, GV93, GRV94}), it was pointed out that there exists a $\Z$-linear isomorphism (see Theorem \ref{thm:Hall=Schursimple}):
\begin{align}  \label{Hall=K}
    \imath_p: \HallZ \stackrel{\cong}{\longrightarrow} \K^*,
\end{align}
which sends a Hall basis $f_\bfm$ (corresponding to characteristic functions on orbits in the Hall algebra over finite fields) to the basis of standard modules $[M_\bfm]$, for $\bfm\in\cM\cS$. Moreover, $\imath_p$ sends the canonical basis to the basis dual to the simple modules. 

The Hall algebra $\Hall$ is generated by the semisimple generators (corresponding to the semisimple quiver representations), and a closed multiplication formula for $f_\bfi f_\bfm$ between a semisimple generator $f_\bfi$ (for $\bfi\in\N\I$) and any Hall basis element is obtained by Du-Fu \cite{DF15}; also see \cite{FL19} for a geometric proof. We view $\Hall$ as a left regular $\U_\cA^-(\widehat\gl_p)$-module via left multiplication and the identification $\U^-_\cA(\widehat\gl_p)\cong\Hall$. 

On the other hand, we define the $\bfi$-restriction and $\bfi$-induction functors diagrammatically which give rise to the linear operators $f_\bfi$ on $[\qAW \modf_\eps]^*$ and $f_\bfi$ and $e_\bfi$ on $[\qW_{\eps^{2\bfs}} \modf]^*$. 

\begin{alphatheorem}
     [Theorem \ref{thm:Cat_affine}]
     \label{thm:C}
 Let $p\le \infty$.
 \begin{enumerate}
 \item The $f_\bfi$ via $\bfi$-restriction functors provide an action of $\U^-_\Z(\widehat\gl_p)$ on $[\qAW\modf_\eps]^*.$
     \item 
     There exists a $\U^-_\Z(\widehat\gl_p)$-module isomorphism
\[
\jmath_p: \HallZ \longrightarrow [\qAW\modf_\eps]^*, 
\qquad \jmath_p(f_\bfm) =[{\Delta}_\bfm]^*.
\]
\item 
The map $\jmath_p$ sends the canonical basis to the basis dual to the simple modules.
 \end{enumerate} 
\end{alphatheorem}

Theorem \ref{thm:C} (1)-(2) strengthens the basis-preserving linear isomorphism $\imath_p$ of Ginzburg, Varagnolo and Vasserot in \eqref{Hall=K}. We compute $\iRes$ on ${\Delta}_\bfm$ via Theorem \ref{thm:A} and show that it matches exactly with the Du-Fu multiplication formula on $\Hall$, whence (1)--(2). Part (3) follows from the following commutative diagram of $\Z$-linear isomorphisms (as there is no known geometric $\U^-_\Z(\widehat\gl_p)$-action on $\K^*$): 
\begin{equation}   \label{eq:ijk}
\begin{tikzcd}
\HallZ\ar[r,"\imath_p"]\ar[dr,"\jmath_p"] &\K^*\ar[d,"\natural^*"]
\\
 & {[}\qAW\modf_\eps{]}^*
\end{tikzcd} 
\end{equation}

\subsubsection{}
Denote by $\Lambda(\omega_\bfs)$ the integrable  $\U_v(\widehat\gl_p)$-module of highest weight $\omega_\bfs$. Denote by $\varpi_\bfs: \Hall \rightarrow \Lambda(\omega_\bfs)\!_\cA$ the $\U^-_\cA(\widehat\gl_p)$-module homomorphism sending $1$ to the highest weight vector $v^+_{\omega_\bfs}$. It was proved by Schiffmann \cite{Sch00} that the canonical basis on $\Hall$ descends via 
$\varpi_\bfs$ to a canonical basis on $\Lambda (\omega_\bfs)$, the latter being part of Uglov's canonical basis on a higher level Fock space \cite{Ugl00}; moreover, the canonical basis on $\Lambda (\omega_\bfs)$ admits a positive integral expansion into the Schur basis of the Fock space. This generalizes the level one conjecture in \cite{VV99} building on \cite{LT96}. The following provides a generalization of the main results of Ariki and Varagnolo-Vasserot.

\begin{alphatheorem}  [Proposition \ref{prop:lieaction}, Theorem \ref{thm:cyc_cat}]
    \label{thm:D}  
We have the following commutative diagram of $\U^-_\Z(\widehat\gl_p)$-module homomorphisms: 
\begin{equation}  
\label{diag:comm}
\begin{tikzcd}
{\HallZ} \ar[r,"\jmath_p"]\ar[d,two heads,"\varpi_\bfs"]&{[}\qAW \modf_\eps{]}^* \ar[d,two heads,"\pi_\bfs^*"]\\
\Lambda (\omega_\bfs)_\Z  \ar[r,"\jmath_\bfs"]&{[}\qW_{\eps^{2\bfs}} \modf{]}^* 
\end{tikzcd} 
\end{equation}
Moreover, the map $\jmath_\bfs$ is a $\U_\Z(\widehat\gl_p)$-module isomorphism and sends the canonical basis to the basis dual to simple modules. 
\end{alphatheorem}
Theorem \ref{thm:D} is consistent with (and recovers via Schur functors) the classic results of Ariki \cite{Ar96}; see the diagram \eqref{CD:HeckeWeb}.

Denote by $\qW_{\eps^{2\bfs}} \prmod$ the category of finite-dimensional projective $\qW_{\eps^{2\bfs}}$-modules. Via the Cartan pairing, one can replace $[\qW_{\eps^{2\bfs}}\modf]^*$ and the duals of the simple modules in Theorem \ref{thm:D}  by $[\qW_{\eps^{2\bfs}}\prmod]$ and indecomposable projective modules $P^\la$. Then Theorem \ref{thm:D} can be rephrased as computing the cell filtration multiplicities of $P^\la$ in terms of $v$-decomposition matrix (specialized at $v=1$) from  canonical basis to Schur basis as well as providing a classification of simple modules over the cyclotomic  $q$-$W$-Schur algebras; see Corollaries \ref{cor:PIM}--\ref{cor:classifysimple}.


\subsubsection{}
The clockwise orientation of a cyclic quiver (i.e. $i\mapsto i+1$, for $i\in\I$) is adopted in \cite{VV99, Sch00, DF15}, while the opposite orientation is used in \cite{Ar96, LTV99}. In this paper, we follow the clockwise convention mainly because we need to match our computation in $\qAW$ exactly with (the $v=1$ specialization of) the multiplication formula in $\Hall$ with semisimple generators from \cite{DF15}; see Remarks \ref{rem:sign}--\ref{rem:opposite}. 

The constructions of $\bfi$-restriction and $\bfi$-induction functors for cyclotomic $q$-web categories formulated in this paper work equally well for the cyclotomic $q$-Schur categories $\qSch_{\eps^{2\bfs}}$, which categorify Uglov's Fock space \cite{Ugl00}. This follows from 
\cite{RSVV} thanks to the identification of the path algebra of cyclotomic $q$-Schur categories with cyclotomic $q$-Schur algebras. Our diagrammatic constructions of $\bfi$-restriction and $\bfi$-induction functors are more general and conceptual than Wada's \cite{Wad14}.

\subsubsection{}
The paper is organized as follows. In Section \ref{sec:web}, we introduce the restriction and induction functors for the $q$-web categories. In Section \ref{sec:cellular}, we construct and study the cellular modules $\Delta^\la$ of $\wS_\bfu(n)$. A semisimplicity criterion for $\wS_\bfu(n)$ is given. In Section~ \ref{sec:standard}, we define the standard modules of $\hwS(n)$ as the pullback of cell $\wS_\bfu(n)$-modules, and show that they are always induced modules. Theorem \ref{thm:C} is the key categorification result of this paper and will be established in Section \ref{sec:affineSchur_Hall}. Theorem \ref{thm:D} is to a large extent a consequence of Theorem \ref{thm:C} and will be proved in Section \ref{sec:Cat}. We end the paper with a discussion of some open problems in \S\ref{ssec:open}.

\vspace{2mm}
%
{\bf Acknowledgement.} 
We thank Ming Lu and Changjian Su for helpful discussions on Hall algebra and geometric representation theory. We thank Department of Mathematics and IMS at University of Virginia and at National University of Singapore, Academia Sinica and the NSTC in Taipei, as well as University of Hong Kong for support and facilitating this collaboration. L.S. is partially supported by NSFC (Grant No. 12071346). W.W. is partially supported by DMS--2401351. 

%
\section{Restriction and induction functors for $q$-web categories}
\label{sec:web}

In this section, we recall from \cite{SSW25} the affine $q$-web category $\qAW$ and its cyclotomic quotients $\qW_\bfu$, which are a quantum version of the affine and cyclotomic web categories introduced in \cite{SW25web}. We determine the center of $\qAW$. We formulate the restriction functors on $\qAW$ and $\qW_\bfu$ as well as the induction functors on $\qW_\bfu$.

\subsection{Definition of $\qAW$}

Let $\kk =\C[q,q^{-1}]$.
For $n, a \in \N$, the $ q$-numbers and $q$-binomial coefficients are
\begin{align}
\label{qbinom}
    [n]_q =\frac{q^n -q^{-n}}{q-q^{-1}},
    \qquad
    \qbinom{n}{a}_q =\frac{[n]_q[n-1]_q\cdots [n-a+1]_q}{[a]_q!}.
\end{align}

\begin{definition}
  \label{def:qAWeb}
The affine $q$-web category $\qAW$ is the strict $\kk$-linear monoidal category generated by objects $a\in \mathbb Z_{\ge1}$. The object $a$ and its identity morphism  will be drawn as a vertical strand labeled by $a$,
$\stra.$
The generating morphisms are the merges, splits, (thick) positive crossings and thick dots depicted as
\begin{align}
\label{merge+split+crossing}
\merge 
&:(a,b) \rightarrow (a+b),&
\splits
&:(a+b)\rightarrow (a,b),&
\crossingpos
&:(a,b) \rightarrow (b,a), &
\wkdotaa , 
\end{align}
for $a,b \in \Z_{\ge 1}$, where $\crossingpos$ is invertible with inverse being the negative crossing denoted $\crossingneg := \crossingpos^{-1}$ and $\xdota$ is invertible with inverse denoted by $\zdota$. 
These generating morphisms are subject to the following relations \eqref{webassoc}--\eqref{dotmovesplits+merge}, for $a,b,c,d \in \Z_{\ge 1}$ with $d-a=c-b$:
\begin{align}
\label{webassoc}
\begin{tikzpicture}[baseline = 0,color=\clr]
	\draw[-,thick] (0.35,-.3) to (0.08,0.14);
	\draw[-,thick] (0.1,-.3) to (-0.04,-0.06);
	\draw[-,line width=1pt] (0.085,.14) to (-0.035,-0.06);
	\draw[-,thick] (-0.2,-.3) to (0.07,0.14);
	\draw[-,line width=1.5pt] (0.08,.45) to (0.08,.1);
        \node at (0.45,-.41) {$\scriptstyle c$};
        \node at (0.07,-.4) {$\scriptstyle b$};
        \node at (-0.28,-.41) {$\scriptstyle a$};
\end{tikzpicture}
&=
\begin{tikzpicture}[baseline = 0, color=\clr]
	\draw[-,thick] (0.36,-.3) to (0.09,0.14);
	\draw[-,thick] (0.06,-.3) to (0.2,-.05);
	\draw[-,line width=.8pt] (0.07,.14) to (0.19,-.06);
	\draw[-,thick] (-0.19,-.3) to (0.08,0.14);
	\draw[-,line width=1.5pt] (0.08,.45) to (0.08,.1);
        \node at (0.45,-.41) {$\scriptstyle c$};
        \node at (0.07,-.4) {$\scriptstyle b$};
        \node at (-0.28,-.41) {$\scriptstyle a$};
\end{tikzpicture}\:,
\qquad
\begin{tikzpicture}[baseline = -1mm, color=\clr]
	\draw[-,thick] (0.35,.3) to (0.08,-0.14);
	\draw[-,thick] (0.1,.3) to (-0.04,0.06);
	\draw[-,line width=1pt] (0.085,-.14) to (-0.035,0.06);
	\draw[-,thick] (-0.2,.3) to (0.07,-0.14);
	\draw[-,line width=1.5pt] (0.08,-.45) to (0.08,-.1);
        \node at (0.45,.4) {$\scriptstyle c$};
        \node at (0.07,.42) {$\scriptstyle b$};
        \node at (-0.28,.4) {$\scriptstyle a$};
\end{tikzpicture}
=\begin{tikzpicture}[baseline = -1mm, color=\clr]
	\draw[-,thick] (0.36,.3) to (0.09,-0.14);
	\draw[-,thick] (0.06,.3) to (0.2,.05);
	\draw[-,line width=1pt] (0.07,-.14) to (0.19,.06);
	\draw[-,thick] (-0.19,.3) to (0.08,-0.14);
	\draw[-,line width=1.5pt] (0.08,-.45) to (0.08,-.1);
        \node at (0.45,.4) {$\scriptstyle c$};
        \node at (0.07,.42) {$\scriptstyle b$};
        \node at (-0.28,.4) {$\scriptstyle a$};
\end{tikzpicture}\:,
\\
\label{mergesplit}
\begin{tikzpicture}[baseline = 7.5pt,scale=.8, color=\clr]
	\draw[-,line width=1pt] (0,0) to (.28,.3) to (.28,.7) to (0,1);
	\draw[-,line width=1pt] (.6,0) to (.31,.3) to (.31,.7) to (.6,1);
        \node at (0,1.13) {$\scriptstyle b$};
        \node at (0.63,1.13) {$\scriptstyle d$};
        \node at (0,-.1) {$\scriptstyle a$};
        \node at (0.63,-.1) {$\scriptstyle c$};
\end{tikzpicture}
&=
\sum_{\substack{0 \leq s \leq \min(a,b)\\0 \leq t \leq \min(c,d)\\t-s=d-a}}
q^{st}  
\begin{tikzpicture}[baseline = 7.5pt,scale=.8, color=\clr]
\draw[-,thick] (0.58,0) to (0.58,.2) to (.02,.8) to (.02,1);
	\draw[-,line width=4pt,white] (0.02,0) to (0.02,.2) to (.58,.8) to (.58,1);
	\draw[-,thick] (0.02,0) to (0.02,.2) to (.58,.8) to (.58,1);
	\draw[-,thick] (0,0) to (0,1);
	\draw[-,thick] (0.6,0) to (0.6,1);
        \node at (0,1.13) {$\scriptstyle b$};
        \node at (0.6,1.13) {$\scriptstyle d$};
        \node at (0,-.1) {$\scriptstyle a$};
        \node at (0.6,-.1) {$\scriptstyle c$};
        \node at (-0.1,.5) {$\scriptstyle s$};
        \node at (0.75,.5) {$\scriptstyle t$};
\end{tikzpicture},
\\
  \label{splitbinomial}
\begin{tikzpicture}[baseline = -1mm,scale=.8,color=\clr]
\draw[-,line width=1.5pt] (0.08,-.8) to (0.08,-.5);
\draw[-,line width=1.5pt] (0.08,.3) to (0.08,.6);
\draw[-,thick] (0.1,-.51) to [out=45,in=-45] (0.1,.31);
\draw[-,thick] (0.06,-.51) to [out=135,in=-135] (0.06,.31);
\node at (-.33,-.05) {$\scriptstyle a$};
\node at (.45,-.05) {$\scriptstyle b$};
\end{tikzpicture}
&= 
\qbinomsm{a+b}{a}_q\:
\begin{tikzpicture}[baseline = -1mm,scale=.6,color=\clr]
	\draw[-,line width=1.5pt] (0.08,-.8) to (0.08,.6);
        \node at (.08,-1) {$\scriptstyle a+b$};
\end{tikzpicture}, 
\\
\label{dotmovecrossing}
\begin{tikzpicture}[baseline=-1mm, scale=.8, color=\clr]
	\draw[-,line width=1pt] (0.3,-.3) to (-.3,.4);
	\draw[-,line width=4pt,white] (-0.3,-.3) to (.3,.4);
	\draw[-,line width=1pt] (-0.3,-.3) to (.3,.4);
        \node at (-0.3,-.4) {$\scriptstyle a$};
        \node at (0.3,-.4) {$\scriptstyle b$};
     \draw(-0.15,-0.12)   \bdot;
\end{tikzpicture}
 & =
\begin{tikzpicture}[baseline=-1mm, scale=.8, color=\clr]
 \draw[-,line width=1pt] (-0.3,-.3) to (.3,.4);
	\draw[-,line width=4pt,white] (0.3,-.3) to (-.3,.4);
  \draw[-,line width=1pt] (0.3,-.3) to (-.3,.4);
        \node at (-0.3,-.4) {$\scriptstyle a$};
        \node at (0.3,-.4) {$\scriptstyle b$};
     \draw(0.15,0.23)   \bdot;
\end{tikzpicture}
,  
    \qquad
\begin{tikzpicture}[baseline=-1mm, scale=.8, color=\clr]
	\draw[-,line width=1pt] (0.3,-.3) to (-.3,.4);
	\draw[-,line width=4pt,white] (-0.3,-.3) to (.3,.4);
	\draw[-,line width=1pt] (-0.3,-.3) to (.3,.4);
        \node at (-0.3,-.4) {$\scriptstyle a$};
        \node at (0.3,-.4) {$\scriptstyle b$};
     \draw(-0.15,0.23)   \bdot;
\end{tikzpicture}
  =
\begin{tikzpicture}[baseline=-1mm, scale=.8, color=\clr]
 \draw[-,line width=1pt] (-0.3,-.3) to (.3,.4);
	\draw[-,line width=4pt,white] (0.3,-.3) to (-.3,.4);
  \draw[-,line width=1pt] (0.3,-.3) to (-.3,.4);
        \node at (-0.3,-.4) {$\scriptstyle a$};
        \node at (0.3,-.4) {$\scriptstyle b$};
     \draw(0.15,-0.12)   \bdot;
\end{tikzpicture}
,
\\
 \label{dotmovesplits+merge}
\begin{tikzpicture}[baseline = -.5mm,scale=.8,color=\clr]
\draw[-,line width=1.5pt] (0.08,-.5) to (0.08,0.04);
\draw[-,line width=1pt] (0.34,.5) to (0.08,0);
\draw[-,line width=1pt] (-0.2,.5) to (0.08,0);
\node at (-0.22,.6) {$\scriptstyle a$};
\node at (0.36,.65) {$\scriptstyle b$};
\draw (0.08,-.2) \bdot;
\end{tikzpicture} 
& =
\begin{tikzpicture}[baseline = -.5mm,scale=.8,color=\clr]
\draw[-,line width=1.5pt] (0.08,-.5) to (0.08,0.04);
\draw[-,line width=1pt] (0.34,.5) to (0.08,0);
\draw[-,line width=1pt] (-0.2,.5) to (0.08,0);
\node at (-0.22,.6) {$\scriptstyle a$};
\node at (0.36,.65) {$\scriptstyle b$};
\draw (-.05,.24) \bdot;
\draw (.22,.24) \bdot;
\end{tikzpicture},
   \quad 
\begin{tikzpicture}[baseline = -.5mm, scale=.8, color=\clr]
\draw[-,line width=1pt] (0.3,-.5) to (0.08,0.04);
\draw[-,line width=1pt] (-0.2,-.5) to (0.08,0.04);
\draw[-,line width=1.5pt] (0.08,.6) to (0.08,0);
\node at (-0.22,-.6) {$\scriptstyle a$};
\node at (0.35,-.6) {$\scriptstyle b$};
\draw (0.08,.2) \bdot;
\end{tikzpicture}
 =~
\begin{tikzpicture}[baseline = -.5mm,scale=.8, color=\clr]
\draw[-,line width=1pt] (0.3,-.5) to (0.08,0.04);
\draw[-,line width=1pt] (-0.2,-.5) to (0.08,0.04);
\draw[-,line width=1.5pt] (0.08,.6) to (0.08,0);
\node at (-0.22,-.6) {$\scriptstyle a$};
\node at (0.35,-.6) {$\scriptstyle b$};
\draw (-.08,-.3) \bdot; \draw (.22,-.3) \bdot;
\end{tikzpicture} .
\end{align}
\end{definition}
The following relations hold in $\qAW$:
\begin{align}
  \label{intergralballon}
\begin{tikzpicture}[baseline = 1.5mm, scale=.5, color=\clr]
\draw[-, line width=1.2pt] (0.5,2) to (0.5,2.5);
\draw[-, line width=1.2pt] (0.5,0) to (0.5,-.4);
\draw[-,thin]  (0.5,2) to[out=left,in=up] (-.5,1)
to[out=down,in=left] (0.5,0);
\draw[-,thin]  (0.5,2) to[out=left,in=up] (0,1)
 to[out=down,in=left] (0.5,0);      
\draw[-,thin] (0.5,0)to[out=right,in=down] (1.5,1)
to[out=up,in=right] (0.5,2);
\draw[-,thin] (0.5,0)to[out=right,in=down] (1,1) to[out=up,in=right] (0.5,2);
\node at (0.5,.7){$\scriptstyle \cdots$};
\draw (-0.5,1) \bdot; 
\draw (0,1) \bdot; 
\node at (0.5,-.6) {$\scriptstyle a$};
\draw (1,1) \bdot;
\draw (1.5,1) \bdot; 
\node at (-.22,0) {$\scriptstyle 1$};
\node at (1.2,0) {$\scriptstyle 1$};
\node at (.3,0.3) {$\scriptstyle 1$};
\node at (.7,0.3) {$\scriptstyle 1$};
\end{tikzpicture}
   & =
   [a]_q!  
\begin{tikzpicture}[baseline = 1.5mm, scale=.5, color=\clr]
\draw[-,line width=1.2pt] (0,-0.1) to[out=up, in=down] (0,1.4);
\draw(0,0.6) \bdot; 
\node at (0,-.3) {$\scriptstyle a$};
\end{tikzpicture} ,
\qquad
\begin{tikzpicture}[baseline = 1.5mm, scale=.5, color=\clr]
\draw[-, line width=1.2pt] (0.5,2) to (0.5,2.5);
\draw[-, line width=1.2pt] (0.5,0) to (0.5,-.4);
\draw[-,thin]  (0.5,2) to[out=left,in=up] (-.5,1)
to[out=down,in=left] (0.5,0);
\draw[-,thin]  (0.5,2) to[out=left,in=up] (0,1)
 to[out=down,in=left] (0.5,0);      
\draw[-,thin] (0.5,0)to[out=right,in=down] (1.5,1)
to[out=up,in=right] (0.5,2);
\draw[-,thin] (0.5,0)to[out=right,in=down] (1,1) to[out=up,in=right] (0.5,2);
\node at (0.5,.7){$\scriptstyle \cdots$};
\draw (-0.5,1) \zdot; 
\draw (0,1) \zdot; 
\node at (0.5,-.6) {$\scriptstyle a$};
\draw (1,1) \zdot;
\draw (1.5,1) \zdot; 
\node at (-.22,0) {$\scriptstyle 1$};
\node at (1.2,0) {$\scriptstyle 1$};
\node at (.3,0.3) {$\scriptstyle 1$};
\node at (.7,0.3) {$\scriptstyle 1$};
\end{tikzpicture}
    ~=~ [a]_q!  
\begin{tikzpicture}[baseline = 1.5mm, scale=.5, color=\clr]
\draw[-,line width=1.2pt] (0,-0.1) to[out=up, in=down] (0,1.4);
\draw(0,0.6) \zdot; 
\node at (0,-.3) {$\scriptstyle a$};
\end{tikzpicture} .
\end{align}
The $a$-fold merges (and splits) used in \eqref{intergralballon} are defined by using \eqref{webassoc} via iterated merges (and splits), and then the relations \eqref{intergralballon} follow from \eqref{webassoc}, \eqref{splitbinomial} and \eqref{dotmovesplits+merge}. 

We define the following morphisms in $\qAW$
\begin{equation}
\label{eq:wkdota}
\omega_{a,r} :=
\begin{tikzpicture}[baseline = -1mm,scale=.7,color=\clr]
\draw[-,line width=1.2pt] (0.08,-.7) to (0.08,.5);
\node at (.08,-.9) {$\scriptstyle a$};
\draw(0.08,-0.1) \bdot;
\draw(.5,0)node {$\scriptstyle \omega_r$};
\end{tikzpicture}
:=
\begin{tikzpicture}[baseline = -1mm,scale=.7,color=\clr]
\draw[-,line width=1.5pt] (0.08,-.8) to (0.08,-.5);
\draw[-,line width=1.5pt] (0.08,.3) to (0.08,.6);
\draw[-,line width=1pt] (0.1,-.51) to [out=45,in=-45] (0.1,.31);
\draw[-,line width=1pt] (0.06,-.51) to [out=135,in=-135] (0.06,.31);
\draw(-.1,-0.1) \bdot;
\node at (-.25,-.35) {$\scriptstyle r$};
\node at (.08,-1.1) {$\scriptstyle a$};
\end{tikzpicture} \; ,
\qquad
\omega_{a,-r} :=
\begin{tikzpicture}[baseline = -1mm,scale=.7,color=\clr]
\draw[-,line width=1.2pt] (0.08,-.7) to (0.08,.5);
\node at (.08,-.9) {$\scriptstyle a$};
\draw(0.08,-0.1) \bdot;
\draw(.6,0)node {$\scriptstyle \omega_{-r}$};
\end{tikzpicture}
:=
\begin{tikzpicture}[baseline = -1mm,scale=.7,color=\clr]
\draw[-,line width=1.2pt] (0.08,-.7) to (0.08,.5);
\node at (.08,-.9) {$\scriptstyle a$};
\draw(0.08,-0.1) \zdot;
\draw(.5,0)node {$\scriptstyle \omega_r$};
\end{tikzpicture}
:=
\begin{tikzpicture}[baseline = -1mm,scale=.7,color=\clr]
\draw[-,line width=1.5pt] (0.08,-.8) to (0.08,-.5);
\draw[-,line width=1.5pt] (0.08,.3) to (0.08,.6);
\draw[-,line width=1pt] (0.1,-.51) to [out=45,in=-45] (0.1,.31);
\draw[-,line width=1pt] (0.06,-.51) to [out=135,in=-135] (0.06,.31);
\draw(.27,-0.1) \zdot;
\node at (.4,-.35) {$\scriptstyle r$};
\node at (.08,-.9) {$\scriptstyle a$};
\end{tikzpicture}  
\qquad\quad (0 \le r\le a).
\end{equation}
In particular, we have 
$
\begin{tikzpicture}[baseline = 3pt, scale=0.4, color=\clr]
\draw[-,line width=1.2pt] (0,0) to[out=up, in=down] (0,1.4);
\draw(0,0.6) \bdot;
\draw(0.7,0.6)node {$\scriptstyle \omega_a$};
\node at (0,-.22) {$\scriptstyle a$};
\end{tikzpicture}
=\xdota$
and 
$\begin{tikzpicture}[baseline = 3pt, scale=0.4, color=\clr]
\draw[-,line width=1.2pt] (0,0) to[out=up, in=down] (0,1.4);
\draw(0,0.6) \zdot;
\draw(0.7,0.6)node {$\scriptstyle \omega_a$};
\node at (0,-.22) {$\scriptstyle a$};
\end{tikzpicture}
=\zdota$.

For any small $\kk$-linear category $\mathcal C$,  its path algebra is the locally unital algebra  $\oplus_{i,j\in \text{ob}\mathcal C}\Hom_{\mathcal C}(i,j)$ with multiplication induced by composition of morphisms. 

Recall a (resp. strict) composition $\lambda=(\lambda_1,\ldots,\lambda_k)$ of $n$ is a tuple of positive (resp. non-negative) integers such that $\sum_{k}\lambda_k=n$; denote $|\lambda|=n$. Let $\Lambda(n)$ (resp. $\Lambda_\st(n)$) be the set of all (resp. strict) compositions of $n$.
The set of objects of $\qAW$ can be identified with the set of  all strict compositions (including $\emptyset$ which is identified with the unit object).   Let $\qAW(n) $ be the full subcategory 
of $\qAW$ with object set $\Lambda_\st(n)$. Denote by  $\hwS(n)$ and $\hwS$ the path algebras of $\qAW(n) $ and $\qAW$, respectively. Then $\hwS=\oplus_{n\in \N}\hwS(n)$. 
Moreover, it is proved in \cite{SSW25} that $\hwS(n)$ is Morita equivalent to the usual affine $q$-Schur algebra $\ASch(n,n)$; cf. \eqref{def:affineSchur}.

Let $\qAH$ be the full ``thin strand" subcategory of $\qAW$ with object set $\{(1^n)\mid n\in \N\}$. Then by the basis theorem given in \cite{SSW25}, $\qAH$ is the well-known strict monoidal category for the affine Hecke algebras, i.e., the endomorphism algebra $\End_{\qAH} ((1^n)) =\AHe(n)$. Here the affine Hecke algebra $\AHe(n)$ of type $GL_n$  is the unital associative $\kk$-algebra generated by $H_i^{\pm 1}$, $X_j^{\pm 1}$, for $1\le i\le n-1$ and $1\le j\le n$, subject to the relations:
\begin{align*}
    H_iH_i^{-1} &=1 =H_i^{-1}H_i, \qquad  (H_i+q)(H_i-q^{-1}) =0,
    \\
    H_iH_{i+1}H_i &=H_{i+1}H_iH_{i+1}, \qquad H_iH_j =H_jH_i \quad (|i-j|>1),
    \\
     X_iX_i^{-1} &=1 =X_i^{-1}X_i, \qquad\;  X_iX_j =X_jX_i,
    \\
    H_iX_iH_i &=X_{i+1}, \qquad\qquad\quad H_iX_j =X_jH_i  \quad (j\neq i,i+1).
\end{align*}
Diagrammatically, $H_i$ is the positive crossing between the $i$th and $(i+1)$st strands and $X_i$ is the thin dot on the $i$th strand. Moreover, by the basis theorem of the affine  $q$-webs in \cite{SSW25}, 
\begin{equation}
\label{equ:affwebhecke}
    \AHe(n)\cong 1_{(1^n)} \hwS(n) 1_{(1^n)}.
\end{equation}

\subsection{The center of $\qAW$}

The center $Z\big(\qAW(n)\big)$ of the small linear category $\qAW(n) $ is the commutative subalgebra of $\Pi_{\mu \in \Lambda_\st(n)} 1_\mu \qAW(n) 1_\mu$
consists of $(z_\mu)_{\mu\in \Lambda_\st(n)}$ such that $\theta z_\lambda =z_\mu \theta
$ for any $ \theta\in 1_\mu \qAW(n)1_\lambda$.
Then the center $\cZ_{\hwS(n)}$ of the algebra $\hwS(n)$ is 
  \[
  \cZ_{\hwS(n)}= \Big\{\sum_{\mu\in \Lambda_\st(n)}z_\mu\mid (z_\mu)_{\mu\in\Lambda_\st(n)}\in \cZ\big(\qAW(n)\big) \Big\}.
  \]
  The following is well known and it follows from the double centralizer property for $\hwS(n)$ and $\AHe(n)$; cf. \eqref{T:module}--\eqref{def:affineSchur}, for $N\ge n$.
\begin{proposition}
\label{prop:center}
    We have $\cZ_{\hwS(n)}\cong \cZ_{\AHe(n)}= \kk[X_1^{\pm},\ldots,X_n^{\pm}]^{\mathfrak S_n}$, where the isomorphism is given by $z=\sum_\mu z_\mu \mapsto z_{(1^n)}$.
\end{proposition}

\begin{example}
    A central element $z \in \cZ_{\hwS(n)}$ is determined by its thin strand summand $z_{(1^n)}$. Take 
$z_{(1^n)}=q^{n-1}(X_1+X_2+\ldots+ X_n)$. Let     $z_{\lambda}=\sum_{i=1}^kL^\lambda_i$, where    
    $$L^\lambda_i=  q^{n-\lambda_i} \begin{tikzpicture}[baseline = 3pt, scale=0.5, color=\clr]
\draw[-,line width=1pt] (-1,-.2) to (-1,1.2);
\node at (-1,-.5) {$\scriptstyle \lambda_1$};
\node at (-.4,.4) {$\ldots$};
\draw[-,line width=1pt] (1,-.2) to (1,1.2);
\draw(1,0.5) \bdot; \node at (1.6,0.6) {$\scriptstyle \omega_{1}$};
\node at (1,-.5) {$\scriptstyle \lambda_{i}$};
\node at (2.1,.3) {$\ldots$};
\draw[-,line width=1pt] (3,-.2) to (3,1.2);
\node at (3,-.5) {$\scriptstyle \lambda_k$};
\end{tikzpicture}  $$ 
for any $\lambda=(\lambda_1,\lambda_2,\ldots, \lambda_k)\in \Lambda_\st(n)$.
Then $z=\sum _{\lambda\in\Lambda_\st(n)}z_\lambda$.
\end{example}

\subsection{Restriction functors for $\qAW$}

For any $a\le n$, denote
\begin{align} \label{eq:1an}
\begin{split}
\Lambda^a(n) &=\{\mu\in \Lambda_\st(n)\mid \mu=(\mu_1,\mu_2,\ldots,\mu_k) \text{ such that } \mu_k=a\},
\\
\mathbf{1}_{a,n} &= \sum_{\mu\in \Lambda^a(n)}1_\mu. 
\end{split}
\end{align}
Then $\mathbf{1}_{a,n}$ is an idempotent of $\hwS(n)$. Moreover, by the basis theorem of $\qAW$ in \cite[Theorem 4.16]{SSW25},  there is an obvious embedding
\begin{equation}
\label{equ:embedding}
 \iota_{a,n}: \hwS(n-a) \longrightarrow \mathbf{1}_{a,n}\hwS(n)\mathbf{1}_{a,n}   
\end{equation}
sending any diagram $D$ to $D\stra$, the diagram obtained from $D$ by adding the  strand with thickness $a$ on the right.  
Then via the embedding $\iota_{a,n}$ we have a restriction functor
\[ \text{Res}^{n}_{n-a}=\mathbf{1}_{a,n}\hwS(n)\otimes _{\hwS(n)}?: \hwS(n)\modf\longrightarrow \hwS(n-a)\modf .\]

For any $f(X_1,\ldots,X_n)\in \kk[X_1^\pm,\ldots,X_n^\pm]^{\mathfrak S_n}$, let $z_f\in \cZ_{\hwS(n)}$ be the inverse of the isomorphism given in Proposition \ref{prop:center}. Given $\mathbf b=(b_1,b_2,\ldots,b_n)\in \kk^n$, we define the central character 
\[ 
\chi_{\mathbf b}: \cZ_{\hwS(n)}\longrightarrow \kk , \quad z_f\mapsto f(b_1,\ldots,b_n).
\]
Then $\chi_{\mathbf a}=\chi_{\mathbf b}$ if and only if $\mathbf a\sim \mathbf b$, i.e., $\mathbf a$ and $\mathbf b$ lie in the same orbit of $\mathfrak S_n$.

For any finite dimensional $M\in \hwS(n)\modf$ and $\gamma\in \kk^n/\sim$, let 
\[
M[\gamma]= \{v\in M\mid (z-\chi_{\gamma}(z))^kv=0 \text{ for } z\in \cZ_{\hwS(n)}, k\gg0\}. 
\]

For any $\gamma\in \kk^n/\sim$, let 
$\hwS(n)\modf_\gamma$ be the full category of $\hwS(n)\modf$ consisting of all modules $M$ with $M[\gamma]=M$.
Then $ \hwS(n)\modf=\oplus_{\gamma\in \kk^n/\sim}\hwS(n)\modf_\gamma $. 
 
For any multiset $\bfc=\{c_1,\ldots,c_a\}$ in $\kk$, we denote $|\bfc|=a$, and define the $\bfc$-restriction functor 
\begin{align}  \label{iRes_affine_n}
\bfc\text{-Res}^n_{n-|\bfc|}: \hwS(n)\modf\longrightarrow \hwS(n-a)\modf
\end{align}
which sends $M \in \hwS(n)\modf_\gamma$,
for each $\gamma$, to $ \text{ Proj}_{\gamma\setminus\bfc}\circ\text{Res}^n_{n-a}(M)$, where $ \text{ Proj}_{\gamma\setminus \bfc}$ denotes the projection from $\hwS(n-a)\modf $ to $\hwS(n-a)\modf_{\gamma\setminus \bfc}$, and  $\gamma\setminus \bfc $ is the orbit obtained from $\gamma$ by deleting $\bfc$ from $\gamma$.
 This leads to a functor 
\begin{align}  \label{iRes_affine}
\bfc\text{-Res} := \sum_{n} \bfc\text{-Res}^n_{n-|\bfc|}: \hwS\modf \longrightarrow \hwS\modf, 
\end{align}
which in turn gives rise to a linear operator $f_\bfc$ on $[\hwS\modf]$.

\subsection{Cyclotomic $q$-web categories $\qW_\bfu$}
Let
\begin{equation}
\label{equ:parau}
    \bfu=(  u_1, \ldots, u_\ell)\in (\kk^*)^\ell.
\end{equation}

\begin{definition}\cite[Definition 5.2]{SSW25}
 \label{def:cycWeb}
    Suppose  $\ell\ge 1$. For any $\bfu=(u_1,\ldots, u_\ell) $ in \eqref{equ:parau}, the cyclotomic $q$-web category $\qW_\bfu$ is the quotient category of $\qAW$ by the right tensor ideal generated by $\prod_{1\le j\le \ell }g_{r}(u_j)$, for $r\ge 1$, where
    \begin{equation*}
 g_{r}(u):= 
\sum_{0\le i\le r} (-u)^i q^{-i(r-i)} \:\:
 \begin{tikzpicture}[anchorbase,scale=.7,color=\clr]
\draw[-,line width=1.5pt] (0.08,-.6) to (0.08,.5);
\node at (.08,-.8) {$\scriptstyle r$};
\draw(0.08,0) \bdot;
\draw(.7,0)node {$\scriptstyle \omega_{r-i}$};
\end{tikzpicture}.
\end{equation*}
\end{definition}

Let $ \wS_\bfu(n)$ be the path algebra of the subcategory of $\qW_\bfu$ with the object set $\Lambda_\st(n)$. Then $\wS_\bfu:=\oplus _{n\in\N}\wS_\bfu(n)$ is the path algebra of $ \qW_\bfu$. Let $\qH_\bfu$ be the full subcategory of $\qW_\bfu$ with the object set $\{(1^n)\mid n\in \N\}$. Then the path algebra of the endomorphism algebra of $(1^n)$ in $\qH_\bfu$ (or in $\qW_\bfu$)
is the cyclotomic Hecke algebra associated to $\bfu$, $\bfH_\bfu(n):=\AHe(n)\big/\langle \prod_{a=1}^\ell (X_1-u_a)\rangle$. The path algebra of  $\qH_\bfu$ is $\bfH_\bfu:=\oplus_{n\in\N}\bfH_\bfu(n) $.
In particular, 
\begin{equation}
\label{equ:cycwebhecke}
    \bfH_\bfu(n)=1_{(1^n)} \wS_\bfu(n) 1_{(1^n)}.
\end{equation}
  
For any $\lambda\in \Lambda_\st(n)$, denote by $l(\lambda)$ the length of $\la$. 
For $\lambda,\mu\in \Lambda_{\text{st}}(n)$, let 
$\Mat_{\lambda,\mu}$ be the set of all $l(\lambda)\times l(\mu)$
non-negative integer matrices $A=(a_{ij})$ with row sum vector $\la$ and column sum vector $\mu$, i.e.,  $\sum_{1\le h\le l(\mu)}a_{ih}=\lambda_i$ and $\sum_{1\le h\le l(\lambda)}a_{hj}=\mu_j$ for all $1\le i\le l(\lambda) $ and $1\le j\le l(\mu)$.  Moreover, each matrix $A\in \Mat_{\lambda,\mu}$ can be identified with a reduced chicken foot ribbon  diagram \cite[\S 4.1]{SSW25}.  Denote  
\begin{equation*}
 \PMat_{\lambda,\mu}:=\{ (A, P)\mid  A=(a_{ij})\in \Mat_{\lambda,\mu}, P=(\nu_{ij}), \nu_{ij}\in \Par_{a_{ij}} \}.   
 \end{equation*} 
 For a matrix  $ P=(\nu_{ij})_{l(\lambda)\times l(\mu)}$ of partitions, denote 
$l(P)=\max\{l(\nu_{ij})\mid 1\le i\le l(\lambda), 1\le j\le l(\mu)\}.$
We define
\begin{equation}
\label{Def-spansetof-cycqweb}    
\PMat_{\lambda,\mu}^\ell:=\{(A,P)\in\PMat_{\lambda,\mu}\mid l(P)\le \ell-1\}.
\end{equation}

For any $a\le n$ and $\lambda\in \Lambda_\st(n)$, let 
\begin{align}
\mathbf R(\lambda)_a =\Big\{(\lambda_1-a_1, \ldots, \lambda_k-a_k)\in \Lambda(n-a)~\Big |~ \sum_{j=1}^ka_j=a \Big\}.
\end{align}
Here any $\mu \in \mathbf R(\lambda)_a$ can be viewed as a strict composition by omitting zero parts.  
For any $\mu=(\lambda_1-a_1, \ldots, \lambda_k-a_k) \in \mathbf R(\lambda)_a $, let $A_\mu=(a_{ij})\in\Mat_{\lambda,\mu\cup a}$
such that 
\[ a_{i,j}=\begin{cases} \lambda_i-a_i & \text{ if $i=j$}\\
    a_i &  \text{ if $j=k+1$}\\
    0& \text{ otherwise}
\end{cases}, 
\]
where $\mu\cup a=(\lambda_1-a_1, \ldots, \lambda_k-a_k,a) $. Define a subset of $\PMat_{\lambda,\mu\cup a}^\ell$ (cf. \eqref{Def-spansetof-cycqweb}):
\begin{align}
R(\lambda,\mu)=\Big\{(A_\mu,P)\in \PMat_{\lambda,\mu\cup a}^\ell \mid P=(\nu_{ij}) \text{ with } \nu_{ij}= \emptyset \text{ unless } j=k+1
\Big\}. 
\end{align}
Denote by  $c_{\lambda,\mu} $ be the cardinality of $R(\lambda,\mu)$.
Recall $\mathbf{1}_{a,n}= \sum_{\mu\in \Lambda^a(n)}1_\mu$ from \eqref{eq:1an}.

\begin{lemma}
\label{Lem:Projective}
For any $ \eta\in \Lambda_{\text{st}}(n)$, we have an isomorphism of right $\wS_\bfu(n-a)$-modules
\[
1_\eta \wS_\bfu(n)\mathbf {1}_{a,n}\cong \bigoplus_{\lambda\in \mathbf R(\eta)_a} \big(1_\lambda \wS_\bfu(n-a)\big)^{\oplus c_{\eta,\lambda}}. 
\]
In particular, $\wS_\bfu(n)\mathbf {1}_{a,n}$ is a right projective    $\wS_\bfu(n-a)$-module.  
\end{lemma}  

\begin{proof}
  By \cite[Theorem 5.17]{SSW25},  $\Hom_{\qW_{\bfu}}(\mu,\lambda)$ has a basis given by $\PMat_{\lambda,\mu}^\ell$ \eqref{Def-spansetof-cycqweb}, for any $\lambda,\mu \in\Lambda_{\text{st}}(n)$. It follows that 
     any elements in $1_\eta \wS_\bfu(n)\mathbf{1}_{a,n}$ can be written as a linear combination of elements 
     \[ 
     (A_\lambda, P)\circ ((B,Q)\otimes 1_{(a)}) 
     \]
     with $ (A_\lambda, P)\in R(\eta,\lambda)  $ and $(B,Q)\in \PMat_{\lambda,\mu}^\ell $, for $\mu \in \Lambda_{\text{st}}(n-a)$.
     Hence 
     \[
     1_\eta \wS_\bfu(n)\mathbf{1}_{a,n}=\bigoplus_{\lambda\in \mathbf R(\eta)_a, (A_\lambda, P)\in R(\eta,\lambda) } (A_\lambda, P)(\wS_\bfu(n-a)\otimes 1_{(a)}).
     \]
     Now the result follows since 
     \[ (A_\lambda, P)(\wS_\bfu(n-a)\otimes 1_{(a)})  \cong 1_\lambda \wS_\bfu(n-a) \]
     as a right $\wS_\bfu(n-a)$-module, which is projective. 
 \end{proof}

\subsection{Restriction functors and induction functors for $\qW_\bfu$}

Recall $\iota_{a,n}$  in \eqref{equ:embedding} for $\qAW$. Similarly, by the basis theorem of $\qW_\bfu$ in \cite[Theorem 5.17]{SSW25},  there is an obvious embedding from $\wS_\bfu(n-a) \rightarrow \mathbf{1}_{a,n}\wS_\bfu(n)\mathbf{1}_{a,n}$ for any $a\le n$. Then   
we introduce the induction functors and restriction functors 
\begin{align}  \label{eq:IndRes}
\begin{split}
\Ind_{n-a}^{n}:= \wS_\bfu(n)\mathbf{1}_{a,n}\otimes _{\wS_\bfu(n-a)} ? & : \wS_\bfu(n-a)\modf \longrightarrow \wS_\bfu(n)\modf. 
\\
\text{Res}^{n}_{n-a}=\mathbf{1}_{a,n}\wS_\bfu(n)\otimes _{\wS_\bfu(n)}? &: \wS_\bfu(n)\modf\longrightarrow \wS_\bfu(n-a)\modf.
\end{split}
\end{align}

\begin{lemma}
 The functors $\Ind^n_{n-a}$ and $\Res^{n}_{n-a}$ are exact.
\end{lemma}

\begin{proof}
  The exactness of $\Ind^n_{n-a}$ follows from Lemma  \ref{Lem:Projective} since    $\wS_\bfu(n)\mathbf{1}_{a,n}$ is a right projective $\wS_\bfu(n-a)$-module. The restriction functor is clearly exact.
\end{proof}

\section{Cellular structures on cyclotomic $q$-W-Schur algebras}
\label{sec:cellular}

In this section, we establish the cellular structure for $\wS_\bfu(n)$ via the corresponding cellular structure of the cyclotomic $q$-Schur category. We also study the cell filtration of the restriction and induction of the cell modules of $\wS_\bfu$.

 \subsection{Cellular bases and cell modules}   

\begin{definition} [\cite{SW24Schur, SSW25}]
\label{def:qSchur}
    The affine $q$-Schur category  $\qASch$ is a strict $\kk$-linear monoidal category with generating objects $a\in \Z_{\ge 1}$ and $u\in \kk$. We denote the object $a\in \N$ by $\stra$ and denote the object $u \in \kk$ by a red strand labeled by $u$ as $\stru .$
The morphisms are generated by 
\begin{equation} 
\label{generator-affschur}
  \begin{tikzpicture}[anchorbase, scale=0.4, color=\clr] 
                \draw[-,line width=1pt] (-2,0.2) to (-1.5,1);
	\draw[-,line width=1pt ] (-1,0.2) to (-1.5,1);
	\draw[-,line width=1.5pt] (-1.5,1) to (-1.5,1.8);
                \draw (-1,0) node{$\scriptstyle {b}$};
                \draw (-2,0) node{$\scriptstyle {a}$};
                \draw (-1.5,2.2) node{$\scriptstyle {a+b}$};
                \end{tikzpicture}, 
                \quad 
  \begin{tikzpicture}[anchorbase, scale=0.4, color=\clr]
                \draw[-,line width=1pt] (-2,1.8) to (-1.5,1);
	\draw[-,line width=1pt ] (-1,1.8) to (-1.5,1);
	\draw[-,line width=1.5pt] (-1.5,1) to (-1.5,0.2);
                \draw (-1,2.2) node{$\scriptstyle {b}$};
                \draw (-2,2.2) node{$\scriptstyle {a}$};
                \draw (-1.5,0) node{$\scriptstyle {a+b}$};
                \end{tikzpicture}, 
    \quad  \;
    \crossingpos , 
        \quad \;
    ~\xdota ,
 \end{equation}
and
\begin{equation}
\label{redcrossing}
\rightcrossing 
\;\; (\text{traverse-up}), 
                \qquad
                \leftcrossing 
 \;\; (\text{traverse-down}),    
\end{equation}
subject to relations \eqref{webassoc}--\eqref{intergralballon} for the generators in \eqref{generator-affschur} and additional relations \eqref{redslider}--\eqref{redbraid} involving also the traverse-ups and traverse-downs \eqref{redcrossing}:
\begin{align}
\label{redslider}
\begin{tikzpicture}[anchorbase,scale=.6,color=\clr]
	\draw[-,line width=1pt,color=\cred] (0.4,0) to (-0.6,1);
  \draw (-.6,1.2) node{$\scriptstyle \red{u}$};
	\draw[-,line width=1pt] (0.08,0) to (0.08,1);
	\draw[-,line width=1pt] (0.1,0) to (0.1,.6) to (.5,1);
        \node at (0.6,1.18) {$\scriptstyle a$};
        \node at (0.1,1.2) {$\scriptstyle b$};
\end{tikzpicture}
& =
\begin{tikzpicture}[anchorbase,scale=.6,color=\clr]
	\draw[-,line width=1pt,color=\cred] (0.7,0) to (-0.3,1);
  \draw (-.3,1.16) node{$\scriptstyle \red{u}$};
	\draw[-,line width=1pt] (0.08,0) to (0.08,1);
	\draw[-,line width=1pt] (0.1,0) to (0.1,.2) to (.9,1);
        \node at (0.9,1.18) {$\scriptstyle a$};
        \node at (0.1,1.2) {$\scriptstyle b$};
\end{tikzpicture},
\quad
\begin{tikzpicture}[anchorbase,scale=.6,color=\clr]
	\draw[-,line width=1pt,color=\cred] (-0.4,0) to (0.6,1);
  \draw (.6,1.16) node{$\scriptstyle \red{u}$};
	\draw[-,line width=1pt] (-0.08,0) to (-0.08,1);
	\draw[-,line width=1pt] (-0.1,0) to (-0.1,.6) to (-.5,1);
         \node at (-0.1,1.18) {$\scriptstyle b$};
        \node at (-0.6,1.2) {$\scriptstyle a$};
\end{tikzpicture}
\!\!=\!\!
\begin{tikzpicture}[anchorbase,scale=.6,color=\clr]
	\draw[-,line width=1pt,color=\cred] (-0.7,0) to (0.3,1);
 \draw (.3,1.16) node{$\scriptstyle \red{u}$};
	\draw[-,line width=1pt] (-0.08,0) to (-0.08,1);
	\draw[-,line width=1pt] (-0.1,0) to (-0.1,.2) to (-.9,1);
         \node at (-0.1,1.2) {$\scriptstyle b$};
        \node at (-0.95,1.18) {$\scriptstyle a$};
\end{tikzpicture}, 
\quad
\begin{tikzpicture}[baseline=-3.3mm,scale=.6,color=\clr]
\draw[-,line width=1pt,color=\cred] (0.4,.2) to (-0.6,-.8);
\draw (-.6,-1.05) node{$\scriptstyle \red{u}$};
\draw[-,line width=1pt] (0.08,0.2) to (0.08,-.75);
\draw[-,line width=1pt] (0.1,0.2) to (0.1,-.4) to (.5,-.8);
\node at (0.6,-1.05) {$\scriptstyle c$};
\node at (0.07,-1.05) {$\scriptstyle b$};
\end{tikzpicture}
\!\!=\!\!
\begin{tikzpicture}[baseline=-3.3mm,scale=.6,color=\clr]
\draw[-,line width=1pt,color=\cred] (0.7,0.2) to (-0.3,-.8);
\draw (-.3,-1.0) node{$\scriptstyle \red{u}$};
\draw[-,line width=1pt] (0.08,0.2) to (0.08,-.75);
\draw[-,line width=1pt] (0.1,0.2) to (0.1,0) to (.9,-.8);
\node at (1,-1.0) {$\scriptstyle c$};
\node at (0.1,-1.0) {$\scriptstyle b$};
\end{tikzpicture} ,
\quad
\begin{tikzpicture}[baseline=-3.3mm,scale=.6,color=\clr]
\draw[-,line width=1pt,color=\cred] (-0.4,0.2) to (0.6,-.8);
\draw (.6,-.91) node{$\scriptstyle \red{u}$};
\draw[-,line width=1pt] (-0.08,0.2) to (-0.08,-.75);
\draw[-,line width=1pt] (-0.1,0.2) to (-0.1,-.4) to (-.5,-.8);
\node at (-0.1,-1.0) {$\scriptstyle b$};
\node at (-0.6,-1.0) {$\scriptstyle a$};
\end{tikzpicture}
\!\!=\!\!
\begin{tikzpicture}[baseline=-3.3mm,scale=.6,color=\clr]
\draw[-,line width=1pt,color=\cred] (-0.7,0.2) to (0.3,-.8);
\draw (.3,-1) node{$\scriptstyle \red{u}$};
\draw[-,line width=1pt] (-0.08,0.2) to (-0.08,-.75);
\draw[-,line width=1pt] (-0.1,0.2) to (-0.1,0) to (-.9,-.8);
\node at (-0.1,-1.0) {$\scriptstyle b$};
\node at (-0.95,-1.0) {$\scriptstyle a$};
\end{tikzpicture} ,
\\
\label{redcross2}
\begin{split}
\begin{tikzpicture}[baseline = 7pt, scale=0.3, color=\clr]
\draw[-,line width=1pt,color=\cred](0.1,-.22)to (0,2.5);
\draw (0.1,-.5) node{$\scriptstyle \red{u}$};
\draw[-,line width=1.2pt](-.5,-.1) to[out=45, in=down] (.5,1.4);
\draw[-,line width=1.2pt](.5,1.4) to[out=up, in=300] (-.5,2.4);
\draw (-.5,-0.5) node{$\scriptstyle {a}$};
\end{tikzpicture}
 & =
\sum_{0\le t \le a} (-1)^t q^{-t(a-t)} u^t
\begin{tikzpicture}[baseline = 7pt, scale=0.3, color=\clr]
\draw[-,line width=1pt,color=\cred](-.5,-.25)to (-.5,2.15);
\draw (-.5,-.6) node{$\scriptstyle \red{u}$};
\draw[-,line width=1pt] (-1.3,2.15) to (-1.3,-0.2);   
\draw (-1.3,1) \bdot;
\draw (-2.5,1.4) node{$\scriptstyle \omega_{a-t}$};
\draw (-1.3,-.6) node{$\scriptstyle {a}$};
\end{tikzpicture} ~, 
\\    
\begin{tikzpicture}[baseline = 7pt, scale=0.3, color=\clr]
\draw[-,line width=1pt,color=\cred](-0.1,-.6)to (-0.1,2.2);
\draw (-0.1,-.9) node{$\scriptstyle \red{u}$};
\draw[-,line width=1.2pt](.5,-.5) to[out=135, in=down] (-.5,1);
\draw[-,line width=1.2pt](-.5,1) to[out=up, in=270] (.5,2.2);
\draw (.5,-0.9) node{$\scriptstyle {a}$};
\end{tikzpicture}
&=
\sum_{0\le t \le a} (-1)^t q^{-t(a-t)} u^t
\begin{tikzpicture}[baseline = 7pt, scale=0.3, color=\clr]
\draw[-,line width=1pt,color=\cred](-1.8,-.25)to (-1.8,2.15);
\draw (-1.8,-.6) node{$\scriptstyle \red{u}$};
\draw[-,line width=1pt] (-1,2.15) to (-1,-0.25);          
\draw (-1,1) \bdot;
\draw (0.4,1)  node{$\scriptstyle \omega_{a-t}$};
\draw (-1,-.6) node{$\scriptstyle {a}$};
\end{tikzpicture} ,
\end{split}
\\
 \label{redbraid}
\begin{tikzpicture}[anchorbase, scale=0.35, color=\clr]
\draw[-,line width=1pt](0,2) to (1.5,-.35);
\draw[-,line width=4pt, white](0,-.35) to (1.5,2); to (1.5,-.35);
\draw[-,line width=1pt] (0,-.35) to (1.5,2);
\draw[-,line width=1pt,color=\cred](.8,2) to[out=down,in=90] (0.2,1) to [out=down,in=up] (.8,-.4);
\draw (.8,-.8) node{$\scriptstyle \red{u}$};
\node at (0, -.7) {$\scriptstyle a$};
\node at (1.5, -.7) {$\scriptstyle b$};
\end{tikzpicture}
& = 
\sum_{s=0}^{\min (a,b)} (-1)^s q^{s(s-1)/2} (q^{-1} -q)^s [s]_q! 
\begin{tikzpicture}[anchorbase, scale=0.35, color=\clr]
\draw[-,line width=1.5pt](0,-.35) to   (0,0);
\draw[-,line width=1.5pt](0,1.65) to   (0,2);
\draw[-,line width=1.5pt](1.5,-.35) to   (1.5,0);
\draw[-,line width=1.5pt](1.5,1.65) to   (1.5,2);
\draw[-,thick](0,0) to   (0,1.65);
\draw[-,thick](1.5,0) to   (1.5,1.65);
\draw[-,line width=1pt](0,1.65) to   (1.5,0);
\draw[-,line width=4pt, white] (0.2,0.2) to (1.3,1.35);
\draw[-,line width=1pt] (0,0) to (1.5,1.65);
\draw[-,line width=1pt,color=\cred](.8,2) to
[out=-60,in=60]
(.8,-.4);
\draw (.8,-.7) node{$\scriptstyle \red{u}$};
\node at (0, -.7) {$\scriptstyle a$};
\node at (1.5, -.7) {$\scriptstyle b$};
\node at (-0.3, 1.1) {$\scriptstyle s$};
\node at (1.8, 1.1) {$\scriptstyle s$};
\draw (1.5, .7) \bdot;
\end{tikzpicture} . 
\end{align}
\end{definition}

Let $\Lambda_{\text{st}}^{1+\ell}(n)$ be the set of all strict $(1+\ell)$-multicompositions of $n$. Write 
\[
\Lambda_{\text{st}}^{1+\ell}:=\bigcup_{m\in \N}\Lambda_{\text{st}}^{1+\ell}(m).
\]
Thus, the set of objects of $\qASch$
can be identified with $ \bigcup_{\ell\in \N} \big(\Lambda_{\text{st}}^{1+\ell}\times \kk^{\ell} \big)$.

Below we will identify 
$\lambda^{(0)}\red{u_1}\ldots \red{u_\ell} \lambda^{(\ell)}$ with $\lambda=(\lambda^{(0)},\ldots, \lambda^{(\ell)})$ once $\bfu$ is fixed in \eqref{equ:parau}.
For any fixed $\bfu$, the cyclotomic $q$-Schur category is the quotient of the $\kk$-linear category $\qSchu$
    by the relations 
    \begin{equation}
    \label{eq:unitlambda}
       \unit{\lambda}=0 \text{ for all } \lambda =(\la^{(0)},\la^{(1)},\ldots,\la^{(\ell)}) \in \Lambda_{\text{st}}^{1+\ell} \text{ with }
\lambda^{(0)}\neq \emptyset.
\end{equation}

By \eqref{eq:unitlambda}, the set of objects of $\qSchu$ can be identified with $\cup_{n\in \N}\Lambda_{\text{st}}^{\ell}(n)$ by deleting the leftmost $\emptyset$.
Let $\Par^\ell(n)$  be the set of $\ell$-multipartitions  of $m$ with usual dominance order $\lhd$. For $\la\in \Par^\ell(n)$, we denote $|\la|=n$. 

Recall for any $\lambda\in\Par^\ell(n)$ a $\lambda$-tableau $\mathbf S$ is obtained from the Young diagram $Y(\lambda)$ by inserting ordered pairs $(i,p)$, also denoted by $i_p$, into $Y(\lambda)$, where $1\le p\le \ell$ and $i$ is a positive integer.
For $\lambda\in\Par^\ell(n)$ and $\mu\in \Lambda_{st}^\ell(n)$, let $\SST(\lambda,\mu)$ be the set of all semistandard $\lambda$-tableaux of type $\mu$ as defined in \cite{DJM98}. Associated to a semistandard tableau $\bT\in \SST(\lambda,\mu)$  we can construct a matrix $A_{\bT}\in \Mat_{\overline{\mu},\overline{\lambda}}$ as in \cite[\S 4.5]{SW24Schur} such that $a_{(p,i),(q,j)}$ is the number of $i_p$ in $j$th row of $\mathbf T^{(q)}$, where $\bar \lambda $ is the strict composition obtained from the multi-composition  $\lambda$ by a forgetful map. Thus, we have a reduced chicken foot ribbon diagram $[A_{\bT}]$  of shape $A_{\bT}$. Let $[\bT]$ denote the ornamentation associated with $(A_{\bT},(\emptyset))$. We refer to \cite[Example 4.8]{SW24Schur} for an explicit example for the $A_\bT$ and the diagram $[\bT]$, where each crossing therein is now chosen to be positive.

The notion of cellular bases was introduced by Graham-Lehrer \cite{GL96}. Let $\bfS_\bfu(n)$ be the path algebra of the subcategory  $\qSchu(n)$ with object set $\Lambda_{\text{st}}^{\ell}(n)$. 
We have the following diagrammatic description of a cellular basis on $\bfS_\bfu(n)$, which is compatible with its counterpart on the cyclotomic $q$-Schur algebra \cite{DJM98} via the algebra isomorphism established in \cite{SSW25}. 

\begin{proposition}  \cite[Theorems 5.12, 5.13]{SSW25}
    The algebra $\bfS_\bfu(n)$ has a cellular  basis given by 
\begin{align}  \label{doubleSST}
\bigcup_{\overset{\lambda\in \Par^\ell(n)}{\mu,\nu\in \Lambda_{\text{st}}^\ell(n)} }
\big\{ [\bT]\circ [\bS]^\div \mid \bT\in \SST(\lambda,\nu), \bS\in \SST(\lambda,\mu) \big\}, 
\end{align}
where the required anti-involution $\div$ is defined by rotating a string diagram by 180 degrees around a horizontal axis.
\end{proposition}

By \cite[Theorem~ 5.17]{SSW25}, we have a fullly faithful functor from $\qW_\bfu$ to $\qSchu$ which sends any diagram $f$ to 
 $\begin{tikzpicture}[baseline = 10pt, scale=0.5, color=\clr]
\draw[-,line width =1pt,color=\cred] (-4.2,.2) to (-4.2,1.6);
\draw (-4.2,0) node{$\scriptstyle \red{u_1}$};    
\draw[-,line width =1pt,color=\cred] (-3.5,.2) to (-3.5,1.6);
\draw (-3.4,0) node{$\scriptstyle \red{u_2}$};
\draw(-2.6,.8) node{$\ldots$};
\draw[-,line width =1pt,color=\cred] (-2,.2) to (-2,1.6);
\draw (-1.9,0) node{$\scriptstyle \red{u_{\ell}}$};
\draw (-1.2,1) node{$f$};
\end{tikzpicture}.
$
The set $\Lambda_\st(n)$ of objects is sent to 
\[
\bar\Lambda^\ell(n):=\{ (\emptyset^{\ell-1}, \mu)\mid \mu\in \Lambda_\st(n)\}.
\]
The above functor induces an isomorphism of algebras 
\[
\phi_n: \wS_\bfu(n) \cong  {\bf1}_n \bfS_\bfu(n){\bf1}_n = \bigoplus_{\nu,\mu\in \bar \Lambda^\ell(n)} \Hom_{\qSchu}(\nu,\mu)  \] 
where  ${\bf1}_n=\oplus_{\nu\in \bar\Lambda^{\ell}(n)}1_\nu$ .
Below we shall identify $\wS_\bfu(n)$ with ${\bf1}_n \bfS_\bfu(n){\bf1}_n$ via the isomorphism $\phi_n$ and construct a cellular basis of $\wS_\bfu(n)$. 

Note that for any $\bT\in \SST(\lambda,\nu), \bS\in \SST(\lambda,\mu)$, 
\begin{equation}
  {\bf1}_n [\bT]\circ [\bS]^\div {\bf1}_n= \begin{cases}
    [\bT]\circ [\bS]^\div, & \text{ if }\mu,\nu \in \bar\Lambda^\ell(n)\\
    0,& \text{otherwise}.
\end{cases} 
\end{equation}
This implies the cellular basis for $\wS_\bfu$ as follows. 
\begin{theorem}
  \label{thm:cellularforwebu}
The algebra $\wS_\bfu(n)$ admits a cellular basis  
\begin{align}  \label{doubleSSTweb}
\bigcup_{\overset{\lambda\in \Par^\ell(n)}{\mu,\nu\in \bar\Lambda^\ell(n)} }
\big\{ [\bT]\circ [\bS]^\div \mid \bT\in \SST(\lambda,\mu), \bS\in \SST(\lambda,\nu) \big\}. 
\end{align}
\end{theorem}
 For any $\lambda\in \Par^\ell(n)$, the (left) cell module   $\Delta^\lambda$ for $\wS_\bfu(n)$ is the module with basis 
 \[ 
 \{[\bT]\circ [\bS]^\div + \wS_\bfu(n)^{\rhd\lambda} \mid \mu\in \bar\Lambda^\ell(n), \bT \in \SST(\lambda,\mu)\} 
 \]
for any fixed $\bS\in \SST(\lambda, \nu)$, where $\wS_\bfu(n)^{\rhd\lambda}$ is the span of elements $[\bT']\circ [\bS']^\div $ in \eqref{doubleSSTweb}
such that $\bT'\in \SST(\lambda',\nu') $, $\bS\in \SST(\lambda',\mu')$ and $\lambda'\rhd \lambda$.
Note that the cell module does not depend on the choice of $\bf S$. So, we even can write
$[\bT]\circ [\bS]^\div + \wS_\bfu(n)^{\rhd\lambda}$ as $[\bT]$ if there is no confusion.
\begin{rem}
\label{rem:cellularforcychecke}
Multiplying both sides of \eqref{doubleSSTweb} by the idempotent $1_{(1^n)}$, we obtain the well-known  cellular basis of the cyclotomic Hecke algebra $\bfH_\bfu(n)$ (e.g., \cite{DJM98}) 
\begin{align}  \label{doubleSSTheck}
\bigcup_{ \lambda\in \Par^\ell(n) }
\big\{ [\bT]\circ [\bS]^\div \mid \bS, \bT\in \SST(\lambda,(\emptyset^{\ell-1},(1^n))  \big\}. 
\end{align}
Then for any $\lambda\in  \Par^\ell(n)$ we have the cell (Specht) module $ S^\lambda$ of $\bfH_{\bfu}(n)$ with respect to the cellular basis above. Obviously, 
\begin{equation}
\label{equ:idemoncell}
 1_{(1^n)} \Delta^\lambda \cong  S^\lambda, \qquad
 \text{ for } \lambda \in  \Par^\ell(n).   
\end{equation}
\end{rem}

\subsection{Cell filtrations on $\Res \Delta^\lambda$ and $\Ind \Delta^\lambda$}

We identify a multi-partition $\lambda=(\lambda^{(1)}, \ldots, \lambda^{(\ell)})\in \Par^\ell(n)$ with its Young diagram $[\lambda]$. We write a node $x\in [\lambda]$ as $ (i,j,k)$ if $x$ is at $i$th row and $j$th column of the $k$-th component. We say $x$ is  removable (resp. addable) of $\lambda$ if $\lambda\backslash x$ (resp. $\lambda\cup x$) is a multi-partition of $n-1$ (resp. $n+1$). 

A set of nodes $\{x_1,x_2,\ldots,x_a\}$ is called removable (resp., addable) if $\lambda\backslash \{x_1,x_2,\ldots,x_a\}$ (resp. $\lambda\cup \{x_1,x_2,\ldots,x_a\}$) is a multi-partition of $n-a$ (resp. $n+a$). Let
\begin{align}
\begin{split}
    \mathcal R(\lambda)_a &=\{ \text{$\ell$-multipartitions  obtained from $\lambda$ by removing $a$ removable nodes}\},
    \\
    \bar {\mathcal R}(\lambda)_a &=\{\mu \in \mathcal R(\lambda)_a\mid \mu =\lambda\backslash \{x_1, \ldots,x_a\} \text{ with all
$x_i$ in distinct columns}\}.
\end{split}
\end{align}

Similarly, we have the set $\mathcal A(\lambda)_a$ of $\ell$-multipartitions obtained by adding $a$ addable nodes to $\lambda$, and 
\begin{align}
    \bar {\mathcal A}(\lambda)_a &=\{\mu \in \mathcal A(\lambda)_a\mid \mu =\lambda\cup \{x_1, \ldots,x_a\} \text{ with all
$x_i$ in distinct columns}\}.
\end{align}
In particular, we have $\bar {\mathcal R}(\lambda)_1=  \mathcal R(\lambda)_1$
and $\bar {\mathcal A}(\lambda)_1=  \mathcal A(\lambda)_1$.

Recall that a   $\lambda$-tableau $\mathbf S$ in $\SST(\lambda, \mu)$ is obtained from Young diagram $Y(\lambda)$ by inserting ordered pairs $(i,p)$, also denoted by $i_p$, into $Y(\lambda)$ satisfying certain conditions,  where $1\le p\le \ell$ and $i$ is a positive integer (cf. \cite[Definition (4.4)]{DJM98} or \cite[Ex. 4.2]{SW24Schur}). In particular, the number of times of $(i,p)$ as an entry in $\mathbf S$ is $\mu^{(p)}_i$, for all $i,p$. Let $\bar \Lambda^\ell(n)_a$ be the subset of $\bar\Lambda^\ell(n)$ consisting of all  elements such that  $\mu^{(\ell)}=(\mu_1,\ldots,\mu_k)$ with  $\mu_k=a$.
For any $\bT\in \SST(\lambda, \mu)$ with $\mu\in \bar\Lambda^{\ell}(n)$ 
we have $\textbf{1}_{a,n}[\bT]\neq 0$ only if $\mu\in \bar \Lambda^\ell(n)_a$.
In this case, we define $\bT^a$ to be the tableaux obtained from $\bT$ by deleting all $a$ of $k_{\ell}$'s. In other words, the diagram $[\bT^a]$ is obtained from $[\bT]$ by deleting the strands that connect to $\mu_k$.  

Since there are $a$ nodes $R=\{x_1,\ldots, x_a\}$ filled by $k_{\ell}$, we see that  $\bT^a\in \SST(\lambda\backslash R, \mu\backslash a ) $, where $\mu\backslash a\in \bar\Lambda^{\ell}(n-a)$ such that $(\mu\backslash a)^{(\ell)}=(\mu_1,\ldots,\mu_{k-1})$. In particular, 
since $\bT$ is semi-standard, $x_i$ and $x_j$ are not in the same column and hence $\lambda\backslash R\in \bar {\mathcal R}(\lambda)_a$. Moreover, any $\bS\in \SST(\lambda\backslash R, \mu\backslash a )$
can be obtained in this way such that $\bS=\bT^a$.

Recall there is a unique $\bT^{\lambda\backslash R}\in \SST(\lambda\backslash R, \lambda\backslash R) $ such that $ [\bT^{\lambda\backslash R}]=1_{\lambda\backslash R}$.
Let $\bT^{\lambda\backslash R}_a \in \SST(\lambda,\lambda\backslash R)$ be the tableaux obtained from $\bT^{\lambda\backslash R}$  by adding the nodes  $R$ back. 
Then we see that 
\[
[\bT]=[\bT^a] \circ [\bT_a^{\lambda\backslash R}]. 
\]
Recall the restriction functor $\Res^n_{n-a}: \wS_\bfu(n)\modf \rightarrow\wS_\bfu(n-a)\modf$ from \eqref{eq:IndRes}.

\begin{proposition}  \label{prop:restriction}
    For any $\lambda\in \Par^\ell(n)$, there exists a filtration of $\wS_\bfu(n-a)$-modules 
    \[0=M_{k+1}\subset M_k\subset \ldots \subset M_1=\Res^n_{n-a} (\Delta^\lambda) \]
    such that $M_i/M_{i+1}\cong \Delta^{\mu^i}$, 
    where $\{\mu^1\lhd \mu^2\lhd\ldots\lhd\mu^k \}=\bar {\mathcal R}(\lambda)_a$. 
\end{proposition}

\begin{proof}
    Define 
    \[M_j=\{[\bT] \mid \bT \in \SST(\lambda,\mu), \mu\in \bar\Lambda^\ell(n)_a,  
    \bT^a\in \SST(\mu^i,\mu\backslash a), i\ge j  \} .\]
    It suffices to show that the linear map 
    \begin{align*}
        \psi : \Delta^{\mu^j} \longrightarrow M_j/M_{j+1}, \qquad
      [\bT^a] \mapsto  [\bT] + M_{j+1},
    \end{align*}
    is a homomorphism (and hence an isomorphism) of $\wS_\bfu(n-a)$-modules.

For any $h\in \wS_\bfu(n-a)$, we have 
\[ h [\bT^a]=\sum_{\bS^a\in \SST(\mu^j,\mu\backslash a)} c_{\bT^a,\bS^a}
[\bS^a] +\sum_{\nu\rhd \mu^j, \bT_1\in \SST(\nu,\gamma),\bS_1\in \SST(\nu,\mu^j)} (*) \,[\bT_1]\circ [\bS_1]^\div \]
It suffices to establish the claim that 
\[([\bT_1]\circ [\bS_1]^\div)\otimes 1_a \circ [\bT_a^{\mu^j}]\in M^{j+1}.\]
 
  Note that $( [\bS_1]^\div\otimes 1_a) \circ [\bT_a^{\mu^j}]\in \Hom_{\qW_\bfu}(\nu\cup a, \lambda)$.
 By \cite[Proposition 5.11]{SSW25}, 
  \[ ( [\bS_1]^\div\otimes 1_a) \circ [\bT_a^{\mu^j}]= \sum_{\bT_2,\bS_2} (*)\,[\bT_2]\circ [\bS_2]^\div   \]
  where the sum is over $\bT_2\in \SST(\tilde \nu,\nu\cup a)$ and $\bS_2\in \SST(\tilde \nu,\lambda)$.
  Note that $ \SST(\tilde \nu,\lambda)\neq \emptyset$ only if $ \lambda \unlhd \tilde\nu$. 
  If   $\tilde \nu \ntrianglelefteq \lambda$, then 
  $[\bT_2]\circ [\bS_2]^\div=0$ in $\Delta^\lambda$ and the claim trivially holds. 
  So, it remains to consider the case 
  \[ \mu^j\cup a \unlhd \nu\cup a \unlhd\tilde \nu = \lambda.\]
  This condition implies $\nu\in \{ \mu^{j+1},\ldots,\mu^k\}$.
  Furthermore, $ ([\bS_1]^\div \otimes 1_a) \circ \bT^{\mu^j}_a =\bT^{\mu^h}_a$ for some $ \nu=\mu^h$ with $j+1\le h\le k$.
  Hence the claim holds. This completes the proof.
\end{proof}

\begin{proposition}
\label{prop:induction}
    For any $\lambda\in \Par^\ell(n)$, there exists a filtration of $\wS_\bfu(n+a)$-modules 
    \[0=M_{k+1}\subset M_k\subset \ldots \subset M_1=\Ind^{n+a}_{n} (\Delta^\lambda) \]
    such that $M_i/M_{i+1}\cong \Delta^{\mu^i}$, 
    where $\{\mu^1\lhd \mu^2\lhd\ldots\lhd\mu^k \}=\bar {\mathcal A}(\lambda)_a$. 
\end{proposition}

\begin{proof}
    By definition, we have 
    \begin{align*}
        \Ind^{n+a}_{n} (\Delta^\lambda)&= \wS_{\bfu}(n+a) \textbf{1}_{a,n+a}\otimes _{\wS_\bfu(n)} \Delta^\lambda\\
        &=\wS_{\bfu}(n+a) \textbf{1}_{a,n+a}\otimes _{\wS_\bfu(n)}  \wS_\bfu(n) 1_\lambda\circ [\bS]^\div / \wS_\bfu^{\rhd\lambda }\\
        &= M^\lambda/N^\lambda,
    \end{align*} 
    where, for any fixed $\bS\in \SST(\lambda, (\emptyset^{\ell-1},\bar \lambda))$ for some $\bar\lambda\in\Lambda_\st(n)$,
   \begin{align*}
       M^\lambda&= \wS_{\bfu}(n+a) \textbf{1}_{a,n+a}\otimes _{\wS_\bfu(n)}  \wS_\bfu(n) 1_\lambda\circ [\bS]^\div, \\  N^\lambda&=\wS_{\bfu}(n+a) \textbf{1}_{a,n+a}\otimes _{\wS_\bfu(n)}\wS_\bfu^{\rhd\lambda }.
   \end{align*} 
    We write $\lambda\cup a=(\lambda^{(1)}, \lambda^{(2)},\ldots ,\lambda^{(\ell)}\cup a)$, where $\lambda^{(\ell)}\cup a $ is obtained by attaching the part $a$ to the end of $\lambda^{(\ell)}$. Using the cellular basis of $\wS_\bfu(n+a)$ in Theorem \ref{thm:cellularforwebu}, we see that 
    $M^\lambda$ is spanned by 
\[\{ [\bT_1]\circ [\bT_2]^\div \otimes [\bS]^\div \mid \bT_1\in \SST(\mu, \nu), \bT_2\in \SST(\mu, \lambda\cup a), \mu\in \Par^\ell(n+a), \nu \in \bar\Lambda^{\ell}(n+a)\}.\]
Let $R=\{x_1,\ldots,x_a\}$ be the nodes in $\bT_2$ filled by $k_\ell$ where $k=l(\lambda^{(\ell)})+1$. 
Recall that  $\bT_2^a=\bT_2\backslash R\in \SST(\mu\backslash R, \lambda)$.  If this set is not $\emptyset$, we must have $ \mu\backslash R\unrhd \lambda$. 

If $ \mu \backslash R\rhd \lambda$, then $ [\bT_2 ]^\div= [\bT_a^{\mu\backslash R}]\div [\bT_2^a]^\div $ 
and hence $ [\bT_1]\circ [\bT_2]^\div \otimes [\bS]^\div=[\bT_1]\circ[\bT_a^{\mu\backslash R}]\div\otimes [\bT_2^a]^\div[\bS]^\div \in N^\lambda  $ since $ \mu \backslash R\rhd \lambda$. 

If $\mu \backslash R =\lambda$, then we must have $\mu\in \bar{\mathcal A}(\lambda)_a$  and $[\bT_2^a]=1_\lambda$ and $\bT_2=\bT_a^{\mu\backslash R} $.
This implies that $ \Ind^{n+a}_{n} (\Delta^\lambda)$ is spanned by 
\[
B^\lambda :=\Big\{[\bT_1]\circ (\bT_a^{\lambda})^\div \otimes [\bS]^\div  +N^\lambda ~\Big |~ \bT_1\in \SST(\mu,\nu),\mu\in  \bar{\mathcal A}(\lambda)_a , \nu \in \bar\Lambda^{\ell}(n+a) \Big\}.
\] 

Note that the dimension of $\Ind^{n+a}_{n} (\Delta^\lambda) $ is independent of the field and parameters by the proof of  Lemma \ref{Lem:Projective}, and so we can assume that $\wS_\bfu$ is seimisimple when we compute its dimension. In this case, using the result on the restriction functor in Proposition~ \ref{prop:restriction} and the fact that $\Ind_n^{n+a}$ is left adjoint to  
$ \Res_n^{n+a}$, we see that this dimension is $|B^\lambda|$ by the Frobenius reciprocity. Therefore $B^\lambda $ is a basis of $\Ind^{n+a}_{n} (\Delta^\lambda)$.

Now we define 
\[
M_j= \Big\{[\bT_1]\circ (\bT_a^{\lambda})^\div \otimes [\bS]^\div  +N^\lambda ~\Big |~\bT_1\in \SST(\mu^i,\nu), i\ge j , \nu \in \bar\Lambda^{\ell}(n+a) \Big\}.
\]
The above arguments also show that  the linear map 
    \begin{align*}
        \psi : \Delta^{\mu^j}&\longrightarrow M_j/M_{j+1}, \qquad
      [\bT^a] \mapsto  [\bT] + M_{j+1}
    \end{align*}
    is a homomorphism (and hence an isomorphism) of $\wS_\bfu(n+a)$-modules.
\end{proof}

\subsection{Functors $\iInd$ and $\iRes$}

We specialize at $v=\eps$, where $\eps$ is either generic (i.e., not a root of 1) or $\eps=\exp(2\pi i/p')$, a primitive $p'$-th root of 1. We set $p=p'$ if $p'$ is odd and $p=p'/2$ if $p'$ is even. In other words, 
\begin{align}
    \text{$\text{order}(\eps^2)=p$   in all cases (including $p=\infty$ for $\eps$ generic). }
\end{align}
Set 
\begin{align} \label{eq:I}
\I =\Z \text{ or } \Z/p\Z, 
    \end{align}
corresponding to the case for $\eps$ generic or when $\text{order}(\eps^2)=p$. We denote $\Z/p\Z =\{\overline{0}, \overline{1}, \ldots, \overline{p-1}\}$. 

Let $\bfs =(s_1, \ldots, s_\ell) \in \I^\ell$ and $\bfu =(u_1,\ldots, u_\ell) \in (\kk^\times)^\ell$. We shall make the assumption on parameters that 
\begin{align} \label{u:parameter}
(u_1,\ldots, u_\ell)= (\eps^{2s_1}, \ldots, \eps^{2s_\ell}), \qquad \text{ or } \bfu =\eps^{2\bfs} \text{ for short}. 
\end{align}
Note this is well defined by \eqref{eq:I}. The arbitrary parameter cases can be reduced to this by Morita equivalences (cf. \cite{DM02}). 
For any node $x=(i,j,k)\in [\lambda]$, we define the residue of $x$ by 
\[
\res(x)=u_k\eps^{-2(j-i)} = \eps^{s_k-2(j-i)}
.\]

It is known that the central elements 
$f(X_1,\ldots,X_n)\in \kk[X_1^\pm,\ldots,X_n^\pm]^{\mathfrak S_n} $ acts on $1_{(1^n)}\Delta^\lambda$ (the corresponding cell module of the cyclotomic Hecke algebra) as the scalar
\[f(\res(x_1), \ldots, \res(x_n))\]
where $x_1,\ldots, x_n$ are all the nodes of $[\lambda]$. 
Let $\gamma_\lambda=(\res(x_1), \ldots, \res(x_n))$. Then 
$\Delta^\lambda\in \wS_\bfu(n)\modf_{\gamma_\lambda}$ and we can identify $\gamma_\lambda$ with $\sum_{j=1}^n\alpha_{t_j}$ if $ \gamma_\lambda=(\eps^{-2t_1}, \ldots, \eps^{-2t_n} )$. Here $\{\alpha_i|i\in \I\}$ denote the simple roots of $\widehat {\mathfrak {sl}}_p$.

Similar to $\gamma$, we can identify the multiset $\bfi=\{i_1,\ldots, i_a\}$ with $ \alpha_{i_1}+\ldots+ \alpha_{i_a}$, and denote the multiset $\eps^{-2\bfi}:=\{\eps^{-2i_1},\ldots, \eps^{-2i_a}\}$.  We denote $|\bfi|=a$ and let 
\begin{align*}
    \bar{ \mathcal R}(\lambda)_{\bfi} &=\Big\{\mu=\lambda\backslash\{x_1,\ldots,x_a\} \in \bar {\mathcal R}(\lambda)_{|\bfi|} ~\Big |~ 
    \{\res(x_1),\ldots, \res(x_a) \}= \eps^{-2\bfi} 
    \Big\}.
\end{align*}
Similarly, we have the subset 
$\bar{ \mathcal A}(\lambda)_{\bfi} $ of $\bar{ \mathcal A}(\lambda)_{|\bfi|}$.

Similar to $\bfc\text{-}\Res^n_{n-|\bfk|}$ from \eqref{iRes_affine_n} in the affine setting (now with $\bfc=\eps^{-2\bfi}$), we define a functor $\eps^{-2\bfi}\text{-}\Res^n_{n-|\bfi|}$ (and simply denoted by $\iRes^n_{n-|\bfi|}$ from now on) sending any $M\in \wS_\bfu(n)\modf_\gamma$ to $\text{Proj}_{\gamma\setminus\eps^{-2\bfi}}\circ \text{Res}^n_{n-|\bfi|}(M) \in \wS_\bfu(n-|\bfi|)\modf_{\gamma\setminus\eps^{-2\bfi}} $.
Similarly, we have a functor $\iInd^{n+|\bfi|}_n$ (which is a simplified notation for $\eps^{-2\bfi}\text{-}\Ind^{n+|\bfi|}_n$), which sends $M\in \wS_\bfu(n)\modf_\gamma$ to a module in $\wS_\bfu(n+a)\modf_{\gamma\cup\eps^{-2\bfi}} $. Similarly, we will denote the linear operators $f_\bfc$ with $\bfc={\eps^{-2\bfi}}$ by $f_\bfi$. The following result follows from  Propositions \ref{prop:restriction}--\ref{prop:induction}.

\begin{proposition}
\label{prop:IndRes}
Let $\bfi \in \Z \I$. 
\begin{enumerate}
    \item     For any $\lambda\in \Par^\ell(n)$, there exists a filtration of $\wS_\bfu(n-|\bfi|)$-modules 
    \[0=M_{k+1}\subset M_k\subset \ldots \subset M_1=\iRes^n_{n-|\bfi|} (\Delta^\lambda) \]
    such that $M_i/M_{i+1}\cong \Delta^{\mu_i}$, 
    where $\{\mu_1\lhd \mu_2\lhd\ldots\lhd\mu_k \}=\bar {\mathcal R}(\lambda)_\bfi$.
    \item For any $\lambda\in \Par^\ell(n)$, there exists a filtration of $\wS_\bfu(n+|\bfi|)$-modules 
    \[0=M_{k+1}\subset M_k\subset \ldots \subset M_1=\iInd^{n+|\bfi|}_{n} (\Delta^\lambda) \]
    such that $M_i/M_{i+1}\cong \Delta^{\mu_i}$, 
    where $\{\mu_1\lhd \mu_2\lhd\ldots\lhd\mu_k \}=\bar {\mathcal A}(\lambda)_\bfi$. 
\end{enumerate}    
\end{proposition}

We then define functors
\[
\iRes:= \oplus_{n\in \N}\iRes^n_{n-|\bfi|}, \qquad \iInd:= \oplus_{n\in\N}\iInd^{n+|\bfi|}_n
\] 
on $\qW_\bfu\modf.$

\subsection{A semisimplity criterion}

\begin{proposition}
\label{prop:semisimple} 
Suppose $\bfu=(u_1,\ldots,u_\ell)\in (\kk^*)^\ell$. Then $\wS_\bfu(n)$ is semisimple if and only if the cyclotomic Hecke algebra $\bfH_{\bfu}(n)$ is semisimple. 
\end{proposition}

\begin{proof}
   The only if part holds  since $\bfH_{\bfu}(n)= 1_{(1^n)}\wS_{\bfu}(n) 1_{(1^n)}$.

   If $\bfH_{\bfu}(n)$ is semisimple, then $ \bfS_\bfu(n)$ is semisimple since the latter is the endomorphism algebra of a direct sum of permutation modules of $\bfH_{\bfu}(n)$. Now $\wS_{\bfu}(n)$ is semisimple since 
   $ \wS_{\bfu}(n)=e\bfS_\bfu(n)e$, where the idempotent $e=\sum_\mu 1_\mu$ sums over all $\ell$-multicompositions $\mu$ of $n$ such that $\mu^{(k)}=\emptyset$ for $k<\ell$.
\end{proof}

The criterion for semisimplicity of the cyclotomic Hecke algebra $\bfH_{\bfu}(n)$ was worked out by Ariki; see \cite{Ar02}.

\section{Standard modules for affine $q$-Web categories}
\label{sec:standard}

In this section, we construct the standard modules over the affine $q$-Schur algebras $\hwS(n)$ and show that they are induced modules. We relate them to their counterparts over affine Hecke algebras via a Schur functor. We also review two more constructions of affine $q$-Schur algebras. The constructions and results in this section will play a supporting role in Section \ref{sec:affineSchur_Hall}. 

\subsection{Standard $\hwS(n)$-modules as induced modules}
\label{ssec:induced}

In this subsection, we consider the induction functor
$\Ind^{m+n}_{m,n}:= \hwS(m+n)\otimes_{\hwS_{m,n}} ?$
where $ \hwS_{m,n}=\hwS(m)\otimes \hwS(n)$. We will show that this is a functor from $\hwS_{m,n}$-mod to $\hwS(m+n)$-mod
(of finite-dimensional modules).
Similarly, for non-negative integers $n_i$ such that $n=\sum_{i=1}^\ell n_i$, we can define the induction functors 
\begin{align} \label{eq:Ind}
\Ind_{n_1,\ldots,n_\ell}^n := \hwS(n)\otimes_{\hwS_{n_1,\ldots,n_\ell}} ?.
\end{align}

Recall that for $\lambda,\mu\in \Lambda_{\text{st}}(n)$, each matrix $A\in \Mat_{\lambda,\mu}$ can be identified with a reduced Chicken foot ribbon diagram.

    For $\lambda\in \Lambda_{\text{st}}(n)$, $\mu^{(i)}\in \Lambda_{\text{st}}(n_i)$, $1\le i\le \ell$, we define 
  $C^\lambda_{\mu^{(1)}, \ldots,\mu^{(\ell)}}$ to be the subset of $\Mat_{\lambda, (\mu^{(1)}, \ldots, \mu^{(\ell)})}$
  consisting of matrices $A$ which satisfy Conditions (1)--(3): 
\begin{enumerate}
    \item each column of $A$ has exactly one non-zero entry; that is, there is no split in $A$. 
    \item 
    For $1\le k\le \ell$, and $1\le i\le l(\lambda)$, $|\text{row}(A)_{i,k}|\le 1$, where 
    \[
    \text{row}(A)_{i,k}:= 
    \Big\{ a_{ij}\neq 0 \Big | \sum_{1\le h\le k-1}l(\mu^{(h)}) +1 \le j\le\sum_{1\le h\le k}l(\mu^{(h)}) \Big\}. 
    \]
    \item If $a_{i_1,j_1}\neq 0$ and $a_{i_2,j_2}\neq 0$ for some 
        $i_1<i_2$ and $\sum_{1\le h\le k-1}l(\mu^{(h)}) +1 \le j_1,j_2\le\sum_{1\le h\le k}l(\mu^{(h)}) $, $1\le k\le \ell$, then we must have $j_1<j_2$.
\end{enumerate}

Recall $\omega_{a,r}=
\begin{tikzpicture}[baseline = -1mm,scale=.7,color=\clr]
\draw[-,line width=1.2pt] (0.08,-.7) to (0.08,.5);
\node at (.08,-.9) {$\scriptstyle a$};
\draw(0.08,-0.1) \bdot;
\draw(.5,0)node {$\scriptstyle \omega_r$};
\end{tikzpicture}$ in \eqref{eq:wkdota}. For any $a\in \Z_{>0}$, let
    \begin{equation}
    \label{eq:RPar}
          \RPar_a:=\{\lambda=(\lambda_1,\ldots,\lambda_k)\in \Z^k\mid \lambda_1\ge \lambda_2\ge \ldots\ge \lambda_k, \lambda_1-\lambda_k\le a, \lambda_i|\le a, \forall i\}.
    \end{equation}
    For any tuple $\nu=(\nu_1,\nu_2,\ldots, \nu_k)\in \RPar_a$, the {\em elementary dot packet (of thickness $a$)} is defined to be 
$\omega_{a, \nu}:= \omega_{a,\nu_1}\omega_{a,\nu_2}\cdots \omega_{a,\nu_k}\in \End_{\qAW}(a)$, and 
drawn as $
\begin{tikzpicture}[baseline = 3pt, scale=0.4, color=\clr]
\draw[-,line width=1.2pt] (0,-.2) to[out=up, in=down] (0,1.2);
\draw(0,0.5) \bdot; \node at (0.6,0.5) {$ \scriptstyle \nu$};
\node at (0,-.4) {$\scriptstyle a$};
\end{tikzpicture}$.

For $\lambda,\mu\in \Lambda_{\text{st}}(n)$, a $\lambda\times \mu$ {\em elementary chicken foot ribbon diagram (or simply elementary ribbon diagram)} is a $\lambda\times \mu$ reduced chicken foot ribbon diagram with an elementary dot packet $\omega_{a,\nu}$, for some tuple $\nu\in \RPar_a$, attached at the bottom of each leg with thickness $a$. 
Denote  
\begin{equation}\label{dottedreduced}
 \RPMat_{\lambda,\mu}:=\{ (A, P))\mid  A=(a_{ij})\in \Mat_{\lambda,\mu}, P=(\nu_{ij}), \nu_{ij}\in \RPar_{a_{ij}} \}.   
\end{equation}
We will identify the set of all elementary chicken foot ribbon diagrams from $\mu$ to $\lambda$ with $\RPMat_{\lambda,\mu}$.

\begin{lemma}
\label{lem:basisofmorphism}
  Let $\lambda\in \Lambda_\st(n)$ and $\mu= (\mu^{(1)}, \ldots, \mu^{(\ell)})$, with $\mu^{(i)}\in \Lambda_\st(n_i)$ for each $i$, and $n=\sum_in_i$. Then 
  $\Hom_{\qAW}(\mu,\lambda)$ has a basis given by 
  \[
  \{A \circ (B_1\otimes B_2\otimes \ldots \otimes B_\ell) \mid A\in C_{\eta^{(1)},\ldots,\eta^{(\ell)}}^\lambda, B_i \in \RPMat_{\eta^{(i)},\mu^{(i)}}, \eta^{(i)}\in \Lambda_\st(n_i), 1\le i\le \ell \}. 
  \]
\end{lemma}

\begin{proof}
    Recall that $\Hom_{\qAW}(\mu,\lambda)$ has a basis given by 
    $\RPMat_{\lambda,\mu}$ (see \cite[Theorem ~4.16]{SSW25}), For $\mu= (\mu^{(1)}, \ldots, \mu^{(\ell)})$, each $(C,P)\in \RPMat_{\lambda,\mu}$ can be uniquely divided into $\ell+1$ parts, $A, B_1,\ldots, B_\ell$, such that 
    $ (C,P)= A\circ (B_1\otimes B_2\otimes \ldots \otimes B_\ell)$ with  $ A\in C_{\eta^{(1)},\ldots,\eta^{(\ell)}}^\lambda$ and $B_i \in \RPMat_{\eta^{(i)},\mu^{(i)}}$ for some $\eta^{(i)}\in \Lambda_\st(n_i), 1\le i\le \ell $.
This completes the proof.    
\end{proof}

\begin{lemma}
\label{lem:dimofinduction}
Let $\lambda\in\Lambda_{\text{st}}(n)$, and $n_1,\ldots, n_\ell, n \in \N$ with $n=\sum_{i=1}^\ell n_i$. 
\begin{enumerate}
\item 
As right $\hwS_{n_1,\ldots,n_\ell}$-modules, we have
    \[
    1_\lambda \hwS(n)\cong
    \bigoplus_{A\in C^\lambda_{\eta^{(1)},\ldots,\eta^{(\ell)}}, \eta^{(i)}\in \Lambda_{\text{st}}(n_i)} 1_{(\eta^{(1)},\ldots,\eta^{(\ell)})} \hwS_{n_1,\ldots,n_\ell}^{(A)},
    \]
    where $\hwS_{n_1,\ldots,n_\ell}^{(A)}$ is an isomorphic copy of $\hwS_{n_1,\ldots,n_\ell}$ for each $A$.
    \item For any $M\in \hwS_{n_1,\ldots,n_\ell}$-mod, we have 
    \[\dim 1_\lambda \hwS(n)\otimes_{\hwS_{n_1,\ldots,n_\ell}}M= \sum_{\eta^{(i)}\in \Lambda_{\text{st}}(n_i),1\le i\le \ell}|C^\lambda_{\eta^{(1)},\ldots,\eta^{(\ell)}}| \dim 1_{\eta^{(1)},\ldots,\eta^{(\ell)}}M.  \]
\end{enumerate}
\end{lemma}

\begin{proof}
 (1) follows from Lemma \ref{lem:basisofmorphism} since 
 $\Hom_{\qAW}(\mu^{(i)},\eta^{(i)})$ has a basis given by $\RPMat_{\eta^{(i)},\mu^{(i)}}$ by 
  \cite{SSW25}. Part (2) follows from (1).
\end{proof}
Recall $\hwS(n)$-mod is the category of finite-dimensional $\hwS(n)$-modules. 

\begin{corollary}
Let $n=\sum_{i=1}^\ell n_i$.  
  The functor \eqref{eq:Ind} induces an exact functor 
  \[
  \Ind_{n_1,\ldots,n_\ell}^n: \hwS_{n_1,\ldots,n_\ell}\modf \longrightarrow \hwS(n)\modf.
  \]
\end{corollary}

 \begin{proof}
   The functor is well defined since by Lemma \ref{lem:dimofinduction} (2) it sends finite-dimensional modules to finite-dimensional modules. Moreover, by Lemma \ref{lem:dimofinduction} (1), $\hwS(n)$ is a projective right $\hwS_{n_1,\ldots,n_\ell}$-modules, and hence $\text{Ind}_{n_1,\ldots,n_\ell}^n$ is exact. \end{proof}   

The quotient functor  $ \qAW \rightarrow \qW_{\eps^{2\bfs}}$ induces a functor 
\[
\pi_\bfs: \wS_{\eps^{2\bfs}}\modf\longrightarrow \hwS\modf, 
\]
which in turn induces a linear map
$\pi_\bfs^*: [\hwS\modf]^* \rightarrow[\wS_{\eps^{2\bfs}}\modf]^*$. The following result will be instrumental in establishing the web generalization (see Theorem \ref{thm:ind=std}) of the induction theorem of Kazhdan-Lusztig \cite{KL87} and Ariki \cite{Ar96}. 

\begin{theorem} 
\label{thm:induced}
We have an isomorphism of $\hwS(n)$-modules:
\[ 
\Ind_{n_1,\ldots,n_\ell}^n(\pi_{s_1} \Delta^{\mu^{(1)}}\otimes\ldots \otimes \pi_{s_\ell} \Delta^{\mu^{(\ell)}} )\cong \pi_{\bfs} (\Delta^{\mu}),
\]
where $\mu=(\mu^{(1)}, \ldots, \mu^{(\ell)})\in \Par^\ell_{\text{st}}(n)$ with $|\mu^{(i)}|=n_i$ and $n= \sum_i n_i$.    
\end{theorem}

\begin{proof}
   By the Frobenius reciprocity, we have an isomorphism
\begin{align*}
\Phi &: \Hom_{\hwS_{n_1,\ldots,n_\ell}}\big(\pi_{s_1} \Delta^{\mu^{(1)}}\otimes\ldots \otimes \pi_{s_\ell} \Delta^{\mu^{(\ell)}}, \pi_\bfs \Delta^{\mu} \big)\\
&\stackrel{\sim}{\longrightarrow} \Hom_{\hwS(n)}\big(\Ind_{n_1,\ldots,n_\ell}^n \pi_{s_1} \Delta^{\mu^{(1)}}\otimes\ldots \otimes \pi_{s_\ell} \Delta^{\mu^{(\ell)}}, \pi_\bfs \Delta^{\mu}\big),
\\
& h \mapsto \Phi(h), \qquad \text{where}\quad \Phi(h)(g\otimes w )= gh(w),
\end{align*} 
for any $w\in \pi_{s_1} \Delta^{\mu^{(1)}}\otimes\ldots \otimes \pi_{s_\ell} \Delta^{\mu^{(\ell)}}$ and $g\in \hwS$. Consider the natural $\hwS_{n_1,\ldots,n_\ell}$-homomorphism 
\begin{align*}
    f: \pi_{s_1} \Delta^{\mu^{(1)}}\otimes\ldots \otimes \pi_{s_\ell} \Delta^{\mu^{(\ell)}}&\longrightarrow  \pi_\bfs \Delta^{\mu}, \\
     T_1\otimes T_2\otimes \ldots\otimes T_\ell &\mapsto (\tilde T_1,\tilde T_2,\ldots,\tilde T_\ell),
   \end{align*} 
   for any $T_i\in \SST(\mu^{(i)},\eta^{(i)})$.
   Here $(\tilde T_1,\tilde T_2,\ldots, \tilde T_\ell)\in \SST(\mu,\eta)$ is the semi-standard tableau
   obtained by moving the endpoints of   $T_i$   through all red lines. The theorem follows from the following.
   \vspace{2mm}
   
{\bf Claim.} The homomorphism 
$\Phi(f)$ sends 
   a basis of $\Ind_{n_1,\ldots,n_\ell}^n (\pi_{s_1} \Delta^{\mu^{(1)}} \otimes\ldots \otimes \pi_{s_\ell} \Delta^{\mu^{(\ell)}})$
   to a basis of $\pi_\bfs \Delta^{\mu}$
(i.e., it is an isomorphism).
   
It remains to prove the Claim. By Lemma \ref{lem:dimofinduction}, for any $\lambda\in \Lambda_\st(n)$, 
$1_\lambda\Ind_{n_1,\ldots,n_\ell}^n (\pi_{s_1} \Delta^{\mu^{(1)}}\otimes\ldots \otimes \pi_{s_\ell} \Delta^{\mu^{(\ell)}})$
has a basis given by 
\[
\big\{ A\otimes (T_1\otimes \ldots\otimes T_\ell) \mid A\in C^\lambda_{\eta^{(1)},\ldots,\eta^{(\ell)}}, \eta^{(i)}\in \Lambda_\st(n_i), T_i\in \SST(\mu^{(i)},\eta^{(i)}), 1\le i\le \ell \big\}.
\]
On the other hand, $1_\lambda \pi_\bfs \Delta^{\mu} $  has a basis given by  $\SST(\mu,\lambda)$. 

 To complete the proof, we note that 
 $\Phi(f)(A\otimes (T_1\otimes \ldots\otimes T_\ell))= A\circ (\tilde T_1,\ldots, \tilde T_\ell)  .$
 Conditions (1)--(3) for $C^\lambda_{\eta^{(1)},\ldots,\eta^{(\ell)}}$ in \S\ref{ssec:induced}
 guarantee that the diagram $A\circ (\tilde T_1,\ldots, \tilde T_\ell)$ is in 
 $\SST(\mu,\lambda)$. Finally,  $\Phi(f)$ is a bijection between these two bases, since any SST diagram in $\SST(\mu,\lambda)$ can be uniquely  divided  into 
 $\ell+1$ parts $ A\circ (\tilde T_1,\ldots, \tilde T_\ell )$ for some $A \in C^\lambda_{\eta^{(1)},\ldots,\eta^{(\ell)}}$ and $ T_i\in \SST(\mu^{(i)},\eta^{(i)})$. The Claim is proved. 
\end{proof}

\subsection{Induced modules for affine Hecke algebras}

By \eqref{equ:affwebhecke}--\eqref{equ:cycwebhecke}, we have two idempotent functors
\begin{align*}
    1_{(1^n)}: \hwS(n)\modf &\longrightarrow \AHe(n)\modf, \qquad M\mapsto 1_{(1^n)} M,
    \\
    1_{(1^n)}: \wS_\bfu(n)\modf &\longrightarrow \bfH_\bfu(n)\modf, \qquad N\mapsto 1_{(1^n)} N.
\end{align*}

Similarly to \eqref{eq:Ind}, we have the induction functor 
\[  \text{Ind}_{n_1, \ldots,n_\ell}^n:= \AHe(n) \otimes _{\AHe_{n_1,\ldots,n_\ell}}? \]
where $\AHe_{n_1,\ldots,n_\ell}= \AHe(n_1)\otimes \ldots \otimes\AHe(n_\ell)$. Recall that we have the cell module  $S^\lambda$ for the cyclotomic Hecke algebra $\bfH_\bfu(n)$ in Remark \ref{rem:cellularforcychecke}.
The following analogous result of Theorem \ref{thm:induced} for $\AHe(n)$ is well known (and can be proved similarly in any case).

\begin{proposition}
\label{prop:induhe}
    We have an isomorphism of $\AHe_n$-modules:
    \[ 
\Ind_{n_1,\ldots,n_\ell}^n(\pi_{s_1}  S^{\mu^{(1)}}\otimes\ldots \otimes \pi_{s_\ell}  S^{\mu^{(\ell)}} )\cong \pi_\bfs  S^{\mu},
\]
where $\mu=(\mu^{(1)}, \ldots, \mu^{(\ell)})\in \Par^\ell(n)$ with $|\mu^{(i)}|=n_i$ and $n= \sum_i n_i$.
\end{proposition}
\begin{corollary}\label{cor:idemonindu}
For any $\mu=(\mu^{(1)}, \ldots, \mu^{(\ell)})\in \Par^\ell(n)$ with $|\mu^{(i)}|=n_i$ and $n= \sum_i n_i$, we have 
\[ 1_{(1^n)}\Ind_{n_1,\ldots,n_\ell}^n(\pi_{s_1}   \Delta^{\mu^{(1)}}\otimes\ldots \otimes \pi_{s_\ell}   \Delta^{\mu^{(\ell)}} )\cong \Ind_{n_1,\ldots,n_\ell}^n(\pi_{s_1} S^{\mu^{(1)}}\otimes\ldots \otimes \pi_{s_\ell}  S^{\mu^{(\ell)}} ).
\]
\end{corollary}

\begin{proof}
  This follows from  Theorem \ref{thm:induced} and Proposition \ref{prop:induhe} by applying the idempotent functor and \eqref{equ:idemoncell}.  
\end{proof}

Let $\tau$ be the automorphism of $\AHe_n$ given on the generators by 
\[ 
H_i^\tau= -H_i^{-1}, \qquad
X_i^\tau= X_{i}^{-1}. 
\]
 This automorphism induces an exact isomorphism $\tau: \AHe_n\modf\rightarrow \AHe_n\modf$
such that $ M^\tau$ is given by the twist action via $\tau$ and $f^\tau=f$ for any homomorphism  $f: M\rightarrow N$.
Then  we have 
\begin{equation}
\label{equ:antiind}
  \big(\text{Ind}_{n_1,\ldots,n_\ell}^n (N_1\otimes \ldots \otimes N_\ell)\big)^\tau \cong \text{Ind}_{n_1,\ldots,n_\ell}^n (N_1^\tau\otimes \ldots\otimes  N_\ell^\tau).  \end{equation}
Moreover, using the shuffle lemma 
we have 
\begin{equation}
\label{equ:shuffle}
    [\text{Ind}_{m,n}^{m+n} (M\otimes N)]  = [\text{Ind}_{n,m}^{m+n} (N\otimes M)]
\end{equation}
in $K_0(\AHe_{m+n}\modf)$. 

\subsection{Affine $q$-Schur algebras as endmorphism algebras}

In the remainder of this section, we revisit two constructions of affine $q$-Schur algebras, one algebraic and one geometric. 

The extended affine Weyl group of type $GL(n)$ can be expressed as a semi-direct product $\ASym_n =\mathfrak S_n \ltimes \Z^n$; alternatively, $\ASym_n$ is generated by the affine Weyl group $\langle s_0, s_1, \ldots, s_{n-1}\rangle$ and $\pi$ such that $\pi s_{i-1} \pi^{-1} =s_i$. A length function on $\ASym_n$ is given such that $\ell(s_i)=1$ and $\ell(\pi)=0$. 
Given $N\in \Z_{>0}$, the group $\ASym_n$ acts on $\Z^n$ from the right by
\begin{align*} 
    \bfi . \mu &= \bfi+N\mu, \qquad \text{ for }\mu \in \Z^n,
    \\
    \bfi . s_a &=(i_1, \ldots, i_{a-1}, i_{a+1}, i_{a}, \ldots, i_n), \quad \text{ for } a \neq 0,
    \\
     \bfi . s_0 &=(i_1-N, i_2,  \ldots, i_{n-1}, i_n+N).   
\end{align*}
The fundamental domain for this action on $\Z^n$ is given by the alcove
\begin{align*}
    \Xi(N,n) =\{ \bfi \in \Z^n \mid 1-N \le i_1\le i_2 \le \ldots \le i_n \le 0 \}.
\end{align*}
Let $\mathfrak S_\bfi$ denote the stabilizer subgroup of $\bfi \in \Xi(N,n)$. 

The affine Hecke algebra $\AHe(n)$ can be presented with generators $H_i$, for $0\le i \le n-1$, and $\pi$. Then $H_\sigma$, for $\sigma \in \ASym_n$, is well defined, in terms of any reduced expression of $\sigma \in \ASym_n$. The stabilizer $\mathfrak S_\bfi$ of $\bfi \in \Xi(N,n)$ is a parabolic subgroup of $\ASym_n$, and we define 
$e_\bfi =\sum_{\sigma \in \mathfrak S_\bfi} q^{-\ell(\sigma)} H_\sigma$. Then $e_\bfi^2 =(\sum_{\sigma\in \mathfrak S_\bfi} q^{-2\ell(\sigma)}) e_\bfi.$
Consider the right $\AHe(n)$-module 
\begin{align} \label{T:module}
    \hatT_{N,n} := \bigoplus_{\bfi \in \Xi(N,n)} e_\bfi \AHe(n).
\end{align}
By definition (cf. \cite{Gre99}) the affine $q$-Schur algebra is the endomorphism algebra
\begin{align} \label{def:affineSchur}
    \ASch(N,n) :=\End_{\AHe(n)} (\hatT_{N,n}).
\end{align}

\subsection{Geometric affine $q$-Schur algebras}

Let us fix a pair $(N, n)$ of positive integers. 
The variety of $N$-step flags in $\C^n$, $\cF =\{0=F_0\subset F_1\subset \cdots \subset F_{N-1} \subset F_N=\C^n\}$, admits a natural action of the general linear group $GL(n)$ over $\C$.
Denote by $\cB$ the variety of complete flags in $\C^n$, $\cB=\big\{0=V_0\subseteq V_1 \subseteq \cdots \subseteq   V_{n-1} \subseteq V_n= \C^n\mid \dim V_i=i, \forall i\big\}$.

Let $GL(n)$ act diagonally on the products $\cF\times \cF$, $\cF\times \cB$ and $\cB\times \cB$. Denote by $\cN \subset \gl_n$ the nilpotent cone. 
The group $G =GL(n)\times \C^\times$ acts on variants of the Steinberg varieties 
\begin{align*}
    T^*\cF \times_{\cN} T^*\cF, \qquad 
    T^*\cF \times_{\cN} T^*\cB, \qquad
    T^*\cB \times_{\cN} T^*\cB, 
\end{align*}
where $GL(n)$ acts diagonally and $z\in \C^\times$ acts by multiplication by $z^{-2}$ along the fibers. We denote by $\KS(N,n)$, $\KBi_{N,n}$, and  $\KH(n)$ the complexified Grothendieck group of $G$-equivariant coherent sheaves on 
these Steinberg varieties:
\begin{align}
    \label{KSchur}
\begin{split}
\KS(N,n) =  K^G(T^*\cF \times_{\cN} T^*\cF), 
 &  \qquad
\KBi_{N,n} =K^G(T^*\cF \times_{\cN} T^*\cB), 
\\
\KH(n) &=K^G(T^*\cB \times_{\cN} T^*\cB).
\end{split}
\end{align}
Then $(\KS(N,n), \star)$ and $(\KH(n), \star)$ are associative algebras endowed with a convolution product \cite{CG97, GV93}, and $\KBi_{N,n}$ becomes a bimodule under the left and right convolution products:
$\KS(N,n) \circlearrowright \KBi_{N,n} \circlearrowleft \KH(n)$; cf. \cite{GRV94}. 

The following identification is folklore, but no proof was available in the literature until \cite[Theorem 5.21]{LXY26} where it was formulated and proved for arbitrary type. 

\begin{proposition} \cite{LXY26}
\label{prop:SameS}
  We have an algebra isomorphism $\KS(N,n) \cong \ASch(N,n)$, for $N\ge n$.
\end{proposition}

\section{Affine $q$-web, affine $q$-Schur, and Hall algebras}
\label{sec:affineSchur_Hall}

In this section, we show that the isomorphism between the Grothendieck groups of $\qAW \modf_\eps$ and of the module category of K-theoretic affine Hecke algebras preserves the classes of standard modules.

\subsection{Hall algebra of linear and cyclic quivers}

Recall $\I=\Z$ or $\Z/p\Z$ from \eqref{eq:I}. We denote by $\peni$ and $\penp$ the linear and cyclic quiver with arrow sets $\Omega =\{(i,i+1)\mid i\in \I\}$.
Sometimes we allow $p=\infty$ to treat both cases simultaneously. 

Let $\bbF_{\qq^2}$ be the finite field of $\qq^2$ elements. For an $\I$-graded $\bbF_{\qq^2}$-vector space $V=\oplus_{i\in\I} V_i$, denote by $E_V=\oplus_{(i,j)\in\Omega} \Hom(V_i, V_j)$ the space of nilpotent representations of $\pen_p$. The group $G_V =\prod_{i\in \I} GL(V_i)$ acts on $E_V$ by conjugation. Denote by $S_i$ the 1-dimensional simple module supported at $i\in\I$ with dimension vector denoted by $\epsilon_i$. For each pair $(i,\ell)\in \I\times \N$ there exists a unique indecomposable $\pen_p$-module $S_{i,\ell}$ of length $\ell$ and head $S_i$. In this way, each $\pen_p$-module $M$ is isomorphic to a unique
\[
M \cong \bigoplus_{i,\ell} m_{i,\ell} S_{i,\ell}.
\]
Denote by $[i;\ell)$, for $i\in\I, \ell \in\Z$, the segment $[i,i+\ell-1]$ in case $\I=\Z$ and cyclic segment equal to the modulo $p$ projection of $[i_0,i_0+\ell-1]$ for any preimage $i_0\in \Z$ of $i$ in case $\I=\Z/p\Z$. Then the isoclasses of $\pen$-modules are parametrized by the set of multi-segments
\begin{align}
    \label{cMcS}
    \cM\cS =\Big\{\bfm :=\sum_{i\in\I,\ell\in\Z} m_{i,\ell} [i;\ell) \mid m_{i,\ell} \in \N, \text{all but finitely many are }0 \Big\}.
\end{align}
Denote by $|\bfm|=\sum_{i\in\I,\ell\in\Z} \ell m_{i,\ell}$, $\cM\cS_n=\{\bfm\in \cM\cS\mid |\bfm|=n\}$, and so $\cM\cS=\cup_n\cM\cS_n$.

For an $\I$-graded $\bbF_{\qq^2}$-vector space $V=\oplus_{i\in\I} V_i$, let $\C_G(E_V)$ be the set of $\C$-valued $G_V$-equivariant functions on $E_V$. Fix an $\I$-graded $\bbF_{\qq^2}$-vector space $V_{\bfd}$ of dimension $\bfd \in \N^\I$. Given $\bfd', \bfd'' \in \N^\I$ such that $\bfd'+\bfd''=\bfd$, we have the following diagram
\[
E_{V_\bfd'} \times E_{V_\bfd''} \stackrel{p_1}{\longleftarrow} E \stackrel{p_2}{\longrightarrow} F \stackrel{p_3}{\longrightarrow} E_{V_{\bfd}}.
\]
Here $E$ denotes the set of triples $(x,\phi,\psi)$ with $x\in E_{V_\bfd}$ and
\[
0\longrightarrow V_{\bfd'} \stackrel{\phi}{\longrightarrow} V_{\bfd} \stackrel{\psi}{\longrightarrow} V_{\bfd''} \longrightarrow 0
\]
an exact sequence of $\I$-graded vector spaces such that $\phi(V_{\bfd'})$ is stable by $x$, and $F$ denotes the set of pairs $(x,U)$, where $U\subset V_\bfd$ is an $\I$-graded $x$-stable subspace of dimension $\bfd'$. 

The Hall algebra $\bold{Hall}_{p,\bbF_{\qq^2}} :=\oplus_{\bfd} \C_G(E_{V_\bfd})$ is endowed with the following multiplication \cite{Lus91, Lus98}: for $f \in \C_G(E_{V'_\bfd}), g\in \C_G(E_{V''_\bfd})$,
\begin{align}
    \label{product}
    f \circ g =\qq^{-m(\bfd'',\bfd')} (p_3)_! h \in \C_G(E_{V_\bfd}),
\end{align}
where $h \in \C(F)$ is the unique function such that $p_2^*(h) =p_1^*(f,g)$ and $m(\bfd'',\bfd')= \sum_{i\in\I} d_i'd_i'' +\sum_{i\in \I} d_i''d_{i+1}'$. 

For $\bfm \in \cM\cS$ with $\underline{\dim} (\bfm) =\bfd$, we denote by $\cO_\bfm$ the $G_{V_\bfd}$-orbit in $E_{V_\bfd}$ consisting of representations in the isoclass $\bfm$. Let $f_\bfm$ be $\qq^{-\dim \cO_\bfm}$ times the characteristic function of $\cO_\bfm$.

It is known (due to Jin Yun Guo) that the structure constants for the Hall algebras $\bold{Hall}_{p,\bbF_{\qq^2}}$ with respect to the basis $\{f_\bfm\}$ are polynomials in $\qq$. This allows one to define the generic Hall algebra $\Hall$ over $\cA=\Z[v,v^{-1}]$ for an indeterminate $v$ such that $\Hall|_{v\mapsto \qq^{-1}} =\bold{Hall}_{p,\bbF_{\qq^2}}$. We shall use the same notation $f_\bfm$ to denote the corresponding elements in $\Hall$. For $\bfm \in \cM\cS$, we let
\begin{align*}
    b_\bfm =\sum_{i, \bfn} v^{-i+\dim \cO_\bfm -\dim \cO_\bfn} \dim \cH^i_{\cO_\bfn} (\text{IC}_\bfm)f_\bfn,
\end{align*}
where $\cH^i_{\cO_\bfn} (\text{IC}_\bfm)$ is the stalk over a point of $\cO_\bfn$ of the $i$th IC sheaf of the closure $\overline{\cO_\bfm}$ of $\cO_\bfm$. Then $\bfB :=\{b_\bfm\mid \bfm \in \cM\}$ is the canonical basis of $\Hall$. There is a unique bar involution on $\Hall$ such that $\overline{v}=v^{-1}$, and $\overline{b_\bfm}=b_\bfm$, for all $\bfm$; cf. \cite{Lus98, VV99}. 

\begin{proposition}  \cite{VV99, DF15} 
\label{prop:generators}
    The Hall algebra $\Hall$ is generated by $f_\bfi$, for all $\bfi\in \Z\I$. 
\end{proposition}

\subsection{Quantum group $\U_v(\widehat\gl_p)$}

Denote by $\U(\widehat\sll_p)$ the quantum group of type $A_{p-1}^{(1)}$, i.e., the $\C(v)$-algebra generated by $E_i, F_i, K_i^{\pm 1}$, for $i\in \I=\Z/p\Z$. Let $\UA(\widehat\sll_p)$ be the $\cA$-subalgebra generated by the divided powers $E_i^{(r)}, F_i^{(r)}, K_i^{\pm 1}$, for $i\in \I=\Z/p\Z$ and $r\ge 1$.
There is an embedding $\UA(\widehat\sll_p) \rightarrow \Hall$ such that $F_i^{(r)} \rightarrow f_{r[i]}$, for all admissible $i, r$. 

We shall need to enlarge the $\U_v(\widehat\sll_p)$ to the $\C(v)$-algebra 
\begin{align} \label{eq:affine_glp}
\U_v(\widehat\gl_p):=\U_v(\widehat\sll_p) \cdot \Heis,
\end{align}
where the Heisenberg algebra $\Heis$, which  commutes with $\U_v(\widehat\sll_p)$, is the $\C(v)$-algebra generated by $B_k$, for $k\in \Z^\times$, subject to the relations \cite{KMS95, LT96, Ugl00} 
\[
[B_k, B_m]= k\delta_{k,-m} \frac{[kp]_v[k\ell]_v}{[k]_v^2},
\]
We write
$\Heis =\Heis^- \otimes \Heis^+$, where $\Heis^+$ and $\Heis^-$ are commutative subalgebras of $\Heis$ generated by $B_r$ for $\pm r>0$, respectively. Denote $\U^-_v(\widehat\gl_p) = \U^-_v(\widehat\sll_p) \cdot\Heis^-$.

There is a $\C(v)$-algebra isomorphism (cf. \cite{X97, Sch00, DF15})
\begin{align}
    \label{eq:isom}
    \xi: \C(v) \otimes_\cA\Hall \stackrel{\cong}{\longrightarrow}  \U_v^-(\widehat\gl_p).
\end{align}
We then define an $\cA$-form $\U_\cA^-(\widehat\gl_p)$ as the image of $\Hall$ under the isomorphism $\xi$:
\[
\xi: \Hall \stackrel{\cong}{\longrightarrow} \U_\cA^-(\widehat\gl_p), 
\]
and it inherits a canonical basis from $\Hall$ this way. We have $\U_\cA^-(\widehat\gl_p)\supset \U_\cA^-(\widehat\sll_p)$ with compatible canonical bases since Lusztig \cite{Lus90, Lus98} defined the canonical basis of $\U_\cA^-(\widehat\sll_p)$ geometrically by identifying it with the composition subalgebra of $\Hall$.

\subsection{Connecting $\Hall$ to $\K^*$}

Recall the geometric (K-theoretic) construction of the affine $q$-Schur algebras from \eqref{KSchur}, which we shall set $N=n$ from now on. We review some well known constructions in geometric representation theory; cf. \cite{CG97, GV93, Lus91, Lus98, AJL}. The simple modules of the convolution algebra $(\KS(N,n), \star)$ are parametrized by $GL(n)$-orbits of pairs $(s,x) \in GL(n) \times \gl_n$, where $s$ is semisimple, $x$ is nilpotent, and $sxs^{-1} =z^{-2}x$. The $GL(n)$-orbits of the nilpotent elements $x$ are parametrized by the Jordan forms or partitions $\lambda \in \Par(n)$. The $GL(n)$-orbits of pairs $(s,x)$, where the eigenvalues of $s$ are assumed to be in $z^{2\Z}$, are in bijection with the isoclasses of the $n$-dimensional nilpotent representations of $\peni$ if $z$ is generic and of $\penp$ if $z =\varepsilon$ (where $\varepsilon^2$ is a primitive $p$th root of $1$); hence the orbits $O_\bfm$ are naturally parameterized by multi-segments $\bfm\in\cM\cS$;  cf. \eqref{cMcS}.  

Associated to each such orbit $O_\bfm$, there exists a standard module $M_\bfm$ of the convolution algebra $(\KS(n,n), \star)$ defined as Borel-Moore homology of the fixed point locus of a Springer fiber $H_*(\cF_x^s)$ \cite{CG97, GV93}. Denote by $[\KS(n,n)\modf]$ the Grothendieck group of finite-dimensional $\KS(n,n)$-modules over $\C$. Then $[\KS(n,n)\modf]$ admits two bases: 
\[
\{[M_\bfm] \mid \bfm\in \cM\cS_n\},
\quad \text{ and } \quad
\{[L_\bfm] \mid \bfm\in \cM\cS_n\}, 
\]
with uni-triangular transition matrices between them; here $L_\bfm$ denotes the simple modules. Denote by $[\KS(n,n)\modf]^*$ the restricted dual, which is the $\Z$-span of the dual basis $[L_\bfm]^*$, for $\bfm\in\cM\cS$, where $[L_\bfm]^* ([L_{\bfm'}]) =\delta_{\bfm,\bfm'}$. Similarly, we denote by $\{[M_\bfm]^*\}$ the $\Z$-basis dual to $\{[M_{\bfm}]\}$. We denote 
\begin{align} \label{eq:KK}
{\K} :=\bigoplus_{n=0}^\infty [\KS(n,n)\modf],
\qquad
{\K}^* :=\bigoplus_{n=0}^\infty [\KS(n,n)\modf]^*.
\end{align}

\begin{theorem} [\text{\cite[Theorems 6.6, 8.4]{GV93}, \cite[\S12.3]{VV99}}]
\label{thm:Hall=Schursimple}
Let $p\le \infty$. 
    The $\Z$-linear isomorphism 
    \begin{align*}
        \imath_p: \HallZ \longrightarrow {\K}^*,  \qquad
        f_{\bfm} \mapsto [M_{\bfm}]^*
        \textrm{ for } \bfm \in \cM\cS,
    \end{align*}
sends the canonical basis to the basis dual to the simple modules, i.e., $\imath_p(b_\bfm) =[L_{\bfm}]^*.$
\end{theorem}

\subsection{Matching the standard modules}

We now bring back the affine $q$-webs into the consideration. 

\begin{proposition}
\label{prop:Morita}
    For any $N\ge n$, the algebras $\KS(N,n), \ASch(N,n)$ and $\qAW(n)$ are all Morita equivalent. (We only need the case when $N=n$.)
\end{proposition}

\begin{proof}
Thanks to $N\ge n$, each summand in $\hatT_{n,n}$ appears as a summand in $\hatT_{N,n}$; cf. \eqref{T:module}. Hence by the definition \eqref{def:affineSchur}, the affine $q$-Schur algebra $\ASch(N,n)$ is Morita equivalent to $\ASch(n,n)$. It follows by Proposition \ref{prop:SameS} that $\KS(N,n)\cong \ASch(N,n)$ is Morita equivalent to $\KS(n,n) \cong \ASch(n,n)$.

On the other hand, the algebra $\ASch(n,n)$ is Morita equivalent to $\qAW(n)$ by \cite[Proposition 4.18]{SSW25}. We are done. 
\end{proof}

\begin{rem}
\label{rem:composition}
    One can define a variant $'\!\qAW$ of $\qAW$ whose objects are compositions instead of strict compositions; cf. \cite{Bru25} in the setting of the $q$-web category. Then the path algebra for $'\!\qAW(n)$ is isomorphic to $\ASch(n,n)$, and hence isomorphic to $\KS(n,n)$. 
\end{rem}

\begin{definition} 
A finite dimensional $\qAW$-module $M$ is called $\eps$-integral  if it factors through the quotient $\qAW\rightarrow\qW_\bfu$ with $\bfu=(\eps^{2s_1}, \ldots, \eps^{2s_\ell})$, for some $s_i\in \I$. Let $\qAW\modf_\eps$ denote the category of $\eps$-integral $\qAW$-modules.  Categories $\hwS(n)\modf_\eps$, $\qAH\modf_\eps$, and $\AHe(n)\modf_\eps$ are defined similarly.
\end{definition}

The study of general finite-dimensional $\hwS(n)$-modules can be reduced through Morita equivalences to 
$\hwS(m)\modf_\eps$, for various $\eps$ and $m$. 

For any $\bfm=\sum_{i\in \I,\ell\in \Z}m_{i,\ell} [i;\ell) $, let $\lambda_{\bfm}$ be the multipartition such that each part is a column partition and for each $i$ there are $m_{i,\ell}$ times of  $(1^{\ell})$ for $i\in \I,\ell\in \Z $. Here we fix some order of the component of $\lambda_\bfm$. 
Set $\bfu^\diamond=((\eps^{2(i-\ell+1)})^{m_{i,\ell}})_{i\in \I,\ell\in \Z}$ such that the order corresponds to the order of component of $\lambda_\bfm$. Denoting $n=|\lambda_{\bfm}| =|\bfm|$, we define the {\em standard modules} of $\hwS(n)$:
 \begin{align}  \label{eq:std}
\Delta_\bfm:= \pi_{\bfu^\diamond}(\Delta^{\lambda_\bfm}).
\end{align}
These are always induced modules by Theorem \ref{thm:induced}. 
Similarly, the $\AHe(n)$-modules  
\[
\texttt S_\bfm:= \pi_{\bfu^\diamond}( S^{\lambda_\bfm})
\]
are induced modules. 

The Morita equivalence in Proposition \ref{prop:Morita} (and Remark \ref{rem:composition}) induces a canonical isomorphism 
\begin{equation}
\label{equ:Moritiso}
    \natural: [\qAW \modf_\eps]\longrightarrow  \K,
\end{equation}
which matches the simples. 
Similarly, we have an isomorphism $\natural: [\qAH\modf_\eps] \rightarrow\bigoplus_{n=0}^\infty[\KH(n)\modf]$. The following is built on the induction theorem from \cite{KL87} and \cite{Ar96}.

\begin{theorem} \label{thm:ind=std}
Let $p\le \infty$. The isomorphism $\natural$ in \eqref{equ:Moritiso} matches the standard modules, i.e., 
$\natural([\Delta_\bfm] )= [M_\bfm]$, for any $\bfm\in\cM\cS$. ($\natural$ matches the simples as well.)
\end{theorem}

\begin{proof}
First assume that $p=\infty$. In this case we have a commutative diagram 
\begin{equation*}  
\begin{tikzcd}
{[}\hwS(n)\modf_\eps{]}\ar[r,"\natural"]\ar[d,"1_{(1^n)}"] & {[}\KS(n,n)\modf{]} \ar[d,"1_{(1^n)}"]
\\
{[}\AHe(n)\modf_\eps{]} \ar[r,"\natural"] & {[}\KH(n)\modf{]}
\end{tikzcd} 
\end{equation*}
where the Schur functors $1_{(1^n)}$ are  isomorphisms. By Corollary \ref{cor:idemonindu}, the left-column Schur functor satisfies that
$1_{(1^n)} ([\Delta_\bfm]) = [\texttt S_\bfm]$, while the right-column Schur functor satisfies that $1_{(1^n)} ([M_\bfm]) = [N_\bfm]$; cf. \cite[Proposition 18]{V98}. Here $N_\bfm$ is the standard $\KH(n)$-module defined as in \cite[\S 6.2]{AJL}, which is related to the standard module used in \cite{Ar96} via an involution; also see Remarks \ref{rem:sign}--\ref{rem:opposite}.  Then we have $\natural([\Delta_\bfm] )= [M_\bfm]$ since $\natural([ {\texttt S}_\bfm] )= [N_\bfm]$ by Proposition \ref{prop:induhe} and \cite[Theorem 3.2]{Ar96} (also see \cite[Theorem~ 6.2]{KL87}).
Finally, the identity in case for $p>0$ follows from that for $p=\infty$ since $\Delta_\bfm$ and $M_\bfm$ are defined over any field and the equality in the Grothdendick group holds in specializations of $q\mapsto \eps$.
 \end{proof}
 
 \begin{rem}
 \label{rem:donotorder}
 It follows from Theorem \ref{thm:ind=std} that $[\Delta_\bfm]$ does not depend on the order of components in $\lambda_\bfm$ in the Grothendieck group since $[M_\bfm]$ does not. 
 \end{rem}
 
\begin{rem} \label{rem:sign}
The induced modules ${\texttt S}_\bfm$ here are induced from the sign representation of the parabolic subalgebra, while the induced modules used in \cite{Ar96} is from the trivial representation. Let ${\texttt S}'_\bfm := \pi_{\bfu'}(  S^{\lambda_\bfm'})$, where  $\bfu'=((q^{2i})^{m_{i,\ell}})_{i\in \I,\ell\in \Z}$ and $\lambda_\bfm'$ is the multipartition such that each component is a row partition and for each $i$ there are $m_{i,\ell}$ times of  $(\ell)$ for $i\in \I,\ell\in \Z$.
 Here, the order of the components of $\lambda_\bfm'$ is chosen to be compatible with that of $\bfu'$. 
 Then by 
\eqref{equ:antiind}, we have 
\begin{equation}
\label{equ:signtriv}
   {\texttt S}_\bfm^\tau\cong  {\texttt S}'_{\bfm^\tau}. 
\end{equation}
where $\bfm^\tau=\sum_{i\in \I, \ell\in \N}m_{i,\ell} (\ell, -i]$.
See also \cite[\S 6.2]{AJL}.
Note that  by \eqref{equ:shuffle}, the isoclasses of $\texttt S'_{\bfm^\tau}$ and $\texttt S_\bfm$ do not depend on the ordering of the tuple $\lambda_\bfm$ and $\lambda_{\bfm^\tau}'$. Then $\texttt [S_{\bfm }]$ also corresponds to $f_\bfm$   since $\tau$ corresponds to the anti-automorphism  $\rho$ of the Hall algebra in \cite[\S 2.3]{AJL} and $\rho (f_\bfm)=f_{\bfm^\tau}$. 
\end{rem}

\begin{rem}
\label{rem:opposite}
The discrepancy between the induced representations from the sign and trivial representations noted in Remark~\ref{rem:sign} has its root in the choice of opposite orientations on the cyclic quiver. In this paper we use the clockwise orientation ($i\mapsto i+1$) following \cite{VV99, DF15}. If we choose to work with the Hall algebra with opposite orientation as in \cite{Ar96, LTV99} and define $f_\bfi$ via  {\em right} multiplication (with multiplication formula derived from \cite{DF15, FL19} accordingly), then we expect the new formulation will match with the induced modules from the trivial representation. 
\end{rem}

\subsection{Affine $q$-web categorifies the Hall algebra}

The $q$-web category has a great advantage over $\K$ as restriction and induction functors are defined. We define linear operators 
\begin{align}
    \label{euq:defoffi}
    \begin{split}
    f_\bfi: [\qAW \modf_\eps]^* &\longrightarrow [\qAW \modf_\eps]^*, \\
    \phi \mapsto f_\bfi \phi &\qquad \text{ such that }
    f_\bfi \phi ([M]):= \phi ([\bfi\text{-Res}(M)]).
    \end{split}
\end{align}
Let $D_\bfm$ be the irreducible module in $[\qAW \modf_\eps]$ such that  
$\natural ([D_\bfm])= L_\bfm$, where $\natural$ is the isomorphism in \eqref{equ:Moritiso}. Denote $\HallZ =\Z\otimes_\cA \Hall$, the specialization at $v=1$. (Similar notations with subscript $\Z$, e.g., $\Uglzhalf$, will be freely used for the specialization at $v=1$ of $\cA$-algebras and $\cA$-modules.) 

The identifications in Proposition \ref{prop:Morita} and Theorem \ref{thm:ind=std} allow us to formulate an improved version of Theorem \ref {thm:Hall=Schursimple} in terms of the affine $q$-web category. It follows from Theorem \ref{thm:ind=std} that  
$\{[\Delta_{\bfm}]~|~ \bfm\in \cM\cS\}$ is a basis of $[\qAW \modf_\eps]$. Denote by $\{[\Delta_{\bfm}]^*\}$ the basis in $[\qAW \modf_\eps]^*$ dual to the basis $\{[\Delta_{\bfm}]\}$. 

 \begin{theorem}
 \label{thm:Cat_affine}
 Let $p\le \infty$.
 \begin{enumerate}
     \item 
     There exists a $\Uglzhalf$-module isomorphism
\begin{align}
\label{eq:simple=dCB:p}
\jmath_p: \HallZ &\longrightarrow [\qAW \modf_\eps]^*,
\qquad
f_\bfm \mapsto [\Delta_{\bfm}]^*,
\end{align}
where $f_\bfi$ acts on $\Hall$ by the multiplication with the corresponding semisimple generator and 
acts on $[\qAW \modf_\eps]^*$ by \eqref{euq:defoffi}.
\item The isomorphism $\jmath_p$ sends the canonical basis to the dual basis of the simple modules, i.e., $\jmath_p(b_\bfm)=[D_\bfm]^*$.
 \end{enumerate}
\end{theorem}
 
\begin{proof}
(1) Clearly $\jmath_p$ defined by \eqref{eq:simple=dCB:p} is a $\Z$-linear isomorphism as it matches bases. Then it remains to show that 
    \begin{equation}
    \label{equ:isomodule}
        \jmath_p(f_{\mathbf i}f_\bfm) = f_{\mathbf i} ([\Delta_{\bfm}]^*)
    \end{equation} 
for all semisimple generators $ f_{\mathbf i}$ and all multi-segments $\bfm$.

We assume $p=\infty$ first.
Write $f_{\mathbf i}f_\bfm=\sum c_{\bfm',\bfm}f_\bfm'$
and $f_{\mathbf i} [\Delta_{\bfm}]^* =\sum d_{\bfm',\bfm}  [\Delta_{\bfm'}]^*$.
Write a multi-segment as $\bfm_A=\sum_{i<j\in \Z}a_{i,j}[i,j)$ for some  upper integral  triangular  matrix $A=(a_{ij})$. Let $\mathbf T_{\bfi}$
be the set of all upper triangular integral matrices $T$ such that $\sum_{j}t_{kj}=i_k$ for all $k$. Then by \cite[Proposition 4.2]{DF15}
we have 
\[
f_\bfi f_{\bfm_A}= \sum_{ T\in \mathbf T_\bfi} \prod_{i,j\in\Z} \binom{a_{ij}+t_{ij}-t_{i-1,j}}{t_{ij}} f_{\bfm_{A+T-\tilde T^+} }   
\]
where 
$ \tilde T$ is obtained by descending each row of $T$  to the next row and $\tilde T^+$ is the upper triangular part of $\tilde T$.

On the other hand, $d_{\bfm',\bfm} $ is determined by the $\mathbf i$-Res such that 
\begin{equation}
\label{equ:i-RES}
  [\iRes(\Delta_{\bfm'})]= \sum d_{\bfm',\bfm}[\Delta_{\bfm}].  
\end{equation}
Using Proposition \ref{prop:IndRes} and comparing the coefficients (where induced modules in the Grothdendick group for different orderings are identified, see Remark \ref{rem:donotorder}) we  see that 
\[d_{\bfm_{A+T-\tilde T^+},\bfm_A}= \prod_{i,j\in\Z} \binom{a_{ij}+t_{ij}-t_{i-1,j}}{t_{ij}}.\]
For example, if $\bfi=\lambda_i\alpha_i$, then by Proposition \ref{prop:IndRes} and Remark \ref{rem:donotorder}, 
any $\bfm'$ appearing in \eqref{equ:i-RES} is obtained from $\bfm=\sum_{i,j}a_{i,j}[i,j)$ by deleting $\lambda_i$ times of $i$ from the segments starting with $i$, i.e., 
$\bfm'=m_{A+T-\tilde T^+}$ for some $T\in \mathbf T_\bfi$, if $m=m_A$.
Furthermore, we have 
\[d_{\bfm',\bfm}=\prod_{j}\binom{a'_{ij}}{t_{ij}} =\prod_{j\in \Z} \binom{a_{ij}+t_{ij}-t_{i-1,j}}{t_{ij}}=c_{\bfm',\bfm}.\]
For general $\bfi=\sum_{i}\lambda_i\alpha_i$, the equality $d_{\bfm',\bfm}=c_{\bfm',\bfm}$ follows from the case $\bfi=\lambda_i\alpha_i$ and the induction since 
$f_\bfi=\prod_{i}f_{\lambda_i\alpha_i}$ by \cite[Remark 6.1]{VV99}, where the order of the product is from $\infty$ to $-\infty$. This proves \eqref{equ:isomodule} for  $p=\infty$. 

Below we shall prove the result for $p>0$ by using the result proved above for $p=\infty$. 




For $p>0$, we let $\bfW(n)$ be the $\Z$-submodule of 
$ [\qAW(n)\text{-mod}_\eps]$ spanned by $\{[\Delta_\bfm] \mid \bfm\in \cM\cS_n\}$. Let $\bfW=\oplus_{n\in \N}\bfW(n)$ and let $\bfW^*=\oplus_{n\in \N}\Hom(\bfW(n),\C)$ be the graded dual.
For $p=\infty$, we add a subscript $\infty$ to denote the counterparts by $\bfW_\infty$ and $ \bfW_\infty^*$. 

We have shown above that $\bfW_\infty =[\qAW \modf_\eps]$ and that $\{[\Delta_\bfm]\mid \bfm\in\cM\cS_\infty\}$ is a basis of $\bfW_\infty$.
Clearly we have a surjective $\Z$-linear  map 
\[
\varpi: \bfW_\infty \longrightarrow \bfW,
\qquad
\varpi([\Delta_\bfm])=[\Delta_{\bar\bfm}], 
\]
where $\bar\bfm$ is the multi-segment obtaned from $\bfm$ modulo $p$. Dualizing this gives us an injective $\Z$-linear map 
\begin{align*}
  \varpi_*:  \bfW^* &\longrightarrow \bfW_\infty^*, 
  \qquad
    \phi \mapsto \phi \circ \varpi. 
\end{align*}

 Thanks to Proposition \ref{prop:IndRes} and Theorem \ref{thm:induced}, we have $\varpi_*(f_\bfi \phi) ([\Delta_\bfm]) =\sum _{\bfj:\bar\bfj=\bfi } f_\bfj \varpi_* (\phi) ([\Delta_\bfm])$ for $[\Delta_\bfm] \in \bfW_\infty$,  i.e., $\bfi\text{-Res}$ on $\bfW^*$ corresponds to  $\oplus_{ \bar\bfj=\bfi } \bfj\text{-Res}$ on $\bfW^*_\infty$ under the injection $ \varpi_*$.   
That is, we have a commutative diagram 
\begin{equation}  \label{CD:WWp}
\begin{tikzcd}
\bfW^* \ar[r,"\varpi_*"]\ar[d,"f_\bfi"] &\bfW^*_\infty= \bold{Hall}_{\infty,\Z} \ar[d,"\sum\limits_{\bfj:\bar \bfj=\bfi} f_\bfj"]
\\
\bfW^* \ar[r,"\varpi_*"] &\bfW^*_\infty= \bold{Hall}_{\infty,\Z}
\end{tikzcd} 
\end{equation}
where we identify $ \bfW^*_\infty =\bold{Hall}_{\infty,\Z}$ by what we have proved above for $p=\infty$. 
We view $\HallZ$ as a subspace of $\bold{Hall}_{\infty,\Z}$ via the injection from $\HallZ$ to $\bold{Hall}_{\infty,\Z}$ (i.e., from $\U_{p,\Z}^-$ to $\U_{\infty,\Z}^-$) which sends $f_\bfi$ to $\sum_{\bar \bfj=\bfi} f_\bfj$. Applying the generators $f_\bfi$ repeatedly to the cyclic vector $1\in \bfW^*$ gives us $\HallZ\subset \varpi_* ( \bfW^*)$. By comparing the graded ranks of $\HallZ$ and $ \bfW^*$ (for each $n$th component) we see that $\HallZ = \varpi_*(\bfW^*)$. Plugging this back into the diagram \eqref{CD:WWp} proves \eqref{equ:isomodule} for $p>0$, whence (1).

(2) Denote by $\natural^*$ the dual of the  isomorphism  $\natural$ in \eqref{equ:Moritiso}. By Theorem \ref{thm:Hall=Schursimple} and Part (1), we have $\natural^* \imath_p (f_\bfm) = \natural^*([M_\bfm]^*) =[\Delta_\bfm]^* = \jmath_p(f_\bfm)$, i.e., the
following commutative diagram holds:
\begin{equation*}  
\begin{tikzcd}
\HallZ\ar[r,"\imath_p"]\ar[dr,"\jmath_p"] &\K^*\ar[d,"\natural^*"]
\\
 &\bfW^*
\end{tikzcd} 
\end{equation*}
Then Part (2) follows from Theorems~ \ref{thm:Hall=Schursimple} and \ref{thm:ind=std}.
\end{proof}
\section{Categorification of integrable highest weight $\U({\widehat{\gl}_p})$-modules}
\label{sec:Cat}

\subsection{Higher level Fock spaces}

Denote by $\Par$ the set of all partitions. Given $\ell\ge 1$ and $s\in \Z$, let 
\begin{align} \label{eq:Zls}
\Z^\ell(s) =\big\{\bfs =(s_1, \ldots, s_\ell) \in \Z^\ell \mid s_1+ \ldots+ s_\ell =s \big\}.
\end{align}
One has a Fock space $\cF^{s+\infty}$ of level $\ell, p$ and charge $s$, which is equipped with a natural basis $\{|\bfla, \bfs \rangle \}$, where $\bfla =(\lambda^{(1)}, \ldots, \lambda^{(\ell)}) \in \Par^\ell$ and $\bfs \in \Z^\ell(s)$. It admits an action of $\U_v(\widehat\gl_p)$ from \eqref{eq:affine_glp} (and also a commuting action of $\U_v(\widehat\sll_\ell)$ which will be ignored here), cf. \cite{JMMO91, Ugl00}. 

We have a decomposition as $\U(\widehat\gl_p)$-modules
\[
\cF^{s+\infty} = \bigoplus_{\bfs \in \Z^\ell(s)} \cF_\bfs, 
\]
where $\cF_\bfs$ is spanned by $|\bfla, \bfs \rangle$, for all $\bfla \in \Par^\ell$. 

Uglov \cite{Ugl00} defined a bar involution on $\cF_\bfs$, for each $\bfs$ (and hence on $\cF^{s+\infty}$) such that
$\overline{v}=v^{-1}, \overline{|\emptyset, \bfs \rangle} =|\emptyset, \bfs \rangle$, $\overline{ux} =\overline{u}\,\overline{x}$, and $\overline{B_{-r}x} =\overline{B_{-r}}\overline{x}$, for all $x\in \cF_\bfs$, $u \in \U_v(\widehat\sll_p)$ and $r>0$. 

\begin{theorem} \cite{Ugl00}
Let $\bfs\in \Z^\ell(s)$ for $s\in \Z$. 
Then there exists a canonical basis $\bfB_{\bfs} =\{b_{\bfla,\bfs} \mid \bfla \in \Par^\ell \}$ 
on $\cF_\bfs$ which are characterized by 
\begin{enumerate}
    \item 
$\overline{b_{\bfla,\bfs}} =b_{\bfla,\bfs}$,
    \item
$b_{\bfla,\bfs} \in |\bfla, \bfs \rangle + \sum\limits_{\bfmu \in \Par^\ell,\bfmu \lhd \bfla} v\Z[v] |\bfmu, \bfs \rangle$.
\end{enumerate}
\end{theorem}

Let $\omega_k$ be the $k$th fundamental weight, and denote $\omega_\bfs =\omega_{s_1} +\ldots +\omega_{s_\ell}$.
Denote by $\Lambda (\omega_\bfs)$ the integrable $\U(\widehat{\mathfrak {gl}}_p)$-module of highest weight $\omega_\bfs$ with highest weight vector $v^+_{\omega_\bfs}$. Denote by $\Lambda (\omega_\bfs)\!_\cA :=\UA^-(\widehat{\mathfrak {gl}}_p)\,  v^+_{\omega_\bfs}$ the $\cA$-form and by $\Lambda (\omega_\bfs)_\Z$ the specialization at $v=1$. 
Recall the canonical basis $\bfB$ on the generic Hall algebra $\Hall$ and the $\U^-_\cA(\widehat{\mathfrak {gl}}_p)$-module surjection $\varpi_\bfs: \Hall \rightarrow \Lambda(\omega_\bfs)\!_\cA$ thanks to the identification $\U^-_\cA(\widehat\gl_p) \cong \Hall$; see \eqref{eq:isom}.

Denote by $\mathcal C (\omega_\bfs)$ the subgraph for $\Lambda (\omega_\bfs)$ of the crystal graph for $\cF_\bfs$ with vertex set in $\Par^\ell$; cf. \cite{Kas91, Sch00}. 

\begin{theorem} \cite[Theorem 4.2]{Sch00}
\label{thm:Sch00}
    We have $\bfB \cdot |\emptyset, \bfs \rangle \subset \bfB_{\bfs} \cup \{0\}.$ Moreover, for $\bfla\in \mathcal C (\omega_\bfs)$, we have $b_{\bfla,\bfs} \in |\bfla, \bfs \rangle + \sum\limits_{\bfmu \in \Par^\ell, \bfmu\lhd\bfla} v\N[v] |\bfmu, \bfs \rangle$.
\end{theorem}
Thus $\bfB \cdot |\emptyset, \bfs \rangle \big\backslash \{0\}$ forms a $\C(v)$-basis for $\Lambda (\omega_\bfs)$ and an $\cA$-basis for $\Lambda (\omega_\bfs)_\cA$ (called a canonical basis). For $\bfla\in \mathcal C (\omega_\bfs)$, we write
\begin{align} \label{eq:decomposition}
    b_{\bfla,\bfs} = \sum_{\bfmu \unlhd \bfla} d_{\bfmu,\bfla}(v) |\bfmu, \bfs \rangle,
\end{align}
where $d_{\bfla,\bfla}(v)=1$, and $d_{\bfmu,\bfla}(v) \in v\N[v]$, for $\bfmu\lhd\bfla$. Also set $d_{\bfmu,\bfla}(v)=0$, for $\bfmu\ntrianglelefteq\bfla$. We refer to $\bold{D}(v) =\big[d_{\bfu,\bfla}(v)\big]_{\bfmu,\bfla}$ or $\bold{D}_n(v) =\big[d_{\bfmu,\bfla}(v)\big]_{|\bfmu|=\bfla|=n}$ as the $v$-decomposition matrix for $\qW_{\eps^{2\bfs}}$ or for $\wS_{\eps^{2\bfs}}(n)$.

\subsection{Cyclotomic categorification}

We continue to use the notation $\bfs \in \Z^\ell$ from \eqref{eq:Zls}, which differs from the notation used in $\eps^{2\bfs}$ (with $\bfs \in \I^\ell$) in earlier sections, but $\eps^{2\bfs}$ still makes sense. 
Thanks to Proposition \ref{prop:IndRes} we have the following. 
  
\begin{proposition} 
  \label{prop:lieaction}
There is a $\U_\Z(\widehat{\mathfrak {gl}}_p)$-module structure on $[\qW_{\eps^{2\bfs}}\modf]$ such that $e_{\bfi}$ and $f_{\bfi}$ act via  $[\iRes]$ and $[\iInd]$, respectively.
In particular, $[\Delta^\emptyset]$ is of weight $\omega_\bfs$.  This induces a $\U_\Z(\widehat{\mathfrak {gl}}_p)$-module structure on $[\qW_{\eps^{2\bfs}}\modf]^*$ such that 
$\phi\mapsto e_\bfi\phi, f_\bfi\phi$ with
\[ 
(e_\bfi \phi) ([M])=\phi ([\iInd(M)]), \qquad
(f_\bfi\phi) ([M])=\phi([\iRes(M)]). 
\]
\end{proposition}

\begin{proof}
    The second statement follows from the first one. Identifying $[\Delta^\lambda]$ with the Schur basis of the higher level Fock space $\cF_{\bfs,\Z}$ (which can be assumed to be a tensor product without loss of generality), the first statement follows by comparing  Proposition \ref{prop:IndRes} and the actions of $e_\bfi$ and $f_\bfi$ on the Fock space  and applying the double Hall construction; cf. \cite[\S 6.2]{VV99} and \cite{X97}.
\end{proof}

\begin{theorem}  \label{thm:cyc_cat}
\begin{enumerate}
    \item 
There exists an isomorphism of $\U_\Z(\widehat{\mathfrak {gl}}_p)$-modules
\[
\jmath_\bfs: \Lambda (\omega_\bfs)_\Z   \stackrel{\cong}{\longrightarrow}  [\qW_{\eps^{2\bfs}} \modf]^*, \qquad v^+_{\omega_\bfs}\mapsto[\Delta^\emptyset]^*.
\]
\item
We have the following commutative diagram of $\Uglzhalf$-module homomorphisms: 
\begin{equation}  
\label{CD:cat}
\begin{tikzcd}
\HallZ \ar[r,"\jmath_p"]\ar[d,two heads,"\varpi_\bfs"]&{[}\qAW \modf_\eps{]}^* \ar[d,two heads,"\pi_\bfs^*"]\\
\Lambda (\omega_\bfs)_\Z  \ar[r,"\jmath_\bfs"]&{[}\qW_{\eps^{2\bfs}} \modf{]}^* 
\end{tikzcd} 
\end{equation}
\item
 The isomorphism $\jmath_\bfs$ sends the canonical basis to the basis dual to the classes of simple modules. 
 \end{enumerate}
\end{theorem}

\begin{proof}
Let us first explain the maps in the following diagram:
\begin{equation}  
\label{CD:modules}
\begin{tikzcd}
\HallZ \ar[r,two heads,"\varpi_\bfs"]\ar[d,"\jmath_p"]& \Lambda (\omega_\bfs)_\Z \ar[r,hook] \ar[d, dashed,"\jmath_\bfs"] &\cF_{\bfs,\Z} \\
{[}\qAW \modf_\eps{]}^* \ar[r,two heads,"\pi_\bfs^*"]&  {[}\qW_{\eps^{2\bfs}} \modf{]}^*  \ar[ur,"\psi^*"]
\end{tikzcd} 
\end{equation}
The homomorphism $\jmath_p$ was given in Proposition \ref{prop:lieaction}, $\varpi_\bfs$ is the natural surjective homomorphism by the identification $\HallZ =\Uglzhalf$, $\pi_\bfs^*$ is the natural surjective homomorphism induced from the embedding $\pi_\bfs: [\qW_{\eps^{2\bfs}} \modf] \rightarrow [\qAW \modf_\eps]$, and $\Lambda (\omega_\bfs)_\Z \hookrightarrow\cF_\bfs$ is the natural $\U(\widehat{\mathfrak {gl}}_p)_\Z$-module embedding sending $v^+_{\omega_\bfs}\mapsto |\emptyset,\bfs\rangle$. Note that the $\U(\widehat{\mathfrak {gl}}_p)_\Z$-action on $\cF_\bfs$ is interpreted in \cite{DX17} as the double Hall algebra action, extending the identification of the $\U^-(\widehat{\mathfrak {gl}}_p)_\Z$-action with $\Hall$-action on $\cF_\bfs$.

We identify $\cF_{\bfs,\Z}$ with the Grothendieck group of the semisimple category $\qW_{t^\bullet\eps^{2\bfs}}\modf$ for a semisimple parameter of the form $t^\bullet\eps^{2\bfs}$ deforming $\eps^{2\bfs}$ (see \cite[(4.5)]{Ar96} and Proposition \ref{prop:semisimple}), and there is a natural surjection $\psi:\cF_\bfs \rightarrow [\qW_{\eps^{2\bfs}} \modf]$, which sends the Schur basis elements $|\bfla,\bfs\rangle$ (identified with Specht modules) to $[\Delta^{\bfla}]$, for $\bfla\in \Par^\ell$; this is a $\U(\widehat{\mathfrak {gl}}_p)$-module homomorphism as they commute with all $f_\bfi$ and $e_\bfi$ (arising from restriction and induction functors). Dualizing $\psi$ and noting that $\cF_{\bfs,\Z}$ is self-dual with orthnormal Schur basis, we obtain the injective homomorphism $\psi^*: [\qW_{\eps^{2\bfs}} \modf]_\C^*\rightarrow \cF_{\bfs,\C}$ sending $[\Delta^\emptyset]^* \mapsto |\emptyset,\bfs\rangle$.

It remains to construct the isomorphism $\jmath_\bfs$ in order to complete the commutative diagram \eqref{CD:modules}. As a $\Uglzhalf$-module, $\HallZ$ is cyclic and so is ${[}\qW_{\eps^{2\bfs}} \modf{]}^*$ with $[\Delta^\emptyset]^*$ as a cyclic vector. Passing through a base change from $\Z$ to $\C$, $\cF_{\bfs,\C}$ is a semisimple $\U(\widehat{\mathfrak {gl}}_p)$-module, and
hence as its cyclic highest weight submodule $[\qW_{\eps^{2\bfs}} \modf]_\C^*$ must be a simple $\U(\widehat{\mathfrak {gl}}_p)$-module of highest weight $\omega_\bfs$. Therefore, we obtain an isomorphism of $\U(\widehat{\mathfrak {gl}}_p)$-modules $\jmath_\bfs: \Lambda (\omega_\bfs)   \rightarrow  [\qW_{\eps^{2\bfs}} \modf]^*_\C,$ $v^+_{\omega_\bfs} \mapsto[\Delta^\emptyset]^*$, completing the commutative diagram \eqref{CD:modules}. This proves Parts (1)--(2).

It remains to prove (3). The embedding $\pi_\bfs: [\qW_{\eps^{2\bfs}} \modf] \rightarrow [\qAW \modf_\eps]$ sends simples to simples. By Schiffman's Theorem \ref{thm:Sch00}, $\varpi_\bfs: \HallZ \rightarrow \Lambda(\omega_\bfs)_\Z$ sends the canonical basis to canonical basis or $0$. By Theorem \ref{thm:Cat_affine}, $\jmath_p: \HallZ \rightarrow [\qW_{\eps^{2\bfs}}\modf]^*$ sends the canonical basis to the basis dual to the classes of simple modules. Collecting all these together, we conclude that $\jmath_\bfs: \Lambda (\omega_\bfs)_\Z\rightarrow  [\qW_{\eps^{2\bfs}} \modf]^*$ sends the canonical basis to the basis dual to the simple modules.
\end{proof}

\begin{rem}
\begin{enumerate}
\item 
One expects that the same categorification results as Theorem \ref{thm:cyc_cat} for $q$ generic can be obtained for degenerate cyclotomic webs over $\C$ \cite{SW25web}. 
\item
One also expects that Parts (1)-(2) of Theorem \ref{thm:cyc_cat} remain valid while Part (3) on canonical basis will fail for module categories of degenerate affine and cyclotomic webs over a field of positive characteristic $p$. 
\end{enumerate}
\end{rem}

\begin{rem} \label{rem:Cartan_pairing}
    Denote by $\qW_{\eps^{2\bfs}} \prmod$ the category of finite-dimensional projective $\qW_{\eps^{2\bfs}}$-modules. Via the Cartan pairing, one can replace $[\qW_{\eps^{2\bfs}}\modf]^*$ and the duals of the simple modules in Theorem \ref{thm:cyc_cat}  by $[\qW_{\eps^{2\bfs}}\prmod]$ and the classes of indecomposable projective modules.
\end{rem}

The category $\qW_{\eps^{2\bfs}} \modf$  has enough projectives and every projective module admits a cell module filtration since $\qW_{\eps^{2\bfs}}$ is a cellular algebra. 

\begin{corollary}
\label{cor:PIM}
    For $\bfla \in \mathcal C(\omega_\bfs)$, there exists a unique (up to isomorphism) indecomposable projective $\qW_{\eps^{2\bfs}}$-module $P^{\bfla}$, whose sections are all of the form $\Delta^\bfmu$ for $\bfmu\unlhd \bfla$ and $\Delta^\bfla$ appears exactly once, such that  
\begin{align}
    [P^{\bfla}] = \sum_{\bfmu \unlhd \bfla} d_{\bfmu,\bfla}(1) \,[\Delta^{\bfmu}].
\end{align}
\end{corollary}

\begin{proof}
    It follows by Theorem \ref{thm:cyc_cat}, in light of the $q$-decomposition matrix \eqref{eq:decomposition} and Remark~ \ref{rem:Cartan_pairing}(1). 
\end{proof}
Denote by $L^\bfla$ the simple head of $P^\bfla$. Then we have the surjective homomorphisms of $\qW_{\eps^{2\bfs}}$-modules 
\[
P^{\bfla} \twoheadrightarrow \Delta^{\bfla} \twoheadrightarrow L^{\bfla}.
\]
We have the following classification result. 
\begin{corollary} \label{cor:classifysimple}
    The $L^\bfla$, for $\bfla \in \mathcal C(\omega_\bfs)$, forms a complete list of pairwise non-isomorphic simple modules in $\qW_{\eps^{2\bfs}} \modf$. 
\end{corollary}

There is an idempotent functor which we refer to as a Schur functor, for each $n$, $1_{(1^n)}: \qW_{\eps^{2\bfs}}(n)\modf \rightarrow \HH_{\eps^{2\bfs}}(n)\modf$. Then  $1_\bfs: =\oplus_{n\in \N}1_{(1^n)}$ leads to a Schur functor (which is surjective) $1_\bfs: [ \qW_{\eps^{2\bfs}}\modf]
\rightarrow [\HH_{\eps^{2\bfs}}\modf]$
and hence an injective map 
$1_\bfs^*:[\HH_{\eps^{2\bfs}}\modf]^*\rightarrow [ \qW_{\eps^{2\bfs}}\modf]^*$. Moreover, for $i\in \I$, there are functors $i$-$\Res$ and $i$-$\Ind$ on $\HH_{\eps^{2\bfs}}\modf$ \cite{Ar96} defined in the same way as those for $ \qW_{\eps^{2\bfs}}\modf$. Ariki's categorification theorem \cite[Theorem 4.4]{Ar96} shows that these functors define an $\widehat{\mathfrak{sl}_p}$-module structure on
$[\HH_{\eps^{2\bfs}}\modf]^*$ such that $[\HH_{\eps^{2\bfs}}\modf]^*\cong V(\omega_\bfs)_\Z$, where $V(\omega_\bfs)$ denotes the integrable  $\widehat{\mathfrak{sl}}_p$-module of highest weight $\omega_\bfs$. 

\begin{proposition} 
\label{prop:SchurHecke}
We have a commutative diagram of $\U_\Z(\widehat{\mathfrak{sl}}_p)$-module homomorphisms:
\begin{equation}   \label{CD:HeckeWeb}
\begin{tikzcd}
V (\omega_\bfs)_\Z \ar[d,hook,"\iota_\bfs"]\ar[r,"\phi_\bfs"] &{[}\HH_{\eps^{2\bfs}}\modf {]}^*\ar[d,hook,"1_\bfs^*"] \\
\Lambda (\omega_\bfs)_\Z  \ar[r,"\jmath_\bfs"]& {[}\qW_{\eps^{2\bfs}} \modf{]}^* 
\end{tikzcd} 
\end{equation}
where $\phi_\bfs$ is Ariki's isomorphism \cite{Ar96}. Moreover, the embedding $\iota_\bfs$ preserves the canonical basis while $1_\bfs^*$ preserves the basis dual to the simple modules. 
\end{proposition}

\begin{proof}
Since $1_\bfs \qW_{\eps^{2\bfs}} 1_\bfs= \HH_{\eps^{2\bfs}}$ \cite{SSW25}, 
  by definition we have  
  $i\text{-Res}1_\bfs (M)= 1_\bfs i\text{-Res}(M)$ and $i\text{-Ind}1_\bfs (M)= 1_\bfs i\text{-Ind}(M)$, for any $M\in \qW_{\eps^{2\bfs}} \modf$.
  This shows that $1_\bfs^*$ is a homoporphism of $\widehat{\mathfrak{sl}_p}$-modules. 
  Then by comparing the actions of $i$-Res and $i$-Ind on the cell modules $ S^{\lambda}=1_\bfs \Delta^\lambda$ in \cite[Lemma 2.1]{Ar96} and that on $\Delta^\lambda $ in Proposition \ref{prop:IndRes}, we see that  the diagram is commutative. The rest follows from Theorem \ref{thm:cyc_cat}.
\end{proof}

\begin{rem}
Proposition \ref{prop:SchurHecke} shows  that the $v$-decomposition matrix for a cyclotomic $q$-W-Schur algebra contains the counterpart for a cyclotomic Hecke algebra as a submatrix. At level one, this was due to \cite{LT96, VV99}. 
\end{rem}

\begin{rem}
The canonical basis on $\Lambda(\omega_\bfs)$ and thus the $v$-decomposition matrices depend only on the parameters $\bfs$ modulo $p$. This differs from the observation in \cite{Yv06} on the canonical basis of Uglov's higher Fock spaces.
\end{rem}

\subsection{Some open problems}
\label{ssec:open}

There has been much further advances on interaction among Hecke algebras, quantum groups and categorification since the works of Lascoux-Leclerc-Thibon and Ariki. We refer to the survey of Kleshchev \cite{Kle10} for extensive references. The cyclotomic $q$-Schur algebras have also been widely studied. The affine and cyclotomic $q$-web categories sit in between these well-studied algebras as indicated by the diagram \eqref{CD:categories}. It is reasonable to ask for various W-Schur analogues of the topics in the other two contexts, and we list some of them.

\begin{enumerate}
    \item Construct a Jantzen filtration for the cell modules of cyclotomic $q$-W-Schur algebras and compute the Jantzen sum formula. 
    We expect that the Jantzen filtration is compatible with the $v$-decomposition matrix for cyclotomic $q$-W-Schur algebras. 
    \item Classify the blocks for cyclotomic $q$-W-Schur algebras. 
    \item Understand the relation between the crystal basis and modular branching rules for cyclotomic $q$-W-Schur algebras. The rich combinatorics of crystals often corresponds to different ways of parameterizing the simple modules. 
    \item Formulate a $\Z$-graded version.
    \item Realize an action of $\U_\Z^-(\widehat{\mathfrak{gl}}_p)$ on $\K^*$ geometrically so that 
    \eqref{eq:ijk} becomes a commutative diagram of $\U_\Z^-(\widehat{\mathfrak{gl}}_p)$-module isomorphisms. 
    \item  (This is a variant and an enhancement of (5).)  Construct algebraic structures on ${[}\qAW\modf_\eps{]}^*$ diagrammatically and on $\K^*$ geometrically so that 
    \eqref{eq:ijk} becomes a commutative diagram of algebra isomorphisms.
\end{enumerate}
\noindent {\bf Conflict of interest.}

On behalf of all authors, the corresponding author states that there is no conflict of interest.
The authors declare that the data supporting the findings of this study
are available within the paper.

\bibliographystyle{alpha}
\bibliography{qWeb}

@article {AJL,
    AUTHOR = {Ariki, Susumu and Jacon, Nicolas and Lecouvey, C\'edric},
     TITLE = {The modular branching rule for affine {H}ecke algebras of type
              {$A$}},
   JOURNAL = {Adv. Math.},
  FJOURNAL = {Advances in Mathematics},
    VOLUME = {228},
      YEAR = {2011},
    NUMBER = {1},
     PAGES = {481--526},
      ISSN = {0001-8708,1090-2082},
   MRCLASS = {17B37 (20C08)},
  MRNUMBER = {2822237},
MRREVIEWER = {Guiyu\ Yang},
       DOI = {10.1016/j.aim.2011.05.018},
       URL = {https://doi.org/10.1016/j.aim.2011.05.018},
}

@article {Ar96,
    AUTHOR = {Ariki, Susumu},
     TITLE = {On the decomposition numbers of the {H}ecke algebra of
              {$G(m,1,n)$}},
   JOURNAL = {J. Math. Kyoto Univ.},
  FJOURNAL = {Journal of Mathematics of Kyoto University},
    VOLUME = {36},
      YEAR = {1996},
    NUMBER = {4},
     PAGES = {789--808},
      ISSN = {0023-608X},
   MRCLASS = {20C20 (20G99)},
  MRNUMBER = {1443748},
MRREVIEWER = {Meinolf Geck},
       DOI = {10.1215/kjm/1250518452},
       URL = {https://doi.org/10.1215/kjm/1250518452},
}

@book {Ar02,
    AUTHOR = {Ariki, Susumu},
     TITLE = {Representations of quantum algebras and combinatorics of
              {Y}oung tableaux},
    SERIES = {University Lecture Series},
    VOLUME = {26},
      NOTE = {Translated from the 2000 Japanese edition and revised by the
              author},
 PUBLISHER = {American Mathematical Society, Providence, RI},
      YEAR = {2002},
     PAGES = {viii+158},
      ISBN = {0-8218-3232-8},
   MRCLASS = {17B37 (05E10 14M15 17B67 20C08)},
  MRNUMBER = {1911030},
MRREVIEWER = {Andrew\ Mathas},
       DOI = {10.1090/ulect/026},
       URL = {https://doi.org/10.1090/ulect/026},
}

@article {Bru25,
    AUTHOR = {Brundan, Jonathan},
     TITLE = {The {$q$}-{S}chur category and polynomial tilting modules for
              quantum {${\rm GL}_n$}},
   JOURNAL = {Pacific J. Math.},
  FJOURNAL = {Pacific Journal of Mathematics},
    VOLUME = {336},
      YEAR = {2025},
    NUMBER = {1-2},
     PAGES = {63--112},
      ISSN = {0030-8730},
   MRCLASS = {17B10 (17B37)},
  MRNUMBER = {4914986},
       DOI = {10.2140/pjm.2025.336.63},
       URL = {https://doi.org/10.2140/pjm.2025.336.63},
}

@book {CG97,
    AUTHOR = {Chriss, Neil and Ginzburg, Victor},
     TITLE = {Representation theory and complex geometry},
 PUBLISHER = {Birkh\"{a}user Boston, Inc., Boston, MA},
      YEAR = {1997},
     PAGES = {x+495},
      ISBN = {0-8176-3792-3},
   MRCLASS = {22E47 (14F99 19L47 20G05 22-02 58F05)},
  MRNUMBER = {1433132},
MRREVIEWER = {William M. McGovern},
}

@unpublished{DKM25,
    author = {Davidson, Nicholas and Kujawa, Jonathan and Muth, Robert},
    title = {Superalgebra deformations of web categories: {A}ffine and cyclotomic webs},
   JOURNAL = {},
      YEAR = {2025},
NOTE ={\arxiv{2511.21671}},}

@article {DJM98,
    AUTHOR = {Dipper, Richard and James, Gordon and Mathas, Andrew},
     TITLE = {Cyclotomic {$q$}-{S}chur algebras},
   JOURNAL = {Math. Z.},
  FJOURNAL = {Mathematische Zeitschrift},
    VOLUME = {229},
      YEAR = {1998},
    NUMBER = {3},
     PAGES = {385--416},
      ISSN = {0025-5874},
   MRCLASS = {20C99 (16G99 20G05)},
  MRNUMBER = {1658581},
MRREVIEWER = {Jie Du},
       DOI = {10.1007/PL00004665},
       URL = {https://doi.org/10.1007/PL00004665},
}

@article {DX17,
    AUTHOR = {Deng, Bangming and Xiao, Jie},
     TITLE = {Hall algebras of cyclic quivers and {$q$}-deformed {F}ock
              spaces},
   JOURNAL = {J. Algebra},
  FJOURNAL = {Journal of Algebra},
    VOLUME = {480},
      YEAR = {2017},
     PAGES = {168--208},
      ISSN = {0021-8693},
   MRCLASS = {17B37 (16E20 16G20 17B10)},
  MRNUMBER = {3633304},
MRREVIEWER = {Haicheng Zhang},
       DOI = {10.1016/j.jalgebra.2017.02.00x6},
       URL = {https://doi.org/10.1016/j.jalgebra.2017.02.006},
    NOTE ={\arxiv{1507.03064}},
    URL = {
https://arxiv.org/abs/1507.03064}
}

@article {FL19,
    AUTHOR = {Fan, Zhaobing and Li, Yiqiang},
     TITLE = {Affine flag varieties and quantum symmetric pairs, {II}.
              {M}ultiplication formula},
   JOURNAL = {J. Pure Appl. Algebra},
  FJOURNAL = {Journal of Pure and Applied Algebra},
    VOLUME = {223},
      YEAR = {2019},
    NUMBER = {10},
     PAGES = {4311--4347},
      ISSN = {0022-4049,1873-1376},
   MRCLASS = {17B37 (14M15 20G25)},
  MRNUMBER = {3958094},
MRREVIEWER = {Iwan\ Praton},
       DOI = {10.1016/j.jpaa.2019.01.011},
       URL = {https://doi.org/10.1016/j.jpaa.2019.01.011},
}

@incollection {GRV94,
    AUTHOR = {Ginzburg, Victor and Reshetikhin, Nicolai and Vasserot, Eric},
     TITLE = {Quantum groups and flag varieties},
 BOOKTITLE = {Mathematical aspects of conformal and topological field
              theories and quantum groups ({S}outh {H}adley, {MA}, 1992)},
    SERIES = {Contemp. Math.},
    VOLUME = {175},
     PAGES = {101--130},
 PUBLISHER = {Amer. Math. Soc., Providence, RI},
      YEAR = {1994},
   MRCLASS = {17B37 (14M15 16W30)},
  MRNUMBER = {1302015},
MRREVIEWER = {Jia-Chen Ye},
       DOI = {10.1090/conm/175/01840},
       URL = {https://doi.org/10.1090/conm/175/01840},
}

@article {GV93,
    AUTHOR = {Ginzburg, Victor and Vasserot, Eric},
     TITLE = {Langlands reciprocity for affine quantum groups of type
              {$A_n$}},
   JOURNAL = {Internat. Math. Res. Notices},
  FJOURNAL = {International Mathematics Research Notices},
      YEAR = {1993},
    NUMBER = {3},
     PAGES = {67--85},
      ISSN = {1073-7928},
   MRCLASS = {17B37 (17B67)},
  MRNUMBER = {1208827},
MRREVIEWER = {Jie Du},
       DOI = {10.1155/S1073792893000078},
       URL = {https://doi.org/10.1155/S1073792893000078},
}

@article {Gre99,
    AUTHOR = {Green, R. M.},
     TITLE = {The affine {$q$}-{S}chur algebra},
   JOURNAL = {J. Algebra},
  FJOURNAL = {Journal of Algebra},
    VOLUME = {215},
      YEAR = {1999},
    NUMBER = {2},
     PAGES = {379--411},
      ISSN = {0021-8693},
   MRCLASS = {20C20 (17B45)},
  MRNUMBER = {1686197},
MRREVIEWER = {Stuart Martin},
       DOI = {10.1006/jabr.1998.7753},
       URL = {https://doi.org/10.1006/jabr.1998.7753},
}

@article {JMMO91,
    AUTHOR = {Jimbo, Michio and Misra, Kailash C. and Miwa, Tetsuji and
              Okado, Masato},
     TITLE = {Combinatorics of representations of {$U_q(\widehat{{\germ
              s}{\germ l}}(n))$} at {$q=0$}},
   JOURNAL = {Comm. Math. Phys.},
  FJOURNAL = {Communications in Mathematical Physics},
    VOLUME = {136},
      YEAR = {1991},
    NUMBER = {3},
     PAGES = {543--566},
      ISSN = {0010-3616},
   MRCLASS = {17B37 (05E15 82B23)},
  MRNUMBER = {1099695},
       URL = {http://projecteuclid.org/euclid.cmp/1104202436},
}

@article {KL87,
    AUTHOR = {Kazhdan, David and Lusztig, George},
     TITLE = {Proof of the {D}eligne-{L}anglands conjecture for {H}ecke
              algebras},
   JOURNAL = {Invent. Math.},
  FJOURNAL = {Inventiones Mathematicae},
    VOLUME = {87},
      YEAR = {1987},
    NUMBER = {1},
     PAGES = {153--215},
      ISSN = {0020-9910,1432-1297},
   MRCLASS = {11S37 (18F25 19K33 22E50)},
  MRNUMBER = {862716},
MRREVIEWER = {Joe\ Repka},
       DOI = {10.1007/BF01389157},
       URL = {https://doi.org/10.1007/BF01389157},
}

@article {DF15,
    AUTHOR = {Du, Jie and Fu, Qiang},
     TITLE = {Quantum affine {$\germ {gl}_n$} via {H}ecke algebras},
   JOURNAL = {Adv. Math.},
  FJOURNAL = {Advances in Mathematics},
    VOLUME = {282},
      YEAR = {2015},
     PAGES = {23--46},
      ISSN = {0001-8708},
   MRCLASS = {17B37 (17B67 20G43)},
  MRNUMBER = {3374521},
MRREVIEWER = {H. H. Andersen},
       DOI = {10.1016/j.aim.2015.06.007},
       URL = {https://doi.org/10.1016/j.aim.2015.06.007},
}

@article {DM02,
    AUTHOR = {Dipper, Richard and Mathas, Andrew},
     TITLE = {Morita equivalences of {A}riki-{K}oike algebras},
   JOURNAL = {Math. Z.},
  FJOURNAL = {Mathematische Zeitschrift},
    VOLUME = {240},
      YEAR = {2002},
    NUMBER = {3},
     PAGES = {579--610},
      ISSN = {0025-5874,1432-1823},
   MRCLASS = {20C08},
  MRNUMBER = {1924022},
MRREVIEWER = {David\ Hemmer},
       DOI = {10.1007/s002090100371},
       URL = {https://doi.org/10.1007/s002090100371},
}

@article {GL96,
    AUTHOR = {Graham, J. J. and Lehrer, G. I.},
     TITLE = {Cellular algebras},
   JOURNAL = {Invent. Math.},
  FJOURNAL = {Inventiones Mathematicae},
    VOLUME = {123},
      YEAR = {1996},
    NUMBER = {1},
     PAGES = {1--34},
      ISSN = {0020-9910,1432-1297},
   MRCLASS = {20C99},
  MRNUMBER = {1376244},
       DOI = {10.1007/BF01232365},
       URL = {https://doi.org/10.1007/BF01232365},
}

@article {KMS95,
    AUTHOR = {Kashiwara, M. and Miwa, T. and Stern, E.},
     TITLE = {Decomposition of {$q$}-deformed {F}ock spaces},
   JOURNAL = {Selecta Math. (N.S.)},
  FJOURNAL = {Selecta Mathematica. New Series},
    VOLUME = {1},
      YEAR = {1995},
    NUMBER = {4},
     PAGES = {787--805},
      ISSN = {1022-1824},
   MRCLASS = {17B37 (81R50)},
  MRNUMBER = {1383585},
MRREVIEWER = {Anatol Nowicki},
       DOI = {10.1007/BF01587910},
       URL = {https://doi.org/10.1007/BF01587910},
}

@article {LXY26,
    AUTHOR = {Luo, Li and Xu, Zheming and Yang, Yang},
     TITLE = {Geometric construction of {S}chur algebras},
   JOURNAL = {J. Lond. Math. Soc. (2)},
  FJOURNAL = {Journal of the London Mathematical Society. Second Series},
    VOLUME = {113},
      YEAR = {2026},
    NUMBER = {5},
     PAGES = {Paper No. e70558},
      ISSN = {0024-6107,1469-7750},
   MRCLASS = {Prelim},
  MRNUMBER = {5068507},
       DOI = {10.1112/jlms.70558},
       URL = {https://doi.org/10.1112/jlms.70558},
}

@article {LLT96,
    AUTHOR = {Lascoux, Alain and Leclerc, Bernard and Thibon, Jean-Yves},
     TITLE = {Hecke algebras at roots of unity and crystal bases of quantum
              affine algebras},
   JOURNAL = {Comm. Math. Phys.},
  FJOURNAL = {Communications in Mathematical Physics},
    VOLUME = {181},
      YEAR = {1996},
    NUMBER = {1},
     PAGES = {205--263},
      ISSN = {0010-3616,1432-0916},
   MRCLASS = {17B37 (20C20)},
  MRNUMBER = {1410572},
MRREVIEWER = {Thomas\ M.\ Halverson},
       URL = {http://projecteuclid.org/euclid.cmp/1104287629},
}

@article {LT96,
    AUTHOR = {Leclerc, Bernard and Thibon, Jean-Yves},
     TITLE = {Canonical bases of {$q$}-deformed {F}ock spaces},
   JOURNAL = {Internat. Math. Res. Notices},
  FJOURNAL = {International Mathematics Research Notices},
      YEAR = {1996},
    NUMBER = {9},
     PAGES = {447--456},
      ISSN = {1073-7928},
   MRCLASS = {17B67 (05A30 05E10 17B10 17B37)},
  MRNUMBER = {1399410},
MRREVIEWER = {Yvette Kosmann-Schwarzbach},
       DOI = {10.1155/S1073792896000293},
       URL = {https://doi.org/10.1155/S1073792896000293},
}

@article {LTV99,
    AUTHOR = {Leclerc, Bernard and Thibon, Jean-Yves and Vasserot, Eric},
     TITLE = {Zelevinsky's involution at roots of unity},
   JOURNAL = {J. Reine Angew. Math.},
  FJOURNAL = {Journal f\"{u}r die Reine und Angewandte Mathematik. [Crelle's
              Journal]},
    VOLUME = {513},
      YEAR = {1999},
     PAGES = {33--51},
      ISSN = {0075-4102},
   MRCLASS = {20C08 (17B67 20C30 20G42 22E50)},
  MRNUMBER = {1713318},
MRREVIEWER = {Alexander Kleshchev},
       DOI = {10.1515/crll.1999.062},
       URL = {https://doi.org/10.1515/crll.1999.062},
}

@article {Lus90,
    AUTHOR = {Lusztig, G.},
     TITLE = {Canonical bases arising from quantized enveloping algebras},
   JOURNAL = {J. Amer. Math. Soc.},
  FJOURNAL = {Journal of the American Mathematical Society},
    VOLUME = {3},
      YEAR = {1990},
    NUMBER = {2},
     PAGES = {447--498},
      ISSN = {0894-0347},
   MRCLASS = {17B35 (16A64)},
  MRNUMBER = {1035415},
MRREVIEWER = {Ya. S. So\u{\i}bel\cprime man},
       DOI = {10.2307/1990961},
       URL = {https://doi.org/10.2307/1990961},
}

@article {Lus91,
    AUTHOR = {Lusztig, G.},
     TITLE = {Quivers, perverse sheaves, and quantized enveloping algebras},
   JOURNAL = {J. Amer. Math. Soc.},
  FJOURNAL = {Journal of the American Mathematical Society},
    VOLUME = {4},
      YEAR = {1991},
    NUMBER = {2},
     PAGES = {365--421},
      ISSN = {0894-0347},
   MRCLASS = {17B37 (17B67 20G05)},
  MRNUMBER = {1088333},
MRREVIEWER = {H. H. Andersen},
       DOI = {10.2307/2939279},
       URL = {https://doi.org/10.2307/2939279},
}

@incollection {Lus98,
    AUTHOR = {Lusztig, George},
     TITLE = {Canonical bases and {H}all algebras},
 BOOKTITLE = {Representation theories and algebraic geometry ({M}ontreal,
              {PQ}, 1997)},
    SERIES = {NATO Adv. Sci. Inst. Ser. C: Math. Phys. Sci.},
    VOLUME = {514},
     PAGES = {365--399},
 PUBLISHER = {Kluwer Acad. Publ., Dordrecht},
      YEAR = {1998},
   MRCLASS = {17B37 (16G20 17-02)},
  MRNUMBER = {1653038},
MRREVIEWER = {Kailash C. Misra},
}

@article {RSVV,
    AUTHOR = {Rouquier, Raphael and Shan, Peng and Varagnolo, Michela
              and Vasserot, Eric},
     TITLE = {Categorifications and cyclotomic rational double affine
              {H}ecke algebras},
   JOURNAL = {Invent. Math.},
  FJOURNAL = {Inventiones Mathematicae},
    VOLUME = {204},
      YEAR = {2016},
    NUMBER = {3},
     PAGES = {671--786},
      ISSN = {0020-9910,1432-1297},
   MRCLASS = {20C08 (18D10 20G43)},
  MRNUMBER = {3502064},
MRREVIEWER = {Nicolas\ Jacon},
       DOI = {10.1007/s00222-015-0623-7},
       URL = {https://doi.org/10.1007/s00222-015-0623-7},
}

@article {Sch00,
    AUTHOR = {Schiffmann, Olivier},
     TITLE = {The {H}all algebra of a cyclic quiver and canonical bases of
              {F}ock spaces},
   JOURNAL = {Internat. Math. Res. Notices},
  FJOURNAL = {International Mathematics Research Notices},
      YEAR = {2000},
    NUMBER = {8},
     PAGES = {413--440},
      ISSN = {1073-7928},
   MRCLASS = {16G20 (17B37 81R10)},
  MRNUMBER = {1753691},
MRREVIEWER = {Jin Yun Guo},
       DOI = {10.1155/S1073792800000234},
       URL = {https://doi.org/10.1155/S1073792800000234},
}

@article {SW25web,
    AUTHOR = {Song, Linliang and Wang, Weiqiang},
     TITLE = {Affine and cyclotomic webs},
   JOURNAL = {J. Lond. Math. Soc. (2)},
  FJOURNAL = {Journal of the London Mathematical Society. Second Series},
    VOLUME = {112},
      YEAR = {2025},
    NUMBER = {3},
     PAGES = {Paper No. e70278, 43},
      ISSN = {0024-6107},
   MRCLASS = {18M05 (20C08)},
  MRNUMBER = {4955711},
       DOI = {10.1112/jlms.70278},
       URL = {https://doi.org/10.1112/jlms.70278},
}

@unpublished{SW24Schur,
    author = {Song, Linliang and Wang, Weiqiang},
    title = {Affine and cyclotomic {S}chur categories},
   JOURNAL = {},
      YEAR = {2024},
    NOTE ={\arxiv{2407.10119}},
    URL = {
https://doi.org/ }
}

@unpublished{SSW25,
    author = {Shen, Yaolong and Song, Linliang and Wang, Weiqiang},
    title = {Affine {and } cyclotomic $q$-{S}chur categories via webs},
   JOURNAL = {},
      YEAR = {2025},
NOTE ={\arxiv{2504.10270}},
    URL = {
https://doi.org/abs/2504.10270}
}

@incollection {Ugl00,
    AUTHOR = {Uglov, Denis},
     TITLE = {Canonical bases of higher-level {$q$}-deformed {F}ock spaces
              and {K}azhdan-{L}usztig polynomials},
 BOOKTITLE = {Physical combinatorics ({K}yoto, 1999)},
    SERIES = {Progr. Math.},
    VOLUME = {191},
     PAGES = {249--299},
 PUBLISHER = {Birkh\"{a}user Boston, Boston, MA},
      YEAR = {2000},
   MRCLASS = {17B37 (17B67)},
  MRNUMBER = {1768086},
MRREVIEWER = {Nicol\'{a}s Andruskiewitsch},
NOTE ={\arxiv{math/9905196}},
    URL = {https://arxiv.org/abs/math/9905196}
}

@article {V98,
    AUTHOR = {Vasserot, Eric},
     TITLE = {Affine quantum groups and equivariant {$K$}-theory},
   JOURNAL = {Transform. Groups},
  FJOURNAL = {Transformation Groups},
    VOLUME = {3},
      YEAR = {1998},
    NUMBER = {3},
     PAGES = {269--299},
      ISSN = {1083-4362,1531-586X},
   MRCLASS = {19L47 (17B37)},
  MRNUMBER = {1640675},
MRREVIEWER = {Teimuraz\ Pirashvili},
       DOI = {10.1007/BF01236876},
       URL = {https://doi.org/10.1007/BF01236876},
}

@article {VV99,
    AUTHOR = {Varagnolo, Michela and Vasserot, Eric},
     TITLE = {On the decomposition matrices of the quantized {S}chur
              algebra},
   JOURNAL = {Duke Math. J.},
  FJOURNAL = {Duke Mathematical Journal},
    VOLUME = {100},
      YEAR = {1999},
    NUMBER = {2},
     PAGES = {267--297},
      ISSN = {0012-7094},
   MRCLASS = {17B37 (20C08)},
  MRNUMBER = {1722955},
MRREVIEWER = {Alexander N. Rudy\u{\i}},
       DOI = {10.1215/S0012-7094-99-10010-X},
       URL = {https://doi.org/10.1215/S0012-7094-99-10010-X},
}

@article {Wad14,
    AUTHOR = {Wada, Kentaro},
     TITLE = {Induction and restriction functors for cyclotomic
              {$q$}-{S}chur algebras},
   JOURNAL = {Osaka J. Math.},
  FJOURNAL = {Osaka Journal of Mathematics},
    VOLUME = {51},
      YEAR = {2014},
    NUMBER = {3},
     PAGES = {785--822},
      ISSN = {0030-6126},
   MRCLASS = {20G43 (20C08)},
  MRNUMBER = {3272617},
MRREVIEWER = {Guiyu\ Yang},
       URL = {http://projecteuclid.org/euclid.ojm/1414090803},
}

@article {X97,
    AUTHOR = {Xiao, Jie},
     TITLE = {Drinfeld double and {R}ingel-{G}reen theory of {H}all
              algebras},
   JOURNAL = {J. Algebra},
  FJOURNAL = {Journal of Algebra},
    VOLUME = {190},
      YEAR = {1997},
    NUMBER = {1},
     PAGES = {100--144},
      ISSN = {0021-8693},
   MRCLASS = {16G20 (16W30 17B37 81R50)},
  MRNUMBER = {1442148},
MRREVIEWER = {Steffen K\"{o}nig},
       DOI = {10.1006/jabr.1996.6887},
       URL = {https://doi.org/10.1006/jabr.1996.6887},
}

@article {Kas91,
    AUTHOR = {Kashiwara, M.},
     TITLE = {On crystal bases of the {$Q$}-analogue of universal enveloping
              algebras},
   JOURNAL = {Duke Math. J.},
  FJOURNAL = {Duke Mathematical Journal},
    VOLUME = {63},
      YEAR = {1991},
    NUMBER = {2},
     PAGES = {465--516},
      ISSN = {0012-7094,1547-7398},
   MRCLASS = {17B37 (17B10 17B67)},
  MRNUMBER = {1115118},
MRREVIEWER = {Kailash\ C.\ Misra},
       DOI = {10.1215/S0012-7094-91-06321-0},
       URL = {https://doi.org/10.1215/S0012-7094-91-06321-0},
}

@article {Kle10,
    AUTHOR = {Kleshchev, Alexander},
     TITLE = {Representation theory of symmetric groups and related {H}ecke
              algebras},
   JOURNAL = {Bull. Amer. Math. Soc. (N.S.)},
  FJOURNAL = {American Mathematical Society. Bulletin. New Series},
    VOLUME = {47},
      YEAR = {2010},
    NUMBER = {3},
     PAGES = {419--481},
      ISSN = {0273-0979,1088-9485},
   MRCLASS = {20C30 (17B37 17B67 20C08)},
  MRNUMBER = {2651085},
MRREVIEWER = {Andrew\ Mathas},
       DOI = {10.1090/S0273-0979-09-01277-4},
       URL = {https://doi.org/10.1090/S0273-0979-09-01277-4},
}

@article {Yv06,
    AUTHOR = {Yvonne, Xavier},
     TITLE = {A conjecture for {$q$}-decomposition matrices of cyclotomic
              {$v$}-{S}chur algebras},
   JOURNAL = {J. Algebra},
  FJOURNAL = {Journal of Algebra},
    VOLUME = {304},
      YEAR = {2006},
    NUMBER = {1},
     PAGES = {419--456},
      ISSN = {0021-8693},
   MRCLASS = {16S99 (20C08)},
  MRNUMBER = {2256400},
       DOI = {10.1016/j.jalgebra.2006.03.048},
       URL = {https://doi.org/10.1016/j.jalgebra.2006.03.048},
}

\end{document}